\documentclass[10pt,a4paper]{article}
\pdfoutput=1

\usepackage{lmodern}
\usepackage[utf8]{inputenc}
\usepackage[T1]{fontenc}
\usepackage{microtype}

\usepackage{ifthen}
\usepackage{tikz}
\usetikzlibrary{arrows}
\usepackage[margin=10pt,font=small,labelfont=bf,labelsep=period]{caption}
\usepackage{longtable}
\usepackage{mathscinet}
\usepackage{amsmath}
\usepackage{amsthm}
\usepackage{amssymb}
\numberwithin{equation}{section}

\usepackage{aliascnt}
\usepackage[colorlinks,linkcolor=blue,citecolor=blue,pdfauthor={Hans Lundmark, Jacek Szmigielski},pdftitle={An inverse spectral problem related to the Geng--Xue two-component peakon equation}]{hyperref}

\theoremstyle{plain}

\newtheorem{theorem}{Theorem}[section]

\newaliascnt{lemma}{theorem}
\newtheorem{lemma}[lemma]{Lemma}
\aliascntresetthe{lemma}

\newaliascnt{corollary}{theorem}
\newtheorem{corollary}[corollary]{Corollary}
\aliascntresetthe{corollary}

\newaliascnt{proposition}{theorem}
\newtheorem{proposition}[proposition]{Proposition}
\aliascntresetthe{proposition}

\theoremstyle{definition}

\newaliascnt{definition}{theorem}
\newtheorem{definition}[definition]{Definition}
\aliascntresetthe{definition}

\newaliascnt{example}{theorem}
\newtheorem{example}[example]{Example}
\aliascntresetthe{example}

\newaliascnt{remark}{theorem}
\newtheorem{remark}[remark]{Remark}
\aliascntresetthe{remark}

\newaliascnt{assumption}{theorem}

\aliascntresetthe{assumption}

\newcommand{\R}{\mathbf{R}}
\newcommand{\C}{\mathbf{C}}
\newcommand{\Z}{\mathbf{Z}}

\newcommand{\abs}[1]{\left\lvert #1 \right\rvert}

\newcommand{\jump}[1]{\bigl[ #1 \bigr]}

\newcommand{\bilinearRthree}[2]{\left\langle #1, #2 \right\rangle}

\makeatletter
\DeclareRobustCommand{\jumpmatrix}{\genfrac[]\z@{}}
\makeatother

\newcommand{\order}[1]{\mathcal{O}\left( #1 \right)}

\newenvironment{smallpmatrix}{\left(\begin{smallmatrix}}{\end{smallmatrix}\right)}

\newcommand{\twin}{\widetilde}
\newcommand{\coeffmatrix}{\mathcal{A}}
\newcommand{\J}{J}

\renewcommand{\epsilon}{\varepsilon}

\newcommand{\da}{d\alpha(x)}
\newcommand{\db}{d\beta(y)}
\newcommand{\dA}[1]{d\alpha^{#1}(x)}
\newcommand{\dB}[1]{d\beta^{#1}(y)}

\newcommand{\heineintegral}{\mathcal{J}}
\newcommand{\weirddeterminant}{\mathcal{K}}
\newcommand{\starheineintegral}{(\heineintegral^*)}

\DeclareMathOperator{\sgn}{sgn}

\DeclareMathOperator*{\res}{res}
\DeclareMathOperator{\adj}{adj}
\DeclareMathOperator{\diag}{diag}

\DeclareMathOperator{\id}{id}
\DeclareMathOperator{\tr}{tr}

\newcommand{\lel}{\preccurlyeq}
\newcommand{\lle}{\curlyeqprec}

\begin{document}

\title{An inverse spectral problem related to the Geng--Xue two-component peakon equation}
\author{%
  Hans Lundmark\thanks{Department of Mathematics, Linköping University, SE-581\,83 Linköping, Sweden; hans.lundmark@liu.se}
  \and
  Jacek Szmigielski\thanks{Department of Mathematics and Statistics, University of Saskatchewan, 106 Wiggins Road, Saskatoon, Saskatchewan, S7N\,5E6, Canada; szmigiel@math.usask.ca}
}

\date{March 20, 2013}

\maketitle

\begin{abstract}
  We solve a spectral and an inverse spectral problem
  arising in the computation of peakon solutions to the
  two-component PDE derived by Geng and Xue as a generalization of the
  Novikov and Degasperis--Procesi equations.
  Like the spectral problems for those equations, this one is of
  a `discrete cubic string' type -- a nonselfadjoint generalization of
  a classical inhomogeneous string -- but presents some interesting novel
  features: there are two Lax pairs, both of which contribute to the
  correct complete spectral data, and the solution to the inverse
  problem can be expressed using quantities related to Cauchy
  biorthogonal polynomials with two different spectral measures. The
  latter extends the range of previous applications of Cauchy biorthogonal
  polynomials to peakons, which featured either two identical, or
  two closely related, measures. The method used to solve the spectral
  problem hinges on the hidden presence of oscillatory kernels of
  Gantmacher--Krein type implying that the spectrum of the boundary value problem
  is positive and simple. The inverse spectral problem is solved by a
  method which generalizes, to a nonselfadjoint case, M.~G.~Krein's
  solution of the inverse problem for the Stieltjes string.
\end{abstract}

\section{Introduction}
\label{sec:intro}

In this paper, we solve an inverse spectral problem which appears
in the context of computing explicit solutions to a two-component
integrable PDE in $1+1$ dimensions found by Geng and Xue \cite{geng-xue:cubic-nonlinearity}.
Denoting the two unknown functions by $u(x,t)$ and $v(x,t)$,
and introducing the auxiliary quantities
\begin{equation}
  \label{eq:mn}
  m = u - u_{xx}, \qquad n = v - v_{xx},
\end{equation}
we can write the Geng--Xue equation as
\begin{equation}
  \label{eq:GX}
  \begin{gathered}
    m_t + (m_xu + 3mu_x)v = 0, \\ n_t + (n_xv + 3nv_x)u = 0.
  \end{gathered}
\end{equation}
(Subscripts denote partial derivatives, as usual.)
This system arises as the compatibility condition of a Lax pair
with spectral parameter~$z$,
\begin{subequations}
  \label{eq:laxI}
  \begin{equation}
    \label{eq:laxI-x}
    \frac{\partial}{\partial x}
    \begin{pmatrix} \psi_1 \\ \psi_2 \\ \psi_3 \end{pmatrix} =
    \begin{pmatrix}
      0 & zn & 1 \\
      0 & 0 & zm \\
      1 & 0 & 0
    \end{pmatrix}
    \begin{pmatrix} \psi_1 \\ \psi_2 \\ \psi_3 \end{pmatrix},
  \end{equation}
  \begin{equation}
    \label{eq:laxI-t}
    \frac{\partial}{\partial t}
    \begin{pmatrix} \psi_1 \\ \psi_2 \\ \psi_3 \end{pmatrix} =
    \begin{pmatrix}
      -v_xu & v_x z^{-1}-vunz & v_xu_x \\
      u z^{-1} & v_xu-vu_x-z^{-2} & -u_x z^{-1}-vumz \\
      -vu & v z^{-1} & vu_x
    \end{pmatrix}
    \begin{pmatrix} \psi_1 \\ \psi_2 \\ \psi_3 \end{pmatrix},
  \end{equation}
\end{subequations}
but also (because of the symmetry in~\eqref{eq:GX}) as the compatibility condition of a different
Lax pair obtained by interchanging $u$ and~$v$,
\begin{subequations}
  \label{eq:laxII}
  \begin{equation}
    \label{eq:laxII-x}
    \frac{\partial}{\partial x}
    \begin{pmatrix} \psi_1 \\ \psi_2 \\ \psi_3 \end{pmatrix} =
    \begin{pmatrix}
      0 & zm & 1 \\
      0 & 0 & zn \\
      1 & 0 & 0
    \end{pmatrix}
    \begin{pmatrix} \psi_1 \\ \psi_2 \\ \psi_3 \end{pmatrix},
  \end{equation}
  \begin{equation}
    \label{eq:laxII-t}
    \frac{\partial}{\partial t}
    \begin{pmatrix} \psi_1 \\ \psi_2 \\ \psi_3 \end{pmatrix} =
    \begin{pmatrix}
      -u_xv & u_x z^{-1}-uvmz & u_xv_x \\
      v z^{-1} & u_xv-uv_x-z^{-2} & -v_x z^{-1}-uvnz \\
      -uv & u z^{-1} & uv_x
    \end{pmatrix}
    \begin{pmatrix} \psi_1 \\ \psi_2 \\ \psi_3 \end{pmatrix}.
  \end{equation}
\end{subequations}
The subject of our paper is the inverse problem of recovering $m$
and~$n$ from spectral data obtained by imposing suitable boundary
conditions on equations \eqref{eq:laxI-x} and \eqref{eq:laxII-x}, in
the case when $m$ and $n$ are both discrete measures (finite linear
combinations of Dirac deltas) with disjoint supports. To explain
why this is of interest, we will give a short historical background.

When $u=v$ (and consequently also $m=n$), the Lax pairs above reduce to the Lax pair
found by Hone and Wang for V.~Novikov's integrable PDE
\cite{novikov:generalizations-of-CH,hone-wang:cubic-nonlinearity}
\begin{equation}
  \label{eq:Novikov}
  m_t + m_x u^2 + 3m u u_x = 0,
\end{equation}
and it was by generalizing that Lax pair to \eqref{eq:laxI}
that Geng and Xue came up with their new integrable PDE \eqref{eq:GX}.
Novikov's equation, in turn, was found as a cubically nonlinear counterpart
to some previously known integrable PDEs with quadratic nonlinearities,
namely the Camassa--Holm equation \cite{camassa-holm}
\begin{equation}
  \label{eq:CH}
  m_t + m_x u + 2m u_x = 0
\end{equation}
and the Degasperis--Procesi equation \cite{degasperis-procesi,degasperis-holm-hone}
\begin{equation}
  \label{eq:DP}
  m_t + m_x u + 3m u_x = 0.
\end{equation}
The equations \eqref{eq:CH} and \eqref{eq:DP} have been much studied in the literature,
and the references are far too numerous to survey here.
Novikov's equation \eqref{eq:Novikov} is also beginning to attract attention;
see 
\cite{himonas-holliman:novikov-cauchy-problem,
hone-lundmark-szmigielski:novikov,
jiang-ni:novikov-blowup,
lai-li-wu:novikov-global-solutions,
mi-mu:modified-novikov-cauchy-problem,
ni-zhou:novikov,
yan-li-zhang:novikov-cauchy-problem}.
What these equations have in common is that they admit weak
solutions called \emph{peakons} (peaked solitons), taking the form
\begin{equation}
  \label{eq:standard-peakons}
  u(x,t) = \sum_{k=1}^N m_k(t) \, e^{-\abs{x - x_k(t)}},
\end{equation}
where the functions $x_k(t)$ and $m_k(t)$ satisfy an integrable system of
$2N$ ODEs, whose general solution can be written down explicitly in
terms of elementary functions with the help of inverse spectral techniques.
In the Camassa--Holm case, this involves very classical mathematics
surrounding the inverse spectral theory of the vibrating string with mass density~$g(y)$,
whose eigenmodes are determined by the Dirichlet problem
\begin{equation}
  \label{eq:ordinarystring}
  \begin{gathered}
    -\phi''(y) = z \, g(y) \, \phi(y) \quad \text{for $-1 < y < 1$}, \\
    \phi(-1) = 0, \qquad \phi(1) = 0.
  \end{gathered}
\end{equation}
In particular, one considers in this context the \emph{discrete string}
consisting of point masses connected by weightless thread,
so that $g$ is not a function but a linear combination of Dirac delta
distributions.
Then the solution to the inverse spectral problem can be expressed in terms
of orthogonal polynomials and Stieltjes continued fractions
\cite{beals-sattinger-szmigielski:stieltjes,beals-sattinger-szmigielski:moment,beals-sattinger-szmigielski:string-density,moser:three-integrable}.
The reason for the appearance of Dirac deltas here is that
when $u$ has the form \eqref{eq:standard-peakons},
the first derivative $u_x$ has a jump of size $-2m_k$ at each point $x=x_k$,
and this gives deltas in $u_{xx}$
when derivatives are taken in the sense of distributions.
In each interval between these points, $u$ is a linear combination of $e^x$ and $e^{-x}$,
so $u_{xx}=u$ there; thus
$m=u-u_{xx} = 2 \sum_{k=1}^N m_k \, \delta_{x_k}$
is a purely discrete distribution (or a discrete measure if one prefers).
The measure~$g(y)$ in \eqref{eq:ordinarystring} is related to the measure~$m(x)$
through a so-called Liouville transformation, and $g$ will be discrete when $m$
is discrete.

In the case of the Degasperis--Procesi and Novikov equations (and also for
the Geng--Xue equation, as we shall see), the corresponding role is
instead played by variants of a third-order nonselfadjoint spectral problem
called the \emph{cubic string}
\cite{lundmark-szmigielski:DPshort,lundmark-szmigielski:DPlong,kohlenberg-lundmark-szmigielski,lundmark-szmigielski:forced-burgers,hone-lundmark-szmigielski:novikov,bertola-gekhtman-szmigielski:cubicstring};
in its basic form it reads
\begin{equation}
  \label{eq:cubicstring}
  \begin{gathered}
    -\phi'''(y) = z \, g(y) \, \phi(y) \quad \text{for $-1 < y < 1$}, \\
    \phi(-1) = \phi'(-1) = 0, \qquad \phi(1) = 0.
  \end{gathered}
\end{equation}
The study of the discrete cubic string has prompted the development
of a theory of \emph{Cauchy biorthogonal polynomials} by
Bertola, Gekhtman and Szmigielski
\cite{bertola-gekhtman-szmigielski:cauchy,bertola-gekhtman-szmigielski:cubicstring,bertola-gekhtman-szmigielski:twomatrix,bertola-gekhtman-szmigielski:meijerG};
see \autoref{sec:biorth}.
In previous applications to peakon equations, the two measures
$\alpha$ and $\beta$ in the general setup of this theory have
coincided ($\alpha = \beta$), but in this paper we will actually see
two different spectral measures $\alpha$ and $\beta$ entering the picture
in a very natural way.

Like the above-mentioned PDEs,
the Geng--Xue equation also admits peakon solutions, but now with two components,
\begin{equation}
  \label{eq:GXpeakons}
  \begin{split}
    u(x,t) &= \sum_{k=1}^N m_k(t) \, e^{-\abs{x - x_k(t)}}, \\
    v(x,t) &= \sum_{k=1}^N n_k(t) \, e^{-\abs{x - x_k(t)}},
  \end{split}
\end{equation}
where, for each~$k$, at most one of $m_k$ and~$n_k$ is nonzero
(i.e., $m_k n_k = 0$ for all~$k$).
In this case, $m$ and~$n$ will be discrete measures with disjoint support:
\begin{equation}
  \label{eq:mn-discrete}
  m=u-u_{xx} = 2 \sum_{k=1}^N  m_k \, \delta_{x_k},
  \qquad
  n=v-v_{xx} = 2 \sum_{k=1}^N  n_k \, \delta_{x_k}.
\end{equation}
This ansatz satisfies the PDE \eqref{eq:GX} if and only if
the functions $x_k(t)$, $m_k(t)$ and $n_k(t)$ satisfy the following system of ODEs:
\begin{equation}
  \label{eq:GX-peakon-ode}
  \begin{split}
    \dot x_k &= u(x_k) \, v(x_k)
    ,\\
    \dot m_k &= m_k \bigl(u(x_k) \, v_x(x_k) - 2 u_x(x_k) v(x_k) \bigr)
    ,\\
    \dot n_k &= n_k \bigl(u_x(x_k) \, v(x_k) - 2 u(x_k) v_x(x_k) \bigr)
    ,
  \end{split}
\end{equation}
for $k = 1,2,\dots,N$.
(Here we use the shorthand notation
\begin{equation*}
  u(x_k) = \sum_{i=1}^N m_i \, e^{-\abs{x_k - x_i}}
\end{equation*}
and
\begin{equation*}
  u_x(x_k) = \sum_{i=1}^N m_i \, e^{-\abs{x_k - x_i}} \, \sgn(x_k - x_i)
  .
\end{equation*}
If $u(x) = \sum m_i \, e^{-\abs{x-x_i}}$, then 
the derivative $u_x$ is undefined at the points $x_k$
where $m_k \neq 0$,
but here $\sgn 0 = 0$ by definition,
so $u_x(x_k)$ really denotes the average of the one-sided (left and right) derivatives
at those points. Note that the conditions $m_k=0$ and $m_k \neq 0$ both are preserved by the ODEs.
Similar remarks apply to $v$, of course.)

Knowing the solution of the inverse spectral problem
for \eqref{eq:laxI-x}+\eqref{eq:laxII-x} in this discrete case
makes it possible to explicitly determine the solutions to the
peakon ODEs \eqref{eq:GX-peakon-ode}.
Details about these peakon solutions and their dynamics
will be published in a separate paper; here we will focus on the
approximation-theoretical aspects of the inverse spectral problem.
(But see \autoref{rem:peakon-solutions}.)

We will only deal with the special case where the discrete measures
are \emph{interlacing}, meaning that there are $N=2K$ sites
\begin{equation*}
  x_1 <  x_2 < \dots < x_{2K},
\end{equation*}
with the measure~$m$ supported on the odd-numbered sites~$x_{2a-1}$,
and the measure~$n$ supported on the even-numbered sites~$x_{2a}$;
see \autoref{fig:m-n-discrete} and \autoref{rem:non-interlacing}.
The general formulas for recovering the positions $x_k$ and the weights
$m_{2a-1}$ and $n_{2a}$ are given in \autoref{cor:peakon-solution-formulas};
they are written out more explicitly for illustration in
\autoref{ex:solution-K2} (the case $K=2$) and \autoref{ex:solution-K3} (the case $K=3$).
The case $K=1$ is somewhat degenerate, and is treated separately in \autoref{sec:K1}.

\autoref{app:notation} contains an index of the notation used in this article.

\section{Forward spectral problem}
\label{sec:forward}

\subsection{Transformation to a finite interval}
\label{sec:toyota-tranformation}

Let us start by giving a precise definition of the spectral problem
to be studied.
The time dependence in the two Lax pairs for the Geng--Xue equation
will be of no interest to us in this paper, so we consider $t$ as fixed
and omit it in the notation.
The equations which govern the $x$ dependence in the two Lax pairs are
\eqref{eq:laxI-x} and \eqref{eq:laxII-x}, respectively.
Consider the first of these:
\begin{equation*}
  \frac{\partial}{\partial x}
  \begin{pmatrix} \psi_1 \\ \psi_2 \\ \psi_3 \end{pmatrix} =
  \begin{pmatrix}
    0 & zn(x) & 1 \\
    0 & 0 & zm(x) \\
    1 & 0 & 0
  \end{pmatrix}
  \begin{pmatrix} \psi_1 \\ \psi_2 \\ \psi_3 \end{pmatrix},
  \qquad
  \text{for $x \in \R$},
  \tag{\ref{eq:laxI-x}}
\end{equation*}
where $m(x)$ and~$n(x)$ are given.
Our main interest lies in the discrete case,
when $m$ and~$n$ are actually not functions but discrete measures
as in \eqref{eq:mn-discrete}, but we will not
specialize to that case until \autoref{sec:discrete-on-the-interval}.

There is a useful change of variables, similar to the one used for
Novikov's equation \cite{hone-lundmark-szmigielski:novikov}, which produces a slightly simpler
differential equation on a finite interval:
\begin{equation}
  \label{eq:liouville-trf}
  \begin{split}
    y &= \tanh x,\\
    \phi_1(y) &= \psi_1(x) \cosh x - \psi_3(x) \sinh x,\\
    \phi_2(y) &= z \, \psi_2(x),\\
    \phi_3(y) &= z^2 \, \psi_3(x) / \cosh x,\\
    g(y) &= m(x) \, \cosh^3 x,\\
    h(y) &= n(x) \, \cosh^3 x,\\
    \lambda &= -z^2.
  \end{split}
\end{equation}
Under this transformation (with $z\neq 0$),
equation \eqref{eq:laxI-x} is equivalent to
\begin{subequations}
  \label{eq:spectral-problem-y}
  \begin{equation}
    \label{eq:dual-cubic-I}
    \frac{\partial}{\partial y}
    \begin{pmatrix} \phi_1 \\ \phi_2 \\ \phi_3 \end{pmatrix}
    =
    \begin{pmatrix}
      0 & h(y) & 0 \\
      0 & 0 & g(y) \\
      -\lambda & 0 & 0
    \end{pmatrix}
    \begin{pmatrix} \phi_1 \\ \phi_2 \\ \phi_3 \end{pmatrix},
    \qquad
    \text{for $-1<y<1$}.
  \end{equation}
  (Notice that the $1$ in the upper right corner of the matrix has been removed
  by the transformation.
  When $h=g$, equation \eqref{eq:dual-cubic-I} reduces to the \emph{dual cubic string}
  studied in~\cite{hone-lundmark-szmigielski:novikov}.)
  In order to define a spectrum we impose the following boundary conditions
  on the differential equation~\eqref{eq:dual-cubic-I}:
  \begin{equation}
    \label{eq:boundary-conditions-y}
    \phi_2(-1) = \phi_3(-1) = 0,
    \qquad
    \phi_3(1) = 0.
  \end{equation}
  By the \emph{eigenvalues} of the problem~\eqref{eq:spectral-problem-y}
  we then of course mean those
  values of~$\lambda$ for which \eqref{eq:dual-cubic-I} has nontrivial
  solutions satisfying~\eqref{eq:boundary-conditions-y}.
\end{subequations}

The same transformation \eqref{eq:liouville-trf} applied to the twin
Lax equation \eqref{eq:laxII-x} leads to the same equation except that
$g$ and~$h$ are interchanged. The spectrum of this twin equation will
in general be different. To be explicit, the second spectrum is
defined by the differential equation
\begin{subequations}
  \label{eq:spectral-problem-y-twin}
  \begin{equation}
    \label{eq:dual-cubic-II}
    \frac{\partial}{\partial y}
    \begin{pmatrix} \phi_1 \\ \phi_2 \\ \phi_3 \end{pmatrix}
    =
    \begin{pmatrix}
      0 & g(y) & 0 \\
      0 & 0 & h(y) \\
      -\lambda & 0 & 0
    \end{pmatrix}
    \begin{pmatrix} \phi_1 \\ \phi_2 \\ \phi_3 \end{pmatrix},
    \qquad
    \text{for $-1<y<1$},
  \end{equation}
  again with boundary conditions
  \begin{equation}
    \label{eq:boundary-conditions-y-twin}
    \phi_2(-1) = \phi_3(-1) = 0,
    \qquad
    \phi_3(1) = 0.
  \end{equation}
\end{subequations}

\begin{remark}
  \label{rem:transformation}
  Via the transformation \eqref{eq:liouville-trf}, every concept
  pertaining to the original Lax equations \eqref{eq:laxI-x}
  and \eqref{eq:laxII-x} will
  have a counterpart in terms of the transformed equations
  \eqref{eq:dual-cubic-I} and \eqref{eq:dual-cubic-II}, and vice versa.
  In the main text, we will work
  with \eqref{eq:dual-cubic-I} and \eqref{eq:dual-cubic-II}
  on the finite interval.
  However, a few things are more conveniently dealt with directly
  in terms of the original equations
  \eqref{eq:laxI-x} and \eqref{eq:laxII-x} on the real line;
  these are treated in \autoref{app:real-line}.
  More specifically, we prove there that
  the spectra defined above are real and simple, and we also obtain expressions for
  certain quantities that will be constants of motion for the Geng--Xue
  peakon dynamics.
\end{remark}

\begin{remark}
  \label{rem:lambda-vs-z}
  Transforming the boundary conditions \eqref{eq:boundary-conditions-y}
  back to the real line via \eqref{eq:liouville-trf} yields
  \begin{equation}
    \label{eq:boundary-conditions-x}
    \lim_{x \to -\infty} \psi_2(x) =
    \lim_{x \to -\infty} e^x \psi_3(x) = 0,
    \qquad
    \lim_{x \to +\infty} e^{-x} \psi_3(x) = 0.
  \end{equation}
  Each eigenvalue $\lambda \neq 0$ of
  \eqref{eq:spectral-problem-y} corresponds
  to a pair of eigenvalues $z=\pm\sqrt{-\lambda}$ of
  \eqref{eq:laxI-x}+\eqref{eq:boundary-conditions-x}.
  As an exceptional case, $\lambda=0$ is an eigenvalue of
  \eqref{eq:spectral-problem-y},
  but $z=0$ is not an eigenvalue of
  \eqref{eq:laxI-x}+\eqref{eq:boundary-conditions-x};
  this is an artifact caused by the transformation \eqref{eq:liouville-trf}
  being singular for $z=0$.
  When talking about eigenvalues below, we will refer to~$\lambda$
  rather than~$z$.

  In \autoref{sec:Weyl-functions} below we will also encounter the
  condition $\phi_1(-1) = 1$; this translates into
  \begin{equation}
    \label{eq:initial-condition-x}
    \lim_{x \to -\infty} e^{-x} \bigl( \psi_1(x;z) - \psi_3(x;z) \bigr) = 2
    .
  \end{equation}
\end{remark}

\subsection{Transition matrices}
\label{sec:transition}

Let
\begin{equation}
  \label{eq:coeffmatrices}
  \coeffmatrix(y;\lambda) =
  \begin{pmatrix}
    0 & h(y) & 0 \\
    0 & 0 & g(y) \\
    -\lambda & 0 & 0
  \end{pmatrix},
  \quad
  \twin \coeffmatrix(y;\lambda) =
  \begin{pmatrix}
    0 & g(y) & 0 \\
    0 & 0 & h(y) \\
    -\lambda & 0 & 0
  \end{pmatrix}
\end{equation}
denote the coefficient matrices appearing in the spectral problems
\eqref{eq:spectral-problem-y} and \eqref{eq:spectral-problem-y-twin},
respectively.
To improve readability,
we will often omit the dependence on~$y$ in the notation,
and write the differential equations simply as
\begin{equation}
  \label{eq:dual-cubic-abbreviated}
  \frac{\partial \Phi}{\partial y} = \coeffmatrix(\lambda) \Phi,
  \qquad
  \frac{\partial \Phi}{\partial y} = \twin \coeffmatrix(\lambda) \Phi,
\end{equation}
respectively, where $\Phi=(\phi_1,\phi_2,\phi_3)^T$.
Plenty of information about this pair of equations
can be deduced from the following modest observation:
\begin{lemma}
  \label{lem:A-conjugation}
  The matrices $\coeffmatrix$ and~$\twin\coeffmatrix$ satisfy
  \begin{equation}
    \label{eq:A-conjugation}
    \twin \coeffmatrix(\lambda) = -\J \coeffmatrix(-\lambda)^T \J
    , \quad\text{where}\quad
    \J =
    \begin{pmatrix}
      0 & 0 & 1 \\ 0 & -1 & 0 \\ 1 & 0 & 0
    \end{pmatrix}
    = \J^T = \J^{-1}
    .
  \end{equation}
\end{lemma}

\begin{proof}
  A one-line calculation. 
\end{proof}

\begin{definition}[Involution $\sigma$]
  \label{def:involution-sigma}
  Let $\sigma$ denote the following operation on the (loop) group
  of invertible complex $3 \times 3$ matrices $X(\lambda)$ depending
  on the parameter $\lambda \in \C$:
  \begin{equation}
    \label{eq:involution-sigma}
    X(\lambda)^\sigma = \J X(-\lambda)^{-T} \J.
  \end{equation}
  (We use the customary abbreviation $X^{-T} = (X^T)^{-1} = (X^{-1})^T$.)
\end{definition}

\begin{remark}
  It is easily verified that $\sigma$ is a group homomorphism and an involution:
  \begin{equation*}
    \bigl( X(\lambda) Y(\lambda) \bigr)^\sigma
    = X(\lambda)^\sigma Y(\lambda)^\sigma,
    \qquad
    \bigl( X(\lambda)^\sigma \bigr)^\sigma = X(\lambda).
  \end{equation*}
\end{remark}

\begin{definition}[Fundamental matrices and transition matrices]
  Let $U(y;\lambda)$ be the fundamental matrix of \eqref{eq:dual-cubic-I}
  and $\twin U(y;\lambda)$ its counterpart for \eqref{eq:dual-cubic-II};
  i.e., they are the unique solutions of the matrix ODEs
  \begin{equation}
    \label{eq:fundsol}
    \frac{\partial U}{\partial y} = \coeffmatrix(y;\lambda) \, U,
    \qquad
    U(-1;\lambda) = I,
  \end{equation}
  and
  \begin{equation}
    \label{eq:fundsol-twin}
    \frac{\partial \twin U}{\partial y} = \twin \coeffmatrix(y;\lambda) \, \twin U,
    \qquad
    \twin U(-1;\lambda) = I,
  \end{equation}
  respectively, where $I$ is the $3\times 3$ identity matrix.
  The fundamental matrices evaluated at the right endpoint $y=1$
  will be called the \emph{transition matrices} and denoted by
  \begin{equation}
    \label{eq:transition-matrices}
    S(\lambda) = U(1;\lambda),
    \qquad
    \twin S(\lambda) = \twin U(1;\lambda).
  \end{equation}
\end{definition}

\begin{remark}
  The fundamental matrix contains the solution of any initial value problem:
  $\Phi(y) = U(y;\lambda) \Phi(-1)$ is the unique solution to the ODE
  $d\Phi/dy = \coeffmatrix(\lambda) \Phi$
  satisfying given initial data $\Phi(-1)$ at the left endpoint $y=-1$.
  In particular, the value of the solution at the right endpoint $y=1$
  is $\Phi(1) = S(\lambda) \Phi(-1)$.
\end{remark}

\begin{theorem}
  \label{thm:fundmatrix-relation}
  For all $y\in[-1,1]$,
  \begin{equation}
    \label{eq:fundmatrix-determinant-one}
    \det U(y;\lambda) = \det \twin U(y;\lambda) = 1
  \end{equation}
  and
  \begin{equation}
    \label{eq:fundmatrix-relation}
    \twin U(y;\lambda) = U(y;\lambda)^\sigma.
  \end{equation}
  In particular, $\det S(\lambda) = \det \twin S(\lambda) = 1$,
  and $\twin S(\lambda) = S(\lambda)^\sigma$.
\end{theorem}

\begin{proof}
  Equation \eqref{eq:fundmatrix-determinant-one} follows from Liouville's formula,
  since $\coeffmatrix$ is trace-free:
  \begin{equation*}
    \det U(y;\lambda)
    = \bigl( \det U(-1;\lambda) \bigr) \exp \int_{-1}^{y} \tr \coeffmatrix(\xi;\lambda) \, d\xi
    = (\det I) \exp 0 = 1,
  \end{equation*}
  and similarly for $\twin U$.
  To prove \eqref{eq:fundmatrix-relation},
  note first that
  \begin{equation*}
    \begin{split}
      \frac{\partial U(\lambda)^{-1}}{\partial y}
      &= - U(\lambda)^{-1} \, \frac{\partial U(\lambda)}{\partial y} \, U(\lambda)^{-1}
      \\
      &= - U(\lambda)^{-1} \, \coeffmatrix(\lambda) \, U(\lambda) \, U(\lambda)^{-1}
      = - U(\lambda)^{-1} \, \coeffmatrix(\lambda),
    \end{split}
  \end{equation*}
  which implies that
  \begin{equation*}
    \begin{split}
      \frac{\partial}{\partial y} U(\lambda)^\sigma
      &= \frac{\partial}{\partial y} \Bigl( \J \, U(-\lambda)^{-T} \, \J \Bigr)
      \\
      &= \J \, \left( \frac{\partial U(-\lambda)^{-1}}{\partial y} \right)^T \J
      \\
      &= \J \, \bigl( - U(-\lambda)^{-1} \, \coeffmatrix(-\lambda) \bigr)^T \J
      \\
      &= - \J \, \coeffmatrix(-\lambda)^T U(-\lambda)^{-T} \J.
    \end{split}
  \end{equation*}
  With the help of \autoref{lem:A-conjugation} this becomes
  \begin{equation*}
    \begin{split}
      \frac{\partial}{\partial y} U(\lambda)^\sigma
      &= \twin\coeffmatrix(\lambda) \, \J \, U(-\lambda)^{-T} \J
      \\
      &= \twin\coeffmatrix(\lambda) \, U(\lambda)^\sigma.
    \end{split}
  \end{equation*}
  Since $U(\lambda)^\sigma = I = \twin U(\lambda)$ when $y=-1$,
  we see that $U(\lambda)^\sigma$ and $\twin U(\lambda)$
  satisfy the same ODE and the same initial condition;
  hence they are equal for all~$y$ by uniqueness.
\end{proof}

\begin{corollary}
  \label{cor:S-relation}
  The transition matrices $S(\lambda)$ and $\twin S(\lambda)$ satisfy
  $\J S(\lambda)^T \J \twin S(-\lambda) = I$
  and
  $\twin S(-\lambda) = \J \big( \adj S(\lambda) \bigr)^T \J$,
  where $\adj$ denotes the adjugate (cofactor) matrix.
  In detail, this means that
  \begin{equation}
    \label{eq:S-relation-elementwise}
    \begin{pmatrix}
       S_{33} & -S_{23} &  S_{13} \\
      -S_{32} &  S_{22} & -S_{12} \\
       S_{31} & -S_{21} &  S_{11} \\
    \end{pmatrix}_{\!\!\lambda}
    \begin{pmatrix}
      \twin S_{11} & \twin S_{12} & \twin S_{31} \\
      \twin S_{21} & \twin S_{22} & \twin S_{32} \\
      \twin S_{31} & \twin S_{32} & \twin S_{33} \\
    \end{pmatrix}_{\!\!-\lambda}
    = \begin{pmatrix} 1 & 0 & 0 \\ 0 & 1 & 0 \\ 0 & 0 & 1 \end{pmatrix}
  \end{equation}
  and
  \begin{equation}
    \label{eq:S-cofactors}
    \twin S(-\lambda) =
    \begin{pmatrix}
      S_{11} S_{22} - S_{12} S_{21} &
      S_{11} S_{23} - S_{21} S_{13} &
      S_{12} S_{23} - S_{13} S_{22} \\
      S_{11} S_{32} - S_{12} S_{31} &
      S_{11} S_{33} - S_{13} S_{31} &
      S_{12} S_{33} - S_{13} S_{32} \\
      S_{21} S_{32} - S_{22} S_{31} &
      S_{21} S_{33} - S_{23} S_{31} &
      S_{22} S_{33} - S_{23} S_{32}
    \end{pmatrix}_{\!\!\lambda}.
  \end{equation}
  (The subscripts $\pm\lambda$ indicate the point where the matrix
  entries are evaluated.)
\end{corollary}

\subsection{Weyl functions}
\label{sec:Weyl-functions}

Consider next the boundary conditions
$\phi_2(-1) = \phi_3(-1) = 0 = \phi_3(1)$
in the two spectral problems \eqref{eq:spectral-problem-y}
and \eqref{eq:spectral-problem-y-twin}.
Fix some value of $\lambda \in \C$,
and let $\Phi = (\phi_1, \phi_2, \phi_3)^T$
be a solution of $d\Phi/dy = \coeffmatrix(\lambda) \Phi$
satisfying the boundary conditions at the left endpoint:
$\phi_2(-1) = \phi_3(-1) = 0$.
For normalization, we can take $\phi_1(-1)=1$; then the solution $\Phi$ is
unique, and its value at the right endpoint is given by the first column
of the transition matrix:
$\Phi(1) = S(\lambda) \Phi(-1) = S(\lambda) (1,0,0)^T
= (S_{11}(\lambda), S_{21}(\lambda), S_{31}(\lambda))^T$.
This shows that the boundary condition at the right endpoint, $\phi_3(1) = 0$,
is equivalent to $S_{31}(\lambda) = 0$.
In other words: $\lambda$ is an eigenvalue of the first spectral problem
\eqref{eq:spectral-problem-y} if and only if $S_{31}(\lambda) = 0$.

We define the following two \emph{Weyl functions} for the first
spectral problem using the entries from the first column of~$S(\lambda)$:
\begin{equation}
  \label{eq:WZ}
  W(\lambda) = - \frac{S_{21}(\lambda)}{S_{31}(\lambda)}
  ,
  \qquad
  Z(\lambda) = - \frac{S_{11}(\lambda)}{S_{31}(\lambda)}
  .
\end{equation}
The entries of $S(\lambda)$ depend analytically on the parameter~$\lambda$,
so the Weyl functions will be meromorphic, with poles (or possibly
removable singularities) at the eigenvalues.
The signs in \eqref{eq:WZ}
(and also in \eqref{eq:WZ-twin}, \eqref{eq:WZ-adjoint}, \eqref{eq:WZ-adjoint-twin} below)
are chosen so that the residues at these poles will be positive
when $g$ and~$h$ are positive;
see in particular \autoref{thm:weyl-parfrac}.

Similarly, $\lambda$ is an eigenvalue of the twin spectral problem
\eqref{eq:spectral-problem-y-twin} if and only if $\twin S_{31}(\lambda) = 0$,
and we define corresponding Weyl functions
\begin{equation}
  \label{eq:WZ-twin}
  \twin W(\lambda) = - \frac{\twin S_{21}(\lambda)}{\twin S_{31}(\lambda)},
  \qquad
  \twin Z(\lambda) = - \frac{\twin S_{11}(\lambda)}{\twin S_{31}(\lambda)}.
\end{equation}

\begin{theorem}
  \label{thm:Weyl-relation}
  The Weyl functions satisfy the relation
  \begin{equation}
    \label{eq:Weyl-relation}
    Z(\lambda) + W(\lambda) \twin W(-\lambda) + \twin Z(-\lambda) = 0.
  \end{equation}
\end{theorem}

\begin{proof}
  The $(3,1)$ entry in the matrix equality~\eqref{eq:S-relation-elementwise} is
  \begin{equation*}
    S_{31}(\lambda) \twin S_{11}(-\lambda)
    - S_{21}(\lambda) \twin S_{21}(-\lambda)
    + S_{11}(\lambda) \twin S_{31}(-\lambda)
    = 0.
  \end{equation*}
  Division by $-S_{31}(\lambda) \twin S_{31}(-\lambda)$ gives
  the desired result.
\end{proof}

\subsection{Adjoint spectral problems}
\label{sec:adjoint-spectral-problems}

Let us define a bilinear form
on vector-valued functions
\begin{equation*}
  \Phi(y) = \begin{pmatrix} \phi_1(y) \\ \phi_2(y) \\ \phi_3(y) \end{pmatrix}
\end{equation*}
with $\phi_1, \phi_2, \phi_3 \in L^2(-1,1)$:
\begin{equation}
  \label{eq:bilinear-R3}
  \begin{split}
    \bilinearRthree{\Phi}{\Omega}
    &= \int_{-1}^1 \Phi(y)^T \J \, \Omega(y) \, dy
    \\
    &= \int_{-1}^1 \bigl( \phi_1(y) \omega_3(y) - \phi_2(y) \omega_2(y) + \phi_3(y) \omega_1(y) \bigr) \, dy.
  \end{split}
\end{equation}
\autoref{lem:A-conjugation} implies that
\begin{equation*}
  \Bigl( A(\lambda) \Phi \Bigr)^T \J \, \Omega
  = - \Phi^T \J \Bigl( \twin A(-\lambda) \Omega \Bigr),
\end{equation*}
which, together with an integration by parts, leads to
\begin{equation}
  \label{eq:adjoint-operators}
  \bilinearRthree{\left( \tfrac{d}{dy} - A(\lambda) \right) \Phi}{\Omega}
  = \bigl[ \Phi^T \J \, \Omega \bigr]_{y=-1}^{1}
  - \bilinearRthree{\Phi}{\left( \tfrac{d}{dy} - \twin A(-\lambda) \right) \Omega}.
\end{equation}
Now, if $\Phi$ satisfies the boundary conditions
$\phi_2(-1) = \phi_3(-1) = 0 = \phi_3(1)$,
then what remains of the boundary term
$\bigl[ \Phi^T \J \, \Omega \bigr]_{-1}^{1}$
is
\begin{equation*}
  \phi_1(1) \omega_3(1) - \phi_2(1) \omega_2(1) - \phi_1(-1) \omega_3(-1),
\end{equation*}
and this can be killed by imposing the conditions
$\omega_3(-1) = 0 = \omega_2(1) = \omega_3(1)$.
Consequently, when acting on differentiable $L^2$ functions with such
boundary conditions, respectively, the operators
$\tfrac{d}{dy} - A(\lambda)$ and $-\tfrac{d}{dy} + \twin A(-\lambda)$
are adjoint to each other with respect to the bilinear form
$\bilinearRthree{\cdot}{\cdot}$:
\begin{equation*}
  \bilinearRthree{\left( \tfrac{d}{dy} - A(\lambda) \right) \Phi}{\Omega}
  = \bilinearRthree{\Phi}{\left( -\tfrac{d}{dy} + \twin A(-\lambda) \right) \Omega}.
\end{equation*}
This calculation motivates the following definition.
\begin{definition}
  The \emph{adjoint problem} to the spectral problem~\eqref{eq:spectral-problem-y} is
  \begin{subequations}
    \label{eq:adjoint}
    \begin{equation}
      \label{eq:adjoint-ODE}
      \frac{\partial \Omega}{\partial y} = \twin\coeffmatrix(-\lambda) \Omega,
    \end{equation}
    \begin{equation}
      \label{eq:adjoint-BC}
      \omega_3(-1) = 0 = \omega_2(1) = \omega_3(1).
    \end{equation}
  \end{subequations}
\end{definition}

\begin{proposition}
  \label{prop:adjoint-transition-matrix}
  Let $\Omega(y)$ be the unique solution of
  \eqref{eq:adjoint-ODE} which satisfies the boundary conditions
  at the right endpoint, $\omega_2(1) = \omega_3(1) = 0$,
  together with $\omega_1(1) = 1$ (for normalization).
  Then, at the left endpoint $y=-1$, we have
  \begin{equation}
    \label{eq:omega-in-terms-of-S}
    \Omega(-1)
    = \twin S(-\lambda)^{-1}
    \begin{pmatrix} 1 \\ 0 \\ 0 \end{pmatrix}
    = \J S(\lambda)^T \J
    \begin{pmatrix} 1 \\ 0 \\ 0 \end{pmatrix}
    =
    \begin{pmatrix}
      S_{33}(\lambda) \\ -S_{32}(\lambda) \\ S_{31}(\lambda)
    \end{pmatrix}.
  \end{equation}
\end{proposition}

\begin{proof}
  Since \eqref{eq:adjoint-ODE} agrees with the twin ODE
  \eqref{eq:dual-cubic-II} except for the sign of $\lambda$,
  the twin transition matrix with $\lambda$ negated,
  $\twin S(-\lambda)$, will relate boundary values of
  \eqref{eq:adjoint-ODE} at $y=-1$ to boundary values at $y=+1$:
  $(1,0,0)^T = \Omega(1) = \twin S(-\lambda) \Omega(-1)$.
  The rest follows from \autoref{cor:S-relation}.
\end{proof}

\begin{corollary}
  The adjoint problem \eqref{eq:adjoint} has the same spectrum as \eqref{eq:spectral-problem-y}.
\end{corollary}

\begin{proof}
  By \eqref{eq:omega-in-terms-of-S},
  the remaining boundary condition $\omega_3(-1) = 0$
  for \eqref{eq:adjoint}
  is equivalent to $S_{31}(\lambda) = 0$,
  which, as we saw in the previous section,
  is also the condition for $\lambda$ to be an
  eigenvalue of \eqref{eq:spectral-problem-y}.
\end{proof}

We define Weyl functions for the adjoint problem as follows,
using the entries from the third row of~$S(\lambda)$
appearing in \eqref{eq:omega-in-terms-of-S}:
\begin{equation}
  \label{eq:WZ-adjoint}
  W^*(\lambda) = -\frac{S_{32}(\lambda)}{S_{31}(\lambda)},
  \qquad
  Z^*(\lambda) = -\frac{S_{33}(\lambda)}{S_{31}(\lambda)}.
\end{equation}

To complete the picture, we note that there is of course also an adjoint problem
for the twin spectral problem \eqref{eq:spectral-problem-y-twin}, namely
\begin{subequations}
  \label{eq:adjoint-twin}
  \begin{equation}
    \label{eq:adjoint-twin-ODE}
    \frac{\partial \Omega}{\partial y} = \coeffmatrix(-\lambda) \Omega,
  \end{equation}
  \begin{equation}
    \label{eq:adjoint-twin-BC}
    \omega_3(-1) = 0 = \omega_2(1) = \omega_3(1).
  \end{equation}
\end{subequations}
A similar calculation as above shows that
the eigenvalues are given by the zeros of $\twin S_{31}(\lambda)$,
and hence they are the same as for \eqref{eq:spectral-problem-y-twin}.
We define the twin adjoint Weyl functions as
\begin{equation}
  \label{eq:WZ-adjoint-twin}
  \twin W^*(\lambda)
  = -\frac{\twin S_{32}(\lambda)}{\twin S_{31}(\lambda)},
  \qquad
  \twin Z^*(\lambda)
  = -\frac{\twin S_{33}(\lambda)}{\twin S_{31}(\lambda)}.
\end{equation}

\begin{theorem}
  \label{thm:adjoint-Weyl-relation}
  The adjoint Weyl functions satisfy the relation
  \begin{equation}
    \label{eq:adjoint-Weyl-relation}
    Z^*(\lambda) + W^*(\lambda) \twin W^*(-\lambda) + \twin Z^*(-\lambda) = 0.
  \end{equation}
\end{theorem}

\begin{proof}
  Since a matrix commutes with its inverse, we can equally well multiply the factors in \eqref{eq:S-relation-elementwise}
  in the opposite order: $\twin S(-\lambda) \cdot \J S(\lambda)^T \J = I$.
  Division of the $(3,1)$ entry in this identity by $-S_{31}(\lambda) \twin S_{31}(-\lambda)$ gives
  the result.
\end{proof}

\section{The discrete case}
\label{sec:discrete-on-the-interval}

We now turn to the discrete case \eqref{eq:mn-discrete},
when $m(x)$ and $n(x)$ are discrete measures (linear combinations
of Dirac deltas) with disjoint supports.
More specifically, we will study the \emph{interlacing} discrete case where
there are $N=2K$ sites numbered in ascending order,
\begin{equation*}
  x_1 <  x_2 < \dots < x_{2K},
\end{equation*}
with the measure~$m$ supported on the odd-numbered sites~$x_{2a-1}$,
and the measure~$n$ supported on the even-numbered sites~$x_{2a}$.
That is, we take $m_2=m_4=\dots=0$ and $n_1=n_3=\dots=0$,
so that
\begin{equation}
  \label{eq:mn-interlacing}
  \begin{split}
    m &= 2 \sum_{k=1}^N  m_k \, \delta_{x_k}
    = 2 \sum_{a=1}^K  m_{2a-1} \, \delta_{x_{2a-1}},
    \\
    n &= 2 \sum_{k=1}^N  n_k \, \delta_{x_k}
    = 2 \sum_{a=1}^K  n_{2a} \, \delta_{x_{2a}}.
  \end{split}
\end{equation}
We will also assume that the nonzero $m_k$ and $n_k$ are
\emph{positive}; this will be needed in order to prove that the
eigenvalues $\lambda=-z^2$ are positive.
The setup is illustrated in \autoref{fig:m-n-discrete}.

Given such a configuration, consisting of the $4K$ numbers $\{ x_k, m_{2a-1}, n_{2a} \}$,
we are going to define a set of $4K$ \emph{spectral variables}, consisting of
$2K-1$ \emph{eigenvalues} $\lambda_1,\dots,\lambda_K$ and $\mu_1,\dots,\mu_{K-1}$,
together with $2K+1$ \emph{residues} $a_1,\dots,a_K$, $b_1,\dots,b_{K-1}$, $b_\infty$ and~$b^*_\infty$.
In \autoref{sec:inverse-spectral} we will show that this correspondence
is a bijection onto the set of spectral variables with simple positive ordered eigenvalues
and positive residues (\autoref{thm:spectral-map-bijection}),
and give explicit formulas for the inverse map
(\autoref{cor:peakon-solution-formulas}).

\begin{remark}
  \label{rem:non-interlacing}
  Non-interlacing cases can be reduced to the interlacing case by
  introducing auxiliary weights at additional sites so that the
  problem becomes interlacing, and then letting these weights tend to
  zero in the solution of the interlacing inverse problem; the details will
  be published in another paper.
\end{remark}

\begin{remark}
  \label{rem:K1}
  The case $K=1$ is somewhat degenerate, and also rather trivial.
  It is dealt with separately in \autoref{sec:K1}.
  In what follows, we will (mostly without comment)
  assume that $K \ge 2$ whenever that is needed
  in order to make sense of the formulas.
\end{remark}

\begin{figure}
  \centering
  \ifthenelse{\isundefined{\draft}}{%
  \begin{tikzpicture}[scale=1]
    \draw[->,thick] (0,0) -- (10,0) node[below] {\small $x$};
    \begin{scope}[->, very thick]
      \draw (1,0) node[below] {\small $x_1$} -- +(0,3) node[above] {\small $2m_1$};
      \draw (3,0) node[below] {\small $x_3$} -- +(0,2) node[above] {\small $2m_3$};
      \draw (7.5,0) node[below] {\small $x_{2K-1}$} -- +(0,1) node[above] {\small $2m_{2K-1}$};
    \end{scope}
    \begin{scope}[->, very thick, dashed]
      \draw (1.7,0) node[below] {\small $x_2$} -- +(0,1) node[above] {\small $2n_2$};
      \draw (4.2,0) node[below] {\small $x_4$} -- +(0,1.5) node[above] {\small $2n_4$};
      \draw (9,0) node[below] {\small $x_{2K}$} -- +(0,3) node[above] {\small $2n_{2K}$};
    \end{scope}
    \draw (5.7,0) node[below] {$\cdots$};
  \end{tikzpicture}
  }{\textbf{[Draft mode; no picture]}}
  \caption{Notation for the measures $m$ and $n$ on the real line $\R$ in the interlacing discrete case \eqref{eq:mn-interlacing}.}
  \label{fig:m-n-discrete}
\end{figure}
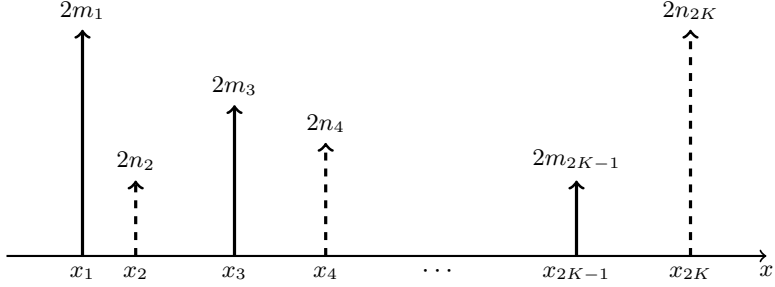

Under the transformation \eqref{eq:liouville-trf},
the discrete measures $m$ and~$n$ on~$\R$ are
mapped into discrete measures $g$ and~$h$, respectively,
supported at the points
\begin{equation}
  \label{eq:yk}
  y_k = \tanh x_k
\end{equation}
in the finite interval $(-1,1)$.
The formulas $g(y) = m(x) \cosh^3 x$ and $h(y) = n(x) \cosh^3 x$
from \eqref{eq:liouville-trf} should be interpreted using the relation
$\delta(x-x_k) dx = \delta(y-y_k) dy$, leading to
$\delta_{x_k}(x) = \delta_{y_k}(y) \frac{dy}{dx}(x_k) = \delta_{y_k}(y) / \cosh^{2} x_k$.
Since we will be working a lot with these measures,
it will be convenient to change the numbering a little,
and call the weights
$g_1,g_2,\dots,g_K$
and
$h_1,h_2,\dots,h_K$
rather than
$g_1,g_3,\dots,g_{2K-1}$
and
$h_2,h_4,\dots,h_{2K}$;
see \autoref{fig:g-h-discrete}.
With this numbering, we get
\begin{equation}
  \label{eq:gh}
  g = \sum_{a=1}^{K} g_{a} \, \delta_{y_{2a-1}}
  , \qquad
  h = \sum_{a=1}^{K} h_{a} \, \delta_{y_{2a}}
  ,
\end{equation}
where
\begin{equation}
  \label{eq:trf-discrete-measures}
  g_a = 2 m_{2a-1} \cosh x_{2a-1}
  , \qquad
  h_a = 2 n_{2a} \cosh x_{2a}
  .
\end{equation}

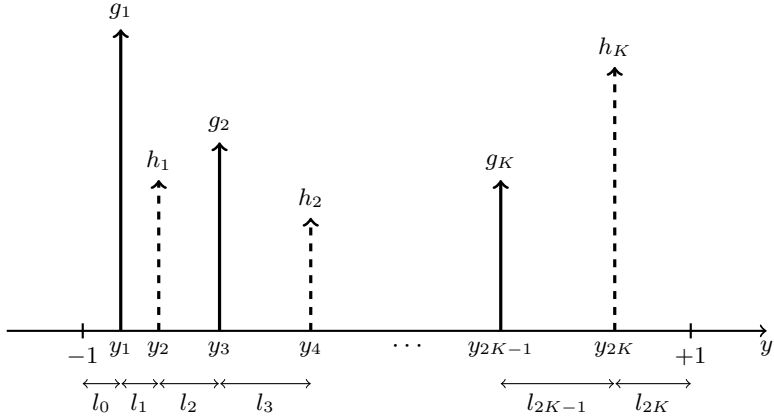
\begin{figure}
  \centering
  \ifthenelse{\isundefined{\draft}}{%
  \begin{tikzpicture}[scale=1]
    \draw[->,thick] (0,0) -- (10,0) node[below] {\small $y$};
    \begin{scope}[thick]
      \draw (1,0.1) -- +(0,-0.2) node[below] {\small $-1$};
      \draw (9,0.1) -- +(0,-0.2) node[below] {\small $+1$};
    \end{scope}
    \begin{scope}[->, very thick]
      \draw (1.5,0) node[below] {\small $y_1$} -- +(0,4) node[above] {\small $g_1$};
      \draw (2.8,0) node[below] {\small $y_3$} -- +(0,2.5) node[above] {\small $g_2$};
      \draw (6.5,0) node[below] {\small $y_{2K-1}$} -- +(0,2) node[above] {\small $g_{K}$};
    \end{scope}
    \begin{scope}[->, very thick, dashed]
      \draw (2,0) node[below] {\small $y_2$} -- +(0,2) node[above] {\small $h_1$};
      \draw (4,0) node[below] {\small $y_4$} -- +(0,1.5) node[above] {\small $h_2$};
      \draw (8,0) node[below] {\small $y_{2K}$} -- +(0,3.5) node[above] {\small $h_{K}$};
    \end{scope}
    \draw (5.3,0) node[below] {$\cdots$};
    \begin{scope}[<->,yshift=-0.7cm]
      \draw (1,0) -- node[below] {\small $l_0$} (1.49,0);
      \draw (1.51,0) -- node[below] {\small $l_1$} (1.99,0);
      \draw (2.01,0) -- node[below] {\small $l_2$} (2.79,0);
      \draw (2.81,0) -- node[below] {\small $l_3$} (4,0);
      \draw (6.5,0) -- node[below] {\small $l_{2K-1}$} (7.99,0);
      \draw (8.01,0) -- node[below] {\small $l_{2K}$} (9,0);
    \end{scope}
  \end{tikzpicture}
  }{\textbf{[Draft mode; no picture]}}
  \caption{Notation for the measures $g$ and $h$ on the finite interval $(-1,1)$ in the interlacing discrete case \eqref{eq:gh}.}
  \label{fig:g-h-discrete}
\end{figure}

\subsection{The first spectral problem}

The ODE \eqref{eq:dual-cubic-I},
$\partial_y \Phi = \coeffmatrix(y;\lambda) \Phi$,
is easily solved explicitly in the discrete case.
Since $g$ and $h$ are zero between the points~$y_k$,
the ODE reduces to
\begin{equation*}
  \frac{\partial}{\partial y}
  \begin{pmatrix}
    \phi_1 \\ \phi_2 \\ \phi_3
  \end{pmatrix}
  =
  \begin{pmatrix}
    0 & 0 & 0 \\
    0 & 0 & 0 \\
    -\lambda & 0 & 0
  \end{pmatrix}
  \begin{pmatrix}
    \phi_1 \\ \phi_2 \\ \phi_3
  \end{pmatrix}
\end{equation*}
in those intervals;
that is, $\phi_1$ and $\phi_2$ are constant in each interval
$y_k < y < y_{k+1}$ (for $0 \le k \le 2K$, where we let $y_0 = -1$ and $y_{2K+1} = +1$),
while $\phi_3$ is piecewise a polynomial in~$y$ of degree one.
The total change in the value of $\phi_3$ over the interval is given
by the product of the length of the interval, denoted
\begin{equation}
  \label{eq:lk}
  l_k = y_{k+1} - y_k,
\end{equation}
and the slope of the graph of~$\phi_3$;
this slope is $-\lambda$ times the constant value of $\phi_1$ in
the interval. In other words:
\begin{equation}
  \label{eq:propagation-L}
  \Phi(y_{k+1}^-) = L_k(\lambda) \Phi(y_k^+),
\end{equation}
where the propagation matrix~$L_k$ is defined by
\begin{equation}
  \label{eq:jump-matrix-L}
  L_k(\lambda) =
  \begin{pmatrix}
    1 & 0 & 0 \\
    0 & 1 & 0 \\
    -\lambda l_k & 0 & 1
  \end{pmatrix}.
\end{equation}
At the points~$y_k$, the ODE
forces the derivative $\partial_y \Phi$ to contain
a Dirac delta, and this imposes jump conditions on~$\Phi$.
These jump conditions will be of different type depending on whether $k$
is even or odd, since that affects whether
the Dirac delta is encountered in entry $(1,2)$ or $(2,3)$
in the coefficient matrix
\begin{equation*}
  \coeffmatrix(y;\lambda) =
  \begin{pmatrix}
    0 & h(y) & 0 \\
    0 & 0 & g(y) \\
    -\lambda & 0 & 0
  \end{pmatrix}.
\end{equation*}
More precisely, when $k=2a$ is even, we get a jump condition of the form
\begin{equation*}
  \Phi(y_k^+) - \Phi(y_k^-) =
  \begin{pmatrix}
    0 & h_a & 0 \\
    0 & 0 & 0 \\
    0 & 0 & 0
  \end{pmatrix}
  \Phi(y_k).
\end{equation*}
This implies that $\phi_2$ and $\phi_3$ don't jump at even-numbered~$y_k$,
and the continuity of $\phi_2$ in particular
implies that the jump in $\phi_1$ has a well-defined value $h_a \phi_2(y_{2a})$.

When $k=2a-1$ is odd, the condition is
\begin{equation*}
  \Phi(y_k^+) - \Phi(y_k^-) =
  \begin{pmatrix}
    0 & 0 & 0 \\
    0 & 0 & g_a \\
    0 & 0 & 0
  \end{pmatrix}
  \Phi(y_k).
\end{equation*}
Thus, $\phi_1$ and $\phi_3$ are continuous at odd-numbered~$y_k$,
and the jump in $\phi_2$ has a well-defined value $g_a \phi_3(y_{2a-1})$.

This step-by-step construction of $\Phi(y)$ is illustrated in
\autoref{fig:Phi} when
$\Phi(-1) = (1,0,0)^T$; as we have already seen,
this particular case is of interest in connection with
the spectral problem \eqref{eq:spectral-problem-y}
where we have boundary conditions
$\phi_2(-1) = \phi_3(-1) = 0 = \phi_3(1)$.

\begin{figure}
  \centering
  \ifthenelse{\isundefined{\draft}}{%
  \begin{tikzpicture}[scale=1]
    \draw[->,thick] (0,0) -- (10,0) node[below] {\small $y$};
    \begin{scope}[thick]
      \draw (1,0.1) -- +(0,-0.2) node[below] {\small $-1$};
    \end{scope}
    \draw (2,0) node[below] {\small $y_1$} (3,0) node[below] {\small $y_2$} (4,0) node[below] {\small $y_3$} (5,0) node[below] {\small $y_4$} (6,0) node[below] {\small $y_5$} (7,0) node[below] {\small $y_6$} (8,0) node[below] {\small $\cdots$};
    \begin{scope}[thick] 
      \draw (1,1) -- +(2,0);
      \draw (3,0.5) -- +(2,0);
      \draw (5,-0.5) -- +(2,0);
    \end{scope}
    \begin{scope}[thick, dashed] 
      \draw (1,0.05) -- +(1,0);
      \draw (2,-0.7) -- +(2,0);
      \draw (4,-1.3) -- +(2,0);
      \draw (6,-1.6) -- +(2,0);
    \end{scope}
    \begin{scope}[thick, dotted] 
      \draw (1,0) -- (3,-2) -- (5,-3) -- (7,-2);
    \end{scope}
    \draw (4.4,0.8) node[right] {\small $\phi_1(y;\lambda)$ (solid)};
    \draw (7.4,-1.3) node[right] {\small $\phi_2(y;\lambda)$ (dashed)};
    \draw (6,-2.7) node[right] {\small $\phi_3(y;\lambda)$ (dotted)};
  \end{tikzpicture}
  }{\textbf{[Draft mode; no picture]}}
  \caption{Structure of the solution to the initial value problem 
    $\partial_y \Phi = \coeffmatrix(y;\lambda) \Phi$
    with $\Phi(-1;\lambda) = (1,0,0)^T$,
    in the discrete interlacing case.
    The components $\phi_1$ and $\phi_2$ are piecewise constant,
    while $\phi_3$ is continuous and piecewise linear,
    with slope equal to $-\lambda$ times the value of~$\phi_1$.
    At the odd-numbered sites $y_{2a-1}$, the value of $\phi_2$ jumps by $g_a \phi_3(y_{2a-1})$.
    At the even-numbered sites $y_{2a}$, the value of $\phi_1$ jumps by $h_a \phi_2(y_{2a})$.
    The parameter~$\lambda$ is an eigenvalue of the spectral problem \eqref{eq:spectral-problem-y}
    iff it is a zero of $\phi_3(1;\lambda)$,
    which is a polynomial in $\lambda$ of degree $K+1$, with constant term zero.
    This picture illustrates a case where $\lambda$ and the weights $g_a$ and $h_a$ are
    all positive.
  }
  \label{fig:Phi}
\end{figure}
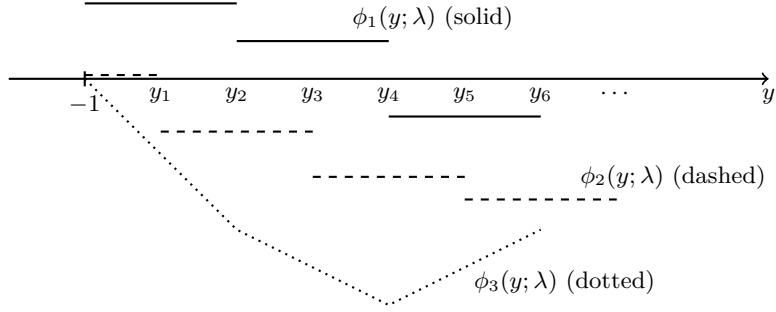

With the notation
\begin{equation}
  \label{eq:jump-matrix}
  \jumpmatrix{x}{y} =
  \begin{pmatrix}
    1 & x & \frac12 xy \\ 0 & 1 & y \\ 0 & 0 & 1
  \end{pmatrix}
  ,
\end{equation}
the jump conditions take the form
\begin{equation*}
  \Phi(y_{2a}^+) = \jumpmatrix{h_a}{0} \Phi(y_{2a}^-),
  \qquad
  \Phi(y_{2a-1}^+) = \jumpmatrix{0}{g_a} \Phi(y_{2a-1}^-).
\end{equation*}
(For the purposes of this paper, the top right entry of $\jumpmatrix{x}{y}$
might as well have been set equal to zero;
we have defined it as $\frac12 xy$ only to make $\jumpmatrix{x}{x}$
agree with a jump matrix appearing in our earlier work
\cite{lundmark-szmigielski:DPlong,hone-lundmark-szmigielski:novikov}.)
We can thus write the transition matrix $S(\lambda)$ as a product of $1+4K$ factors,
\begin{equation}
  \label{eq:S-discrete}
  S(\lambda) =
  L_{2K}(\lambda)
  \jumpmatrix{h_K}{0} L_{2K-1}(\lambda) \jumpmatrix{0}{g_K}  L_{2K-2}(\lambda) 
  \dotsm
  \jumpmatrix{h_1}{0} L_{1}(\lambda) \jumpmatrix{0}{g_1} L_{0}(\lambda).
\end{equation}
For later use in connection with the inverse problem,
we also consider the partial products $T_j(\lambda)$
containing the leftmost $1+4j$ factors (for $j=0,\dots,K$);
put differently, $T_{K-j}(\lambda)$ is
obtained by omitting all factors after $L_{2j}(\lambda)$
in the product for $S(\lambda)$:
\begin{equation}
  \label{eq:Tj}
  T_{K-j}(\lambda) =
  L_{2K}(\lambda)
  \dotsm
  \jumpmatrix{h_{j+1}}{0} L_{2j+1}(\lambda) \jumpmatrix{0}{g_{j+1}} L_{2j}(\lambda)
  .
\end{equation}
Thus $S(\lambda)=T_K(\lambda)$,
and $T_{K-j}(\lambda)$ depends on
$(g_{j+1},\dots,g_K)$,
$(h_{j+1},\dots,h_K)$,
$(l_{2j},\dots,l_{2K})$.

\begin{proposition}
  \label{prop:degrees-Tj}
  The entries of $T_j(\lambda)$ are polynomials in $\lambda$,
  with degrees as follows:
  \begin{equation}
    \label{eq:degrees-Tj}
    \deg T_j(\lambda) =
    \begin{pmatrix}
      j & j-1 & j-1 \\
      j & j-1 & j-1 \\
      j+1 & j & j
    \end{pmatrix}
    \qquad
    (j \ge 1).
  \end{equation}
  The constant term in each entry is given by
  \begin{equation}
    \label{eq:constant-term-Tj}
    T_{K-j}(0) = 
    \jumpmatrix{h_K}{0}
    \jumpmatrix{0}{g_K}
    \dotsm
    \jumpmatrix{h_{j+1}}{0}
    \jumpmatrix{0}{g_{j+1}}
    =
    \begin{pmatrix}
      1 & \displaystyle \sum_{a>j} h_a & \displaystyle \sum_{a \ge b > j} h_a g_b \\[1.2em]
      0 & 1 & \displaystyle \sum_{a>j} g_a \\[1.2em]
      0 & 0 & 1
    \end{pmatrix}.
  \end{equation}
  For those entries whose constant term is zero,
  the coefficient of $\lambda^1$ is given by
  \begin{equation}
    \label{eq:linear-term-Tj}
    \frac{d T_{K-j}}{d\lambda}(0) =
    \begin{pmatrix}
      * & * & * \\[1ex]
      \displaystyle -\sum_{a>j} \sum_{k=2j}^{2a-2} g_a l_k & *  & *\\[1.2em]
      \displaystyle -\sum_{k=2j}^{2K} l_k & \displaystyle -\sum_{a>j} \sum_{k=2a}^{2K} h_a l_k & *
    \end{pmatrix}.
  \end{equation}
  The highest coefficient in the $(3,1)$ entry is given by
  \begin{equation}
    \label{eq:highest-term-Tj31}
    (T_j)_{31}(\lambda) = (-\lambda)^{j+1} \left( \prod_{m=K-j}^K l_{2m} \right) \left( \prod_{a=K+1-j}^K g_a h_a \right)
    + \dotsb
    .
  \end{equation}
\end{proposition}

\begin{proof}
  Equation \eqref{eq:constant-term-Tj} follows at once from setting $\lambda = 0$
  in \eqref{eq:Tj}.
  Next, group the factors in fours (except for the lone first factor $L_{2K}(\lambda)$)
  so that \eqref{eq:Tj} takes the form
  $T_{K-j} = L_{2K} \, t_{K} \, t_{K-1} \dotsm t_{j+1}$,
  where
  \begin{equation*}
    t_a(\lambda)
    = \jumpmatrix{h_a}{0} L_{2a-1}(\lambda) \jumpmatrix{0}{g_a} L_{2a-2}(\lambda)
    =
    \begin{smallpmatrix}
      1 & h_a & h_a g_a \\
      0 & 1 & g_a \\
      0 & 0 & 1
    \end{smallpmatrix}
    - \lambda
    \begin{smallpmatrix}
      l_{2a-2} h_a g_a \\
      l_{2a-2} g_a \\
      l_{2a-1} + l_{2a-2}
    \end{smallpmatrix}
    (1,0,0).
  \end{equation*}
  The degree count \eqref{eq:degrees-Tj} follows easily by considering
  the highest power of $\lambda$ arising from multiplying these factors,
  and \eqref{eq:highest-term-Tj31} also falls out of this.
  Differentiating $T_{K-j+1}(\lambda) = T_{K-j}(\lambda) \, t_j(\lambda)$
  and letting $\lambda = 0$ gives
  \begin{equation}
    \label{eq:Tj-prime-recursion}
    \frac{d T_{K-j+1}}{d\lambda}(0) =
    \frac{d T_{K-j}}{d\lambda}(0)
    \begin{smallpmatrix}
      1 & h_j & h_j g_j \\
      0 & 1 & g_j \\
      0 & 0 & 1
    \end{smallpmatrix}
    - T_{K-j}(0)
    \begin{smallpmatrix}
      l_{2j-2} h_j g_j & 0 & 0 \\
      l_{2j-2} g_j & 0 & 0 \\
      l_{2j-1} + l_{2j-2} & 0 & 0
    \end{smallpmatrix}.
  \end{equation}
  With the help of \eqref{eq:constant-term-Tj} one sees that
  the $(3,1)$ entry of this equality reads
  \begin{equation}
    (T'_{K-j+1})_{31}(0) = (T'_{K-j})_{31}(0) - (l_{2j-1} + l_{2j-2}),
  \end{equation}
  the $(3,2)$ entry is
  \begin{equation}
    (T'_{K-j+1})_{32}(0) =
    (T'_{K-j})_{32}(0) + h_j (T_{K-j})'_{31}(0),
  \end{equation}
  and the $(2,1)$ entry is
  \begin{equation}
    (T'_{K-j+1})_{21}(0) =
    (T'_{K-j})_{21}(0) - l_{2j-2} g_j - \biggl( \sum_{a>j} g_a \biggr) (l_{2j-1} + l_{2j-2}).
  \end{equation}
  Solving these recurrences, with the inital conditions coming from
  $T_0'(0) = L_{2K}'(0)$ (i.e., $-l_{2K}$ in the $(3,1)$ position, zero elsewhere),
  gives equation \eqref{eq:linear-term-Tj}.
\end{proof}

We also state the result for the important special case $S(\lambda) = T_K(\lambda)$.
(In the $(3,1)$ entry, $\sum_{k=0}^{2K} l_k = 2$
is the length of the whole interval $[-1,1]$.)

\begin{corollary}
  \label{prop:degrees-S}
  The entries of $S(\lambda)$ are polynomials in $\lambda$,
  with
  \begin{equation}
    \label{eq:degrees-S}
    \deg S(\lambda) =
    \begin{pmatrix}
      K & K-1 & K-1 \\
      K & K-1 & K-1 \\
      K+1 & K & K
    \end{pmatrix},
  \end{equation}
  \begin{equation}
    \label{eq:constant-term-S}
    S(0) = 
    \begin{pmatrix}
      1 & \displaystyle \sum_{a} h_a & \displaystyle \sum_{a \ge b} h_a g_b \\[1.2em]
      0 & 1 & \displaystyle \sum_{a} g_a \\[1.2em]
      0 & 0 & 1
    \end{pmatrix},
  \end{equation}
  \begin{equation}
    \label{eq:linear-term-S}
    S'(0) =
    \begin{pmatrix}
      * & * & * \\[1ex]
      \displaystyle -\sum_{a=1}^K \sum_{k=0}^{2a-2} g_a l_k & *  & * \\[1.2em]
      \displaystyle -\sum_{k=0}^{2K} l_k & \displaystyle -\sum_{a=1}^K \sum_{k=2a}^{2K} h_a l_k & *
    \end{pmatrix},
  \end{equation}
  \begin{equation}
    \label{eq:highest-term-S31}
    S_{31}(\lambda) = (-\lambda)^{K+1} \left( \prod_{m=0}^K l_{2m} \right) \left( \prod_{a=1}^K g_a h_a \right)
    + \dotsb
    .
  \end{equation}
\end{corollary}

\begin{remark}
  The ideas that we use go back to Stieltjes's memoir on
  continued fractions \cite{stieltjes} and its relation to an
  inhomogeneous string problem, especially its inverse problem,
  discovered by Krein in the 1950s.
  A comprehensive account of the inverse string problem can be found in
  \cite{dym-mckean:gaussian}, especially Section~5.9.
  The connection to Stieltjes continued fractions
  is explained in \cite[Supplement~II]{gantmacher-krein}
  and in \cite{beals-sattinger-szmigielski:stieltjes}. 
  Briefly stated, if $\phi(y;\lambda)$ satisfies the string equation
  \begin{equation*}
    -\phi_{yy} = \lambda g(y) \phi
    ,\qquad
    -1 < y < 1
    ,\qquad
    \phi(-1;\lambda) = 0,
  \end{equation*}
  with a discrete mass distribution $g(y)=\sum_{j=1}^n g_j \delta_{y_j}$,
  then the Weyl function
  $W(\lambda) = \frac{\phi_y(1;\lambda)}{\phi(1;\lambda)}$
  admits the continued fraction expansion
  \begin{equation*}
    \label{eq:W-contfrac}
    W(z)
    = \cfrac{1}{l_n +
      \cfrac{1}{-z g_n +
        \cfrac{1}{l_{n-1} +
          \cfrac{1}{\raisebox{1.5ex}{$\ddots$} +
            \cfrac{1}{-z g_2 +
              \cfrac{1}{l_1+
                \cfrac{1}{-z g_1 +
                  \cfrac{1}{l_0}
                }}}}}}}
  \end{equation*}
  (where $l_j = y_{j+1}-y_j$),
  whose convergents (Pad\'{e} approximants)
  $T_{2j}(\lambda) = \frac{P_{2j}(\lambda)}{Q_{2j}(\lambda)}$ satisfy
  \begin{equation*}
    \begin{aligned}
      P_{2j}(\lambda) &= (-1)^j g_n \biggl( \prod_{k=n-j+1}^{n-1} l_k g_k \biggr) \lambda^{j} + \dotsb
      ,\\
      Q_{2j}(\lambda) &= (-1)^j \biggl( \prod_{k=n-j+1}^n l_k g_k \biggr) \lambda ^j + \dotsb
      .
    \end{aligned}
  \end{equation*}
\end{remark}

\subsection{The second spectral problem}

For the twin ODE \eqref{eq:dual-cubic-II},
$\partial_y \Phi = \twin\coeffmatrix(y;\lambda) \Phi$,
where the measures $g$ and $h$ are swapped,
the construction is similar.
The only difference is that the weights $g_a$
at the odd-numbered sites will occur in the type of jump condition
that we previously had for the weights $h_a$ at the even-numbered sites
(and vice versa).
Thus, the transition matrix is in this case
\begin{equation}
  \label{eq:S-twin-discrete}
  \twin S(\lambda) =
  L_{2K}(\lambda) \jumpmatrix{0}{h_K}
  L_{2K-1}(\lambda) \jumpmatrix{g_K}{0}
  L_{2K-2}(\lambda) 
  \dotsm
  \jumpmatrix{0}{h_1}
  L_{1}(\lambda) \jumpmatrix{g_1}{0}
  L_{0}(\lambda).
\end{equation}
This solution is illustrated in \autoref{fig:Phi-twin}
for the initial condition $\Phi(-1) = (1,0,0)^T$.
It is clear that it behaves a bit differently,
since the first weight $g_1$ has no influence on this
solution $\Phi$, and therefore not on the second spectrum either.
(The first column in $\jumpmatrix{g_1}{0} L_{0}(\lambda)$
does not depend on~$g_1$.)

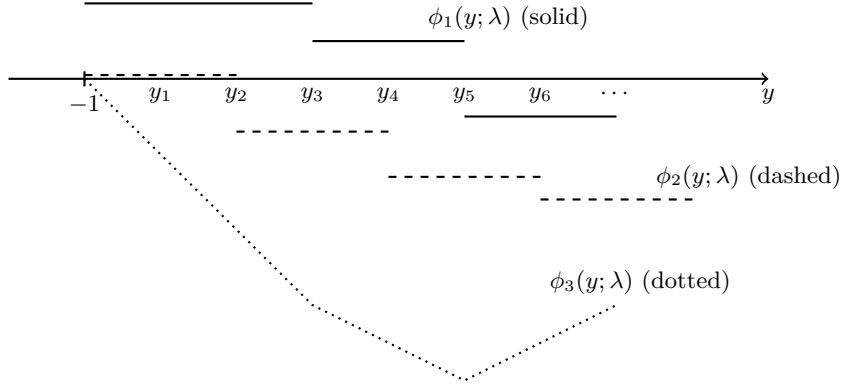
\begin{figure}
  \centering
  \ifthenelse{\isundefined{\draft}}{%
  \begin{tikzpicture}[scale=1]
    \draw[->,thick] (0,0) -- (10,0) node[below] {\small $y$};
    \begin{scope}[thick]
      \draw (1,0.1) -- +(0,-0.2) node[below] {\small $-1$};
    \end{scope}
    \draw (2,0) node[below] {\small $y_1$} (3,0) node[below] {\small $y_2$} (4,0) node[below] {\small $y_3$} (5,0) node[below] {\small $y_4$} (6,0) node[below] {\small $y_5$} (7,0) node[below] {\small $y_6$} (8,0) node[below] {\small $\cdots$};
    \begin{scope}[thick] 
      \draw (1,1) -- +(3,0);
      \draw (4,0.5) -- +(2,0);
      \draw (6,-0.5) -- +(2,0);
    \end{scope}
    \begin{scope}[thick, dashed] 
      \draw (1,0.05) -- +(2,0);
      \draw (3,-0.7) -- +(2,0);
      \draw (5,-1.3) -- +(2,0);
      \draw (7,-1.6) -- +(2,0);
    \end{scope}
    \begin{scope}[thick, dotted] 
      \draw (1,0) -- (4,-3) -- (6,-4) -- (8,-3);
    \end{scope}
    \draw (5.4,0.8) node[right] {\small $\phi_1(y;\lambda)$ (solid)};
    \draw (8.4,-1.3) node[right] {\small $\phi_2(y;\lambda)$ (dashed)};
    \draw (7,-2.7) node[right] {\small $\phi_3(y;\lambda)$ (dotted)};
  \end{tikzpicture}
  }{\textbf{[Draft mode; no picture]}}
  \caption{Structure of the solution to the twin problem 
    $\partial_y \Phi = \twin\coeffmatrix(y;\lambda) \Phi$
    with $\Phi(-1;\lambda) = (1,0,0)^T$,
    in the discrete interlacing case.
    The differences compared to \autoref{fig:Phi} are the following:
    At the odd-numbered sites $y_{2a-1}$, the value of~$\phi_1$ (not~$\phi_2$) jumps by $g_a \phi_2(y_{2a-1})$.
    At the even-numbered sites $y_{2a}$, the value of~$\phi_2$ (not~$\phi_1$) jumps by $h_a \phi_3(y_{2a})$.
    The parameter~$\lambda$ is an eigenvalue of the twin spectral problem \eqref{eq:spectral-problem-y-twin}
    iff it is a zero of $\phi_3(1;\lambda)$,
    which is a polynomial in $\lambda$ of degree $K$ (not $K+1$), with constant term zero.
    Note that the first mass $g_1$ has no influence here.
    (Indeed, since $\phi_2(y_1;\lambda) = 0$, there is no jump in $\phi_1$ at~$y=y_1$,
    regardless of the value of~$g_1$.)
  }
  \label{fig:Phi-twin}
\end{figure}

Let $\twin T_j(\lambda)$ be the partial product
containing the first $1+4j$ factors in the product for $\twin S(\lambda)$;
in other words,
\begin{equation}
  \label{eq:Tj-twin}
  \twin T_{K-j}(\lambda) =
  L_{2K}(\lambda)
  \dotsm
  \jumpmatrix{0}{h_{j+1}} L_{2j+1}(\lambda) \jumpmatrix{g_{j+1}}{0} L_{2j}(\lambda).
\end{equation}

\begin{proposition}
  \label{prop:degrees-Tj-twin}
  The entries of $\twin T_j(\lambda)$ are polynomials in $\lambda$,
  satisfying
  \begin{equation}
    \label{eq:T1-twin}
    \twin T_1(\lambda) =
    \begin{pmatrix}
      1 & g_K & 0 \\
      -\lambda h_K (l_{2K-1}+l_{2K-2}) & 1- \lambda h_K g_K l_{2K-1} & h_K \\
      -\lambda (l_{2K}+l_{2K-1}+l_{2K-2}) & -\lambda g_K (l_{2K}+l_{2K-1}) & 1
    \end{pmatrix},
  \end{equation}
  \begin{equation}
    \label{eq:degrees-Tj-twin}
    \deg \twin T_j(\lambda) =
    \begin{pmatrix}
      j-1 & j-1 & j-2 \\
      j & j & j-1 \\
      j & j & j-1
    \end{pmatrix}
    \qquad
    (j \ge 2),
  \end{equation}
  \begin{equation}
    \label{eq:constant-term-Tj-twin}
    \twin T_{K-j}(0) = 
    \jumpmatrix{0}{h_K}
    \jumpmatrix{g_K}{0}
    \dotsm
    \jumpmatrix{0}{h_{j+1}}
    \jumpmatrix{g_{j+1}}{0}
    =
    \begin{pmatrix}
      1 & \displaystyle\sum_{a>j} g_a & \displaystyle\sum_{a > b > j} g_a h_b \\[1.2em]
      0 & 1 & \displaystyle\sum_{a>j} h_a \\[1.2em]
      0 & 0 & 1
    \end{pmatrix},
  \end{equation}
  \begin{equation}
    \label{eq:linear-term-Tj-twin}
    \frac{d \twin T_{K-j}}{d\lambda}(0) =
    \begin{pmatrix}
      * & * & * \\[1ex]
      \displaystyle -\sum_{a>j} \sum_{k=2j}^{2a-1} h_a l_k & *  & *\\[1.2em]
      \displaystyle -\sum_{k=2j}^{2K} l_k & \displaystyle -\sum_{a>j} \sum_{k=2a-1}^{2K} g_a l_k & *
    \end{pmatrix}
    .
  \end{equation}
  For $0 \le j \le K-2$,
  the highest coefficients in the $(2,1)$ and $(3,1)$ entries are given by
  \begin{equation}
    \label{eq:highest-term-Tj21-twin}
    (\twin T_{K-j})_{21}(\lambda) = (-\lambda)^{K-j} \left( \prod_{m=j+2}^{K} l_{2m-1} \right) (l_{2j+1}+l_{2j}) \, h_K \left( \prod_{a=j}^{K-1} g_{a+1} h_a \right)
    + \dotsb
    ,
  \end{equation}
  \begin{multline}
    \label{eq:highest-term-Tj31-twin}
    (\twin T_{K-j})_{31}(\lambda) =
    \\
    (-\lambda)^{K-j} (l_{2K}+l_{2K-1}) \left( \prod_{m=j+2}^{K-1} l_{2m-1} \right) (l_{2j+1}+l_{2j}) \left( \prod_{a=1}^{K-1} g_{a+1} h_a \right)
    + \dotsb
    ,
  \end{multline}
  where $\prod_{m=K}^{K-1} = 1$.
  Moreover, $\twin T_j$ and $T_j$ are related by the involution $\sigma$
  (see \autoref{def:involution-sigma}):
  \begin{equation}
    \label{eq:Tj-involution}
    \twin T_j(\lambda) = T_j(\lambda)^\sigma.
  \end{equation}
\end{proposition}

\begin{proof}
  The degree count and the coefficients are obtained like in the proof of
  \autoref{prop:degrees-Tj}, although the details are a bit more
  involved in this case.
  (Group the factors in $\twin T_j$ as follows:
  $L_{2K}(\lambda)$ times a pair of factors,
  times a number a quadruples of the same form as $t_a(\lambda)$ in
  the proof of \autoref{prop:degrees-Tj} but with $h_a$ and~$g_a$
  replaced by $g_{a+1}$ and~$h_a$ respectively,
  times a final pair at the end.)

  The $\sigma$-relation \eqref{eq:Tj-involution} can be seen as yet
  another manifestation of \autoref{thm:fundmatrix-relation}, and
  (since $\sigma$ is a group homomorphism) it also follows directly
  from the easily verified formulas
  $L_k(\lambda)^\sigma = L_k(\lambda)$ and $\jumpmatrix{x}{y}^\sigma = \jumpmatrix{y}{x}$.
\end{proof}

We record the results in particular for the case $\twin S(\lambda) = \twin T_K(\lambda)$:
\begin{corollary}
  \label{prop:degrees-S-twin}
  The entries of $\twin S(\lambda)$ are polynomials in $\lambda$,
  satisfying
  \begin{equation}
    \label{eq:degrees-S-twin}
    \deg \twin S(\lambda) =
    \begin{pmatrix}
      K-1 & K-1 & K-2 \\
      K & K & K-1 \\
      K & K & K-1
    \end{pmatrix},
  \end{equation}
  \begin{equation}
    \label{eq:constant-term-S-twin}
    \twin S(0) =
    \begin{pmatrix}
      1 & \displaystyle \sum_{a} g_a & \displaystyle \sum_{a > b} g_a h_b \\[1.2em]
      0 & 1 & \displaystyle \sum_{a} h_a \\[1.2em]
      0 & 0 & 1
    \end{pmatrix},
  \end{equation}
  \begin{equation}
    \label{eq:linear-term-S-twin}
    \twin S'(0) =
    \begin{pmatrix}
      * & * & * \\[1ex]
      \displaystyle -\sum_{a=1}^K \sum_{k=0}^{2a-1} h_a l_k & *  & * \\[1.2em]
      -2 & \displaystyle -\sum_{a=1}^K \sum_{k=2a-1}^{2K} g_a l_k & *
    \end{pmatrix}.
  \end{equation}
  (The interpretation of \eqref{eq:degrees-S-twin} when $K=1$ is that
  the $(1,3)$ entry is the zero polynomial.)
  The leading terms of $\twin S_{21}(\lambda)$
  and~$\twin S_{31}(\lambda)$ are given by
  \begin{equation}
    \label{eq:highest-term-S21-twin}
    \twin S_{21}(\lambda) = (-\lambda)^{K} \left( \prod_{m=2}^{K} l_{2m-1} \right) (l_{0}+l_{1}) \, h_K \left( \prod_{a=1}^{K-1} g_{a+1} h_a \right)
    + \dotsb
    ,
  \end{equation}
  \begin{equation}
    \label{eq:highest-term-S31-twin}
    \twin S_{31}(\lambda) = (-\lambda)^{K} (l_{2K}+l_{2K-1}) \left( \prod_{m=2}^{K-1} l_{2m-1} \right) (l_{0}+l_{1}) \left( \prod_{a=1}^{K-1} g_{a+1} h_a \right)
    + \dotsb
    ,
  \end{equation}
  with the exception of the case $K=1$ where we simply have $\twin S_{31}(\lambda) = -2\lambda$.
  (The empty product $\prod_{m=2}^{1} l_{2m-1}$ is omitted from $\twin S_{31}$ in the case $K=2$,
  and from $\twin S_{21}$ in the case $K=1$.)
\end{corollary}

\subsection{Weyl functions and spectral measures}

Since the entries of $S(\lambda)$ are polynomials,
the Weyl functions $W=-S_{21}/S_{31}$ and $Z=-S_{11}/S_{31}$
are rational functions in the discrete case.
They have poles at the eigenvalues of the spectral problem \eqref{eq:spectral-problem-y}.
Likewise, the twin Weyl functions
$\twin W=-\twin S_{21}/\twin S_{31}$, $\twin Z=-\twin S_{11}/\twin S_{31}$
are rational functions, with poles 
at the eigenvalues of the twin spectral problem \eqref{eq:spectral-problem-y-twin}.

\begin{theorem}
  \label{thm:simple-spectra}
  If all $g_k$ and $h_k$ are positive,
  then both spectra are nonnegative and simple.
  The eigenvalues of \eqref{eq:spectral-problem-y} and
  \eqref{eq:spectral-problem-y-twin} will be denoted by
  \begin{equation}
    0 = \lambda_0 < \lambda_1 < \dots < \lambda_K
    \qquad\text{(zeros of $S_{31}$)},
  \end{equation}
  \begin{equation}
    0 = \mu_0 < \mu_1 < \dots < \mu_{K-1}
    \qquad\text{(zeros of $\twin S_{31}$)}.
  \end{equation}
\end{theorem}

\begin{proof}
  This is proved in the appendix; see \autoref{thm:positive-simple-spectrum}.
  (It is clear that if the zeros of the polynomials $S_{31}(\lambda)$ and $\twin S_{31}(\lambda)$
  are real, then they can't be negative, since the coefficients in the polynomials
  have alternating signs and all terms therefore have the same sign if $\lambda<0$.
  However, it's far from obvious that the zeros are real, much less simple.
  These facts follow from properties of oscillatory matrices,
  belonging to the beautiful theory of oscillatory kernels
  due to Gantmacher and Krein; see~\cite[Ch.~II]{gantmacher-krein}.)
\end{proof}

\begin{remark}
  Taking \hyperref[prop:degrees-S]{Propositions \ref*{prop:degrees-S}} and~\ref{prop:degrees-S-twin} into account,
  we can thus write
  \begin{equation}
    \label{eq:S31-factorized}
    S_{31}(\lambda) = -2 \lambda \prod_{i=1}^{K} \left( 1 - \frac{\lambda}{\lambda_i} \right),
    \qquad
    \twin S_{31}(\lambda) = -2 \lambda \prod_{j=1}^{K-1} \left( 1 - \frac{\lambda}{\mu_j} \right).
  \end{equation}
\end{remark}

\begin{theorem}
  \label{thm:weyl-parfrac}
  If all $g_k$ and $h_k$ are positive,
  then the Weyl functions have partial fraction decompositions
  \begin{subequations}
    \label{eq:weyl-parfrac}
    \begin{align}
      \label{eq:W-parfrac}
      W(\lambda) &= \sum_{i=1}^K \frac{a_i}{\lambda - \lambda_i},
      \\
      \label{eq:W-twin-parfrac}
      \twin W(\lambda) &= -b_{\infty} + \sum_{j=1}^{K-1} \frac{b_j}{\lambda - \mu_j},
      \\
      \label{eq:Z-parfrac}
      Z(\lambda) &= \frac{1}{2\lambda} + \sum_{i=1}^K \frac{c_i}{\lambda - \lambda_i},
      \\
      \label{eq:Z-twin-parfrac}
      \twin Z(\lambda) &= \frac{1}{2\lambda} + \sum_{j=1}^{K-1} \frac{d_j}{\lambda - \mu_j},
    \end{align}
  \end{subequations}
  where $a_i$, $b_j$, $b_{\infty}$, $c_i$, $d_j$ are positive,
  and where $W$ and~$\twin W$ determine
  $Z$ and~$\twin Z$ through the relations
  \begin{equation}
    \label{eq:residue-relation}
    c_i = a_i b_{\infty} + \sum_{j=1}^{K-1} \frac{a_i b_j}{\lambda_i + \mu_j},
    \qquad
    d_j = \sum_{i=1}^K \frac{a_i b_j}{\lambda_i + \mu_j}.
  \end{equation}
\end{theorem}

\begin{proof}
  The form of the decompositions follows from
  \hyperref[prop:degrees-S]{Propositions \ref*{prop:degrees-S}} and~\ref{prop:degrees-S-twin}
  (polynomial degrees),
  together with \autoref{thm:simple-spectra} (all poles are simple).
  In $W = -S_{21} / S_{31}$ the factor $\lambda$ cancels, so there is no residue at $\lambda = 0$,
  and similarly for $\twin W = -\twin S_{21} / \twin S_{31}$
  (which however is different from $W$ in that the degree
  of the numerator equals the degree of the denominator; hence the constant term~$-b_{\infty}$).
  The residue
  of $Z = -S_{11} / S_{31}$
  at $\lambda = 0$
  is $-S_{11}(0)/S'_{31}(0) = 1/2$ by \autoref{prop:degrees-S},
  and similarly for $\twin Z(\lambda)$.

  From the expressions \eqref{eq:highest-term-S21-twin} and
  \eqref{eq:highest-term-S31-twin} for the highest coefficient of
  $\twin S_{21}$ and~$\twin S_{31}$ we obtain (for $K \ge 2$)
  \begin{equation}
    \label{eq:formula-for-b-infty}
    b_{\infty} = -\lim_{\lambda \to \infty} \twin W(\lambda)
    = \lim_{\lambda \to \infty} \frac{\twin S_{21}(\lambda)}{\twin S_{31}(\lambda)}
    = \frac{h_K l_{2K-1}}{l_{2K}+l_{2K-1}}
    ,
  \end{equation}
  which shows that $b_\infty > 0$.
  (In the exceptional case $K=1$ we have instead
  $-\twin W(\lambda) = \frac12 h_1 (l_2 + l_1) = b_\infty > 0$.)

  The proof that $a_i$ and~$b_j$ are positive will be given at the end
  of \autoref{sec:adjoint-Weyl-discrete}.
  It will then follow from \eqref{eq:residue-relation} that
  $c_i$ and~$d_j$ are positive as well.

  To prove \eqref{eq:residue-relation},
  recall the relation
  $Z(\lambda) + W(\lambda) \twin W(-\lambda) + \twin Z(-\lambda) = 0$
  from \autoref{thm:Weyl-relation}.
  Taking the residue at $\lambda = \lambda_i$ on both sides yields
  \begin{equation*}
    c_i + a_i \twin W(-\lambda_i) + 0 = 0.
  \end{equation*}
  Taking instead the residue at $\lambda = \mu_j$ in
  $Z(-\lambda) + W(-\lambda) \twin W(\lambda) + \twin Z(\lambda) = 0$,
  we obtain
  \begin{equation*}
    0 + W(-\mu_j) b_j + d_j = 0.
  \end{equation*}
\end{proof}

\begin{definition}[Spectral measures]
  Let $\alpha$ and $\beta$ denote the discrete measures
  \begin{equation}
    \label{eq:spectral-measures}
    \alpha = \sum_{i=1}^K a_i \delta_{\lambda_i},
    \qquad
    \beta = \sum_{j=1}^{K-1} b_j \delta_{\mu_j},
  \end{equation}
  where $a_i$ and $b_j$ are the residues in $W(\lambda)$ and $\twin W(\lambda)$
  from \eqref{eq:W-parfrac} and \eqref{eq:W-twin-parfrac}.
\end{definition}

We can write $W$ and $\twin W$
in terms of these spectral measures $\alpha$ and $\beta$,
and likewise for $Z$ and $\twin Z$ if we use \eqref{eq:residue-relation}:
\begin{subequations}
  \label{eq:WZ-as-integrals}
  \begin{align}
    \label{eq:W-as-integral}
    W(\lambda) &= \int \frac{\da}{\lambda - x},
    \\
    \label{eq:W-twin-as-integral}
    \twin W(\lambda) &= \int \frac{\db}{\lambda - y} - b_{\infty},
    \\
    \label{eq:Z-as-integral}
    Z(\lambda) &= \frac{1}{2\lambda} + \iint \frac{\da\db}{(\lambda - x)(x+y)} + b_{\infty} W(\lambda),
    \\
    \label{eq:Z-twin-as-integral}
    \twin Z(\lambda) &= \frac{1}{2\lambda} + \iint \frac{\da\db}{(x+y)(\lambda - y)}.
  \end{align}
\end{subequations}
(Note the appearance here of the Cauchy kernel $1/(x+y)$.)

We have now completed the spectral characterization of the boundary
value problems \eqref{eq:dual-cubic-I} and \eqref{eq:dual-cubic-II}.
The remainder of \autoref{sec:discrete-on-the-interval}
is devoted to establishing some basic facts
which will be needed for formulating and solving the inverse problem
in \autoref{sec:inverse-spectral}.

\subsection{Rational approximations to the Weyl functions}
\label{sec:approx}

The Weyl functions $W(\lambda)$ and $Z(\lambda)$ are defined using entries of
the transition matrix $S(\lambda)$.
Next, we will see how entries of the matrices $T_j(\lambda)$
(partial products of $S(\lambda)$; see \eqref{eq:Tj})
produce rational approximations to the Weyl functions.
We have chosen here to work with the second column of $T_j(\lambda)$, since it seems to
be the most convenient for the inverse problem, but this choice is by no
means unique; many other similar approximation results could be derived.

\begin{theorem}
  \label{thm:QPR-approximation}
  Fix some $j$ with $1 \le j \le K$, write $T(\lambda) = T_j(\lambda)$ for simplicity,
  and consider the polynomials
  \begin{equation}
    \label{eq:QPR}
    Q(\lambda) = -T_{32}(\lambda),
    \quad
    P(\lambda) = T_{22}(\lambda),
    \quad
    R(\lambda) = T_{12}(\lambda).
  \end{equation}
  Then the following properties hold:
  \begin{equation}
    \label{eq:QPR-degrees}
    \deg Q = j,
    \qquad
    \deg P = j-1,
    \qquad
    \deg R = j-1,
  \end{equation}
  \begin{equation}
    \label{eq:QP-normalization}
    Q(0) = 0,
    \qquad
    P(0) = 1,
  \end{equation}
  and, as $\lambda \to \infty$,
  \begin{subequations}
    \label{eq:QPR-approximation}
    \begin{gather}
      W(\lambda) Q(\lambda) - P(\lambda) = \order{\frac{1}{\lambda}},
      \label{eq:W-approx}
      \\
      Z(\lambda) Q(\lambda) - R(\lambda) = \order{\frac{1}{\lambda}},
      \label{eq:Z-approx}
      \\
      R(\lambda) + P(\lambda) \twin W(-\lambda) + Q(\lambda) \twin Z(-\lambda) = \order{\frac{1}{\lambda^{j}}}.
      \label{eq:ZW-relation-approx}
    \end{gather}
  \end{subequations}
  (For $j=K$, the right-hand side of \eqref{eq:ZW-relation-approx} can be replaced by zero.)
\end{theorem}

\begin{proof}
  Equations \eqref{eq:QPR-degrees} and \eqref{eq:QP-normalization}
  were already proved in \autoref{prop:degrees-Tj}.
  With the notation used in that proof,
  the first column of the transition matrix $S(\lambda)$ is given by
  \begin{equation*}
    \begin{pmatrix}
      S_{11}(\lambda) \\ S_{21}(\lambda) \\ S_{31}(\lambda)
    \end{pmatrix}
    =
    \underbrace{L_{2K}(\lambda) \, t_{K}(\lambda) \dotsm t_{K+1-j}(\lambda)}_{= T(\lambda)}
    \,
    \underbrace{t_{K-j}(\lambda) \dotsm t_1(\lambda)
      \begin{smallpmatrix}
        1 \\ 0 \\ 0
      \end{smallpmatrix}
    }_{=
      \begin{smallpmatrix}
        a_1(\lambda) \\ a_2(\lambda) \\ a_3(\lambda)
      \end{smallpmatrix}
    }
    ,
  \end{equation*}
  where $a_1$, $a_2$, $a_3$ have degree at most $K-j$ in~$\lambda$.
  Hence,
  \begin{equation*}
    \begin{split}
      WQ - P
      &= -\frac{S_{21}}{S_{31}} \, (-T_{32}) - T_{22}
      = \frac{T_{32} S_{21} - T_{22} S_{31}}{S_{31}}
      \\
      &= \frac{T_{32} (T_{21}, T_{22}, T_{23})
        \begin{smallpmatrix}
          a_1 \\ a_2 \\ a_3
        \end{smallpmatrix}
        - T_{22} (T_{31}, T_{32}, T_{33})
        \begin{smallpmatrix}
          a_1 \\ a_2 \\ a_3
        \end{smallpmatrix}
      }{S_{31}} 
      \\
      &= \frac{-a_1
        \begin{vmatrix}
          T_{21} & T_{22} \\ T_{31} & T_{32}
        \end{vmatrix}
        + a_3
        \begin{vmatrix}
          T_{22} & T_{23} \\ T_{32} & T_{33}
        \end{vmatrix}
      }{S_{31}}
      = \frac{-a_1 (T^{-1})_{31} + a_3 (T^{-1})_{11}}{S_{31}},
    \end{split}
  \end{equation*}
  where the last step uses that $\det T(\lambda) = 1$
  (since each factor in $T$ has determinant one).
  By \eqref{eq:Tj-involution}, $T^{-1}(\lambda) = \J \twin T(-\lambda)^T \J$,
  where $\twin T(\lambda)$ is shorthand for $\twin T_j(\lambda)$ (defined by \eqref{eq:Tj-twin}).
  In particular, $(T^{-1})_{31}(\lambda) = \twin T_{31}(-\lambda)$ and $(T^{-1})_{11}(\lambda) = \twin T_{33}(-\lambda)$,
  so
  \begin{equation*}
    W(\lambda) Q(\lambda) - P(\lambda)
    = \frac{-a_1(\lambda) \, \twin T_{31}(-\lambda) + a_3(\lambda) \, \twin T_{33}(-\lambda)}{S_{31}(\lambda)}.
  \end{equation*}
  By \eqref{eq:degrees-S} and \eqref{eq:degrees-Tj-twin} we have
  \begin{equation*}
    \deg S_{31} = K+1,
    \qquad
    \deg \twin T_{31} = j,
    \qquad
    \deg \twin T_{33} = j-1,
  \end{equation*}
  which shows that $WQ - P = \order{\lambda^{(K-j)+j-(K+1)}} = \order{\lambda^{-1}}$ as $\lambda \to \infty$.

  The proof that $ZQ - R = \order{\lambda^{-1}}$ is entirely similar.

  To prove \eqref{eq:ZW-relation-approx}, we start from
  \begin{equation*}
    \begin{pmatrix}
      \twin S_{11}(\lambda) \\ \twin S_{21}(\lambda) \\ \twin S_{31}(\lambda)
    \end{pmatrix}
    = \twin T(\lambda)
    \begin{pmatrix}
      b_1(\lambda) \\ b_2(\lambda) \\ b_3(\lambda)
    \end{pmatrix}
    ,
  \end{equation*}
  where $b_1$, $b_2$, $b_3$ have degree at most $K-j$.
  Using again $\twin T(\lambda) = T(\lambda)^\sigma = \J T(-\lambda)^{-T} \J$, we obtain
  \begin{equation*}
    \begin{split}
      -b_2(-\lambda)
      &= -(0,1,0) \J T(\lambda)^T \J
      \begin{smallpmatrix}
        \twin S_{11}(-\lambda) \\ \twin S_{21}(-\lambda) \\ \twin S_{31}(-\lambda)
      \end{smallpmatrix}
      \\
      &= (0,1,0) T(\lambda)^T
      \begin{smallpmatrix}
        \twin S_{31}(\lambda) \\ -\twin S_{21}(\lambda) \\ \twin S_{11}(-\lambda)
      \end{smallpmatrix}
      \\
      &= \bigl( R(\lambda), P(\lambda), -Q(\lambda) \bigr)
      \begin{smallpmatrix}
        1 \\ \twin W(\lambda) \\ -\twin Z(-\lambda)
      \end{smallpmatrix}
      \twin S_{31}(\lambda).
    \end{split}
  \end{equation*}
  Since $\twin S_{31}$ has degree $K$ by \eqref{eq:degrees-S-twin},
  we find that $R(\lambda) + P(\lambda) \twin W(-\lambda) + Q(\lambda) \twin Z(-\lambda)
  = -b_2(-\lambda) / \twin S_{31}(\lambda) = \order{\lambda^{(K-j)-K}} = \order{\lambda^{-j}}$.
  (When $j=K$ we have $b_2(\lambda)=0$.)
\end{proof}

\begin{remark}
  Loosely speaking, the approximation conditions \eqref{eq:QPR-approximation} say that
  \begin{equation*}
    \frac{P(\lambda)}{Q(\lambda)} \approx W(\lambda),
    \quad
    \frac{R(\lambda)}{Q(\lambda)} \approx Z(\lambda),
  \end{equation*}
  and moreover these approximate Weyl functions satisfy
  \begin{equation*}
    \tfrac{R}{Q}(\lambda) + \tfrac{P}{Q}(\lambda) \twin W(-\lambda) + \twin Z(-\lambda) \approx 0
  \end{equation*}
  in place of the exact relation
  \begin{equation*}
    Z(\lambda) + W(\lambda) \twin W(-\lambda) + \twin Z(-\lambda) = 0
  \end{equation*}
  from \autoref{thm:Weyl-relation}.
  We say that the triple $(Q, P, R)$ provides a Type I Hermite--Pad{\'e} approximation of the
  functions $W$ and $Z$,
  and simultaneously a Type II Hermite--Pad{\'e} approximation of the
  functions $\twin W$ and $\twin Z$;
  see Section~5 in \cite{bertola-gekhtman-szmigielski:cauchy}.
\end{remark}

We will see in \autoref{sec:inverse-spectral} that for given Weyl functions
and a given order of approximation~$j$,
the properties in \autoref{thm:QPR-approximation}
are enough to determine the polynomials $Q$, $P$, $R$ uniquely.
This is the key to the inverse problem,
together with the following simple proposition.
We will need to consider $Q$, $P$, $R$
for different values of~$j$,
and we will write $Q_j$, $P_j$, $R_j$ to indicate this.
As a somewhat degenerate case not covered by \autoref{thm:QPR-approximation}
(the degree count \eqref{eq:QPR-degrees} fails),
we have
\begin{equation}
  \label{eq:Q0P0R0}
  Q_0(\lambda) = 0,
  \qquad
  P_0(\lambda) = 1,
  \qquad
  R_0(\lambda) = 0,
\end{equation}
coming from the second column of $T_0(\lambda) = L_{2K}(\lambda)$.

\begin{proposition}
  \label{prop:recover-even}
  If all $Q_j(\lambda)$ and $R_j(\lambda)$ are known,
  then the weights $h_j$ and their positions $y_{2j}$ can be determined:
  \begin{align}
    h_{j} &= R_{K-j+1}(0) - R_{K-j}(0),
    \label{eq:recover-evenmass}
    \\
    (1 - y_{2j}) h_j &= Q_{K-j+1}'(0) - Q_{K-j}'(0),
    \label{eq:recover-evenposition}
  \end{align}
  for $j=1,\dots,K$.
\end{proposition}

\begin{proof}
  By definition, $Q_j = -(T_j)_{32}$ and $R_j = (T_j)_{12}$,
  and \autoref{prop:degrees-Tj} says that
  $R_{K-j}(0) = \sum_{a > j} h_a$
  and
  $Q_{K-j}'(0) = \sum_{a>j} \sum_{k=2a}^{2K} h_a l_k$,
  for $0 \le j \le K-1$.
  The statement follows. (Note that $\sum_{k=2j}^{2K} l_k = 1 - y_{2j}$.)
\end{proof}

In order to access the weights $g_j$ and their positions $y_{2j-1}$
we will exploit the symmetry of the setup, via the adjoint problem;
see \autoref{sec:symmetry}.

\subsection{Adjoint Weyl functions}
\label{sec:adjoint-Weyl-discrete}

Recall the adjoint Weyl functions defined by \eqref{eq:WZ-adjoint}
and \eqref{eq:WZ-adjoint-twin},
\begin{equation*}
  W^* = -S_{32}/S_{31},
  \quad
  Z^* = -S_{33}/S_{31},
  \quad
  \twin W^* = -\twin S_{32}/\twin S_{31},
  \quad
  \twin Z^* = -\twin S_{33}/\twin S_{31},
\end{equation*}
which have the same denominators as the ordinary Weyl functions
\begin{equation*}
  W = -S_{21}/S_{31},
  \quad
  Z = -S_{11}/S_{31},
  \quad
  \twin W = -\twin S_{21}/\twin S_{31},
  \quad
  Z = -\twin S_{11}/\twin S_{31},
\end{equation*}
but different numerators.
Since the transition matrices $S(\lambda)$ and $\twin S(\lambda)$ both have the property that
the $(2,1)$ and $(3,2)$ entries have the same degree,
and the $(1,1)$ and $(3,3)$ entries have the same degree
(see \hyperref[prop:degrees-S]{Propositions \ref*{prop:degrees-S}} and~\ref{prop:degrees-S-twin}),
the adjoint Weyl functions will have partial fraction decompositions of exactly the same form as
their non-starred counterparts (cf. \autoref{thm:weyl-parfrac}),
with the same poles but different residues:
\begin{subequations}
  \label{eq:weyl-star-parfrac}
  \begin{align}
    \label{eq:W-star-parfrac}
    W^*(\lambda) &= \sum_{i=1}^K \frac{a_i^*}{\lambda - \lambda_i},
    \\
    \label{eq:W-star-twin-parfrac}
    \twin W^*(\lambda) &= -b_{\infty}^* + \sum_{j=1}^{K-1} \frac{b_j^*}{\lambda - \mu_j},
    \\
    \label{eq:Z-star-parfrac}
    Z^*(\lambda) &= \frac{1}{2\lambda} + \sum_{i=1}^K \frac{c_i^*}{\lambda - \lambda_i},
    \\
    \label{eq:Z-star-twin-parfrac}
    \twin Z^*(\lambda) &= \frac{1}{2\lambda} + \sum_{j=1}^{K-1} \frac{d_j^*}{\lambda - \mu_j}.
  \end{align}
\end{subequations}
Just like in the proof of \autoref{thm:weyl-parfrac},
it follows from \autoref{thm:adjoint-Weyl-relation} that
\begin{equation}
  \label{eq:residue-relation-star}
  c_i^* = a_i^* b_{\infty}^* + \sum_{j=1}^{K-1} \frac{a_i^* b_j^*}{\lambda_i + \mu_j},
  \qquad
  d_j^* = \sum_{i=1}^K \frac{a_i^* b_j^*}{\lambda_i + \mu_j},
\end{equation}
so that $Z^*$ and $\twin Z^*$ are determined by $W^*$ and~$\twin W^*$.
Moreover, there is the following connection between the ordinary Weyl functions and their adjoints.

\begin{theorem}
  \label{thm:residue-relation-W-Wstar}
  Assume that $K \ge 2$. The residues of $W$ and $W^*$ satisfy
  \begin{equation}
    \label{eq:residue-relation-W-Wstar}
    a_k a_k^*
    = \frac{\displaystyle \lambda_k \prod_{j=1}^{K-1} \left( 1 + \frac{\lambda_k}{\mu_j} \right) }{\displaystyle 2 \prod_{\substack{i=1 \\ i \neq k}}^{K} \left( 1 - \frac{\lambda_k}{\lambda_i} \right)^2}
    , \qquad
    k = 1, \dots, K
    .
  \end{equation}
  Likewise, the residues of $\twin W$ and $\twin W^*$ satisfy
  \begin{equation}
    \label{eq:residue-relation-W-Wstar-twin}
    b_k b_k^*
    = \frac{\displaystyle \mu_k \prod_{i=1}^{K} \left( 1 + \frac{\mu_k}{\lambda_i} \right) }{\displaystyle 2 \prod_{\substack{j=1 \\ j \neq k}}^{K-1} \left( 1 - \frac{\mu_k}{\mu_j} \right)^2}
    , \qquad
    k = 1, \dots, K-1
    .
  \end{equation}
  (The empty product appearing when $K=2$ should be omitted;
  thus, $b_1 b_1^* = \frac12 \mu_1 \prod_{i=1}^2 (1+\mu_1/\lambda_i)$ in this case.)
  Moreover,
  \begin{equation}
    \label{eq:b-infty-b-infty-star-product}
    b_{\infty} b^*_{\infty}
    = \frac{l_1 l_3 \dotsm l_{2K-1}}{l_0 l_2 l_4 \dotsm l_{2K}}
    \times \biggl( \prod_{j=1}^{K-1} \mu_j \biggr) \biggm/ \biggl( \prod_{i=1}^{K} \lambda_i \biggr)
    .
  \end{equation}
\end{theorem}

\begin{proof}
  We first prove \eqref{eq:residue-relation-W-Wstar}.
  From \eqref{eq:S-cofactors} we have $\twin S_{31}(-\lambda) = S_{21}(\lambda) S_{32}(\lambda) - S_{22}(\lambda) S_{31}(\lambda)$.
  Evaluation at $\lambda = \lambda_k$ kills $S_{31}$, so
  \begin{equation*}
    \twin S_{31}(-\lambda_k) = S_{21}(\lambda_k) S_{32}(\lambda_k).
  \end{equation*}
  Since the poles of $W$ and $W^*$ are simple, the residues are given by
  $a_k = -S_{21}(\lambda_k)/S_{31}'(\lambda_k)$ and
  $a_k^* = -S_{32}(\lambda_k)/S_{31}'(\lambda_k)$.
  Multiplication yields
  \begin{equation*}
    a_k a_k^* = \frac{S_{21}(\lambda_k) S_{32}(\lambda_k)}{S_{31}'(\lambda_k)^2}
    = \frac{\twin S_{31}(-\lambda_k)}{S_{31}'(\lambda_k)^2},
  \end{equation*}
  and insertion of the expressions for $S_{31}$ and $\twin S_{31}$ from 
  \eqref{eq:S31-factorized} finishes the job.

  The proof of equation \eqref{eq:residue-relation-W-Wstar-twin} is similar.

  As for \eqref{eq:b-infty-b-infty-star-product},
  we saw in \eqref{eq:formula-for-b-infty} that
  \begin{equation*}
    b_{\infty} = \frac{h_K l_{2K-1}}{l_{2K}+l_{2K-1}}
    .
  \end{equation*}
  In the same way, or by using the symmetry transformation
  \eqref{eq:symmetry-substitution} described in the next section,
  one shows that
  \begin{equation*}
    b^*_{\infty} = \frac{g_1 l_1}{l_0 + l_1}
    .
  \end{equation*}
  Combining $S_{31}(\lambda) = -2\lambda \prod_{i=1}^K (1 - \lambda/\lambda_i)$ with the expression
  \eqref{eq:highest-term-S31} for the highest coefficient of $S_{31}$ yields
  \begin{equation*}
    \prod_{i=1}^{K} \lambda_i = \frac12 \left( \prod_{m=0}^K l_{2m} \right) \left( \prod_{a=1}^K g_a h_a \right)
    ,
  \end{equation*}
  and similarly we find by comparing
  $\twin S_{31}(\lambda) = -2\lambda \prod_{j=1}^{K-1} (1 - \lambda/\mu_j)$
  to \eqref{eq:highest-term-S31-twin} that
  \begin{equation*}
    \prod_{j=1}^{K-1} \mu_j = \frac12 (l_{2K}+l_{2K-1}) \left( \prod_{m=2}^{K-1} l_{2m-1} \right) (l_{0}+l_{1}) \left( \prod_{a=1}^{K-1} g_{a+1} h_a \right)
    .
  \end{equation*}
  Equation \eqref{eq:b-infty-b-infty-star-product} follows.
\end{proof}

\begin{remark}
  When $K=1$, we have
  \begin{equation}
    a_1 a_1^* = \frac{2}{\lambda_1}
  \end{equation}
  as shown in \eqref{eq:residue-relation-W-Wstar-K1},
  while \eqref{eq:b-infty-b-infty-star-product} breaks down for the same reason that
  \eqref{eq:formula-for-b-infty} did;
  by
  \eqref{eq:lambda1-a1-K1},
  \eqref{eq:binfty-K1}
  and
  \eqref{eq:binftystar-K1},
  we have instead
  \begin{equation}
    \label{eq:b-infty-b-infty-star-product-K1}
    b_{\infty} b_{\infty}^* = \frac{(l_0+l_1)(l_1+l_2)}{2 l_0 l_2 \lambda_1}
  \end{equation}
  in this case.
\end{remark}

\begin{remark}
  \autoref{thm:residue-relation-W-Wstar} shows that $W$ and $\twin W$ together determine $W^*$,
  since $a_1^*, \dots, a_K^*$ can be computed from \eqref{eq:residue-relation-W-Wstar} if one knows
  $\{ a_k, b_k, b_{\infty}, \lambda_i, \mu_j \}$.
  But they only \emph{almost} determine $\twin W^*$;
  the residues $b_1^*, \dots, b_{K-1}^*$ can be computed from \eqref{eq:residue-relation-W-Wstar-twin},
  but the constant $b_{\infty}^*$ is not determined!
  This turns out to be highly significant for the inverse spectral problem:
  the Weyl functions $W$ and $\twin W$ don't contain enough information to
  recover the first weight $g_1$ and its position $y_1$;
  for this we need to know the value of $b_{\infty}^*$ as well.
\end{remark}

We can now prove the positivity of the residues $a_i$ and $b_j$
in \autoref{thm:weyl-parfrac}.
(The notation introduced in this proof will not be used elsewhere,
and is omitted from the index of notation in \autoref{app:notation}.)

\begin{proof}[Proof of \autoref{thm:weyl-parfrac}, continued]
  We consider the residues $\{ a_i \}_{i=1}^K$ first.
  For $K=1$ we have $S_{21}(\lambda) = -g_1 l_0 \lambda$ and
  $S_{31}(\lambda) = -2\lambda + g_1 h_1 l_0 l_2 \lambda^2$,
  so that
  \begin{equation*}
    W(\lambda) = -\frac{S_{21}(\lambda)}{S_{31}(\lambda)}
    = \frac{\frac{1}{h_1 l_2}}{\lambda - \frac{2}{g_1 h_1 l_0 l_2}}
    ;
  \end{equation*}
  hence $a_1 = \frac{1}{h_1 l_2} > 0$.
  We now proceed by induction on~$K$.
  Suppose that the residues~$a_i$ are positive when $K=m-1$,
  and consider the case $K=m \ge 2$.
  Because of \eqref{eq:residue-relation-W-Wstar}, no $a_i$ can ever be zero
  as long as all masses are positive,
  and therefore it is sufficient to verify that all $a_i$ are positive
  when the last pair of masses are given by $g_m = h_m = \epsilon$
  with $\epsilon > 0$ small;
  since the residues depend continuously on the masses, 
  they will keep their signs as $g_m$ and~$h_m$ are allowed to vary arbitrarily
  over all positive values.
  From \eqref{eq:S-discrete} we get
  \begin{equation*}
    \begin{pmatrix}
      S_{11}(\lambda,\epsilon) \\ S_{21}(\lambda,\epsilon) \\ S_{31}(\lambda,\epsilon)
    \end{pmatrix}
    =
    L_{2m}(\lambda) \jumpmatrix{\epsilon}{0} L_{2m-1}(\lambda) \jumpmatrix{0}{\epsilon}  L_{2m-2}(\lambda) 
    \dotsm
    \jumpmatrix{h_1}{0} L_{1}(\lambda) \jumpmatrix{0}{g_1} L_{0}(\lambda)
    \begin{pmatrix}
      1 \\ 0 \\ 0
    \end{pmatrix}
    ,
  \end{equation*}
  where we consider all positions and all masses except $g_m=h_m=\epsilon$ as fixed,
  and treat the $S_{ij}(\lambda,\epsilon)$ as polynomials in two variables.
  The spectral data defined by these polynomials will then of course also be
  considered as functions of~$\epsilon$: $\{ \lambda_i(\epsilon), a_i(\epsilon) \}_{i=1}^m$.
  (As we will soon see, the largest eigenvalue $\lambda_m(\epsilon)$ has a pole
  of order~$2$ at $\epsilon=0$, while the other eigenvalues are analytic functions of~$\epsilon$.)
  The first four factors in the product above are
  \begin{equation*}
    L_{2m}(\lambda) \jumpmatrix{\epsilon}{0} L_{2m-1}(\lambda) \jumpmatrix{0}{\epsilon}
    =
    \begin{pmatrix}
      1 & \epsilon & \epsilon^2 \\
      0 & 1 & \epsilon \\
      -(l_{2m} + l_{2m-1}) \lambda & - \epsilon l_{2m} \lambda & 1 - \epsilon^2 l_{2m} \lambda
    \end{pmatrix}
    .
  \end{equation*}
  We denote the product of the remaining factors by
  $(s_{11}(\lambda),s_{21}(\lambda),s_{31}(\lambda))^T$;
  these polynomials have the same form as $S_{11}$, $S_{21}$ and $S_{23}$
  (see \autoref{prop:degrees-S}),
  but with $m-1$ instead of~$m$,
  so their degrees are one step lower, and they only depend on
  $\{ g_k, h_k \}_{k=1}^{m-1}$ and $\{ l_k \}_{k=0}^{2m-2}$,
  not on $l_{2m-1}$, $l_{2m}$ and $g_m=h_m=\epsilon$.
  We thus have
  \begin{equation}
    \label{eq:Sij-split-off-epsilon}
    \begin{split}
      \begin{pmatrix}
        S_{11}(\lambda,\epsilon) \\ S_{21}(\lambda,\epsilon) \\ S_{31}(\lambda,\epsilon)
      \end{pmatrix}
      &=
      \begin{pmatrix}
        1 & \epsilon & \epsilon^2 \\
        0 & 1 & \epsilon \\
        -(l_{2m} + l_{2m-1}) \lambda & - \epsilon l_{2m} \lambda & 1 - \epsilon^2 l_{2m} \lambda
      \end{pmatrix}
      \begin{pmatrix}
        s_{11}(\lambda) \\ s_{21}(\lambda) \\ s_{31}(\lambda)
      \end{pmatrix}
      \\ &=
      \begin{pmatrix}
        S_{11}(\lambda,0) \\ S_{21}(\lambda,0) \\ S_{31}(\lambda,0)
      \end{pmatrix}
      +
      \begin{pmatrix}
        0 & \epsilon & \epsilon^2 \\
        0 & 0 & \epsilon \\
        0 & - \epsilon l_{2m} \lambda & - \epsilon^2 l_{2m} \lambda
      \end{pmatrix}
      \begin{pmatrix}
        s_{11}(\lambda) \\ s_{21}(\lambda) \\ s_{31}(\lambda)
      \end{pmatrix}
      .
    \end{split}
  \end{equation}
  The polynomials $S_{ij}(\lambda,0)$ define the spectral data for the case $K=m-1$
  (since the final pair of masses is absent when $\epsilon=0$); in particular we know from
  \autoref{thm:simple-spectra} that
  $S_{31}(\lambda,0)$ has a zero at $\lambda=0$, and that the other $m-1$
  zeros are positive and simple.
  If $\lambda = \lambda_i \neq 0$ is one of these other zeros, then at the point
  $(\lambda,\epsilon) = (\lambda_i,0)$ we therefore have $S_{31}=0$ and
  $\partial S_{31}/\partial \lambda \neq 0$,
  so by the Implicit Function Theorem there is an analytic function $\lambda_i(\epsilon)$,
  defined around $\epsilon=0$,
  such that $\lambda_i(0)=\lambda_i$ and $S_{31}(\lambda_i(\epsilon),\epsilon) = 0$.
  It follows that for $i=1,\dots,m-1$, the residue
  \begin{equation*}
    a_i(\epsilon) = \res_{\lambda=\lambda_i(\epsilon)} W(\lambda,\epsilon)
    = - \frac{S_{21}(\lambda_i(\epsilon),\epsilon)}{\dfrac{\partial S_{31}}{\partial \lambda}(\lambda_i(\epsilon),\epsilon)}
  \end{equation*}
  depends analytically on $\epsilon$ too, and it is therefore positive for small $\epsilon > 0$,
  since it is positive for $\epsilon = 0$ by the induction hypothesis.
  This settles part of our claim.

  It remains to show that the last residue $a_m(\epsilon)$ is positive.
  As a first step, we show that $\lambda_m(\epsilon)$ has a pole of order~$2$ at $\epsilon=0$.
  For convenience, let
  \begin{equation*}
    f(\lambda,\epsilon) = \frac{S_{31}(\lambda,\epsilon)}{\lambda}
    ;
  \end{equation*}
  this is a polynomial of degree~$m$ in~$\lambda$,
  and $\lambda_m(\epsilon)$ is the largest root of the equation $f(\lambda,\epsilon)=0$.
  From \eqref{eq:Sij-split-off-epsilon} we have
  \begin{equation*}
    f(\lambda,\epsilon) =
    f(\lambda,0) - l_{2m} \Bigl( \epsilon s_{21}(\lambda) + \epsilon^2 s_{31}(\lambda) \Bigr)
    .
  \end{equation*}
  Using \autoref{prop:degrees-S}, we see that the leading terms of
  $f(\lambda,0) = S_{31}(\lambda,0)/\lambda$ and $l_{2m} s_{31}(\lambda)$ are
  $(-1)^m C_1 \lambda^{m-1}$ and $(-1)^m C_2 \lambda^m$, respectively, with
  \begin{gather*}
    C_1 = \left( \prod_{r=0}^{m-2} l_{2r} \right) (l_{2m-2} + l_{2m-1} + l_{2m}) \left( \prod_{a=1}^{m-1} g_a h_a \right) > 0
    ,\\
    C_2 = \left( \prod_{r=0}^{m} l_{2r} \right) \left( \prod_{a=1}^{m-1} g_a h_a \right) > 0
    .
  \end{gather*}
  (The precise form of these constants is not very important, only their positivity.)
  Moreover, $s_{21}(\lambda)$ has degree~$m-1$.
  Thus
  \begin{equation*}
    \begin{split}
      f(\lambda,\epsilon)
      &=
      f(\lambda,0)
      - l_{2m} \Bigl( \epsilon s_{21}(\lambda) + \epsilon^2 s_{31}(\lambda) \Bigr)
      \\ &=
      (-1)^{m+1} C_2 \epsilon^2 \lambda^m + p(\lambda,\epsilon)
      ,
    \end{split}
  \end{equation*}
  with a polynomial $p(\lambda,\epsilon)$ of degree $m-1$ in $\lambda$.
  Since $p(\lambda,0) = f(\lambda,0)$ has leading term
  $(-1)^{m} C_1 \lambda^{m-1}$, we see that
  \begin{equation*}
    \epsilon^{2m-2} p(\kappa \epsilon^{-2},\epsilon)
    = (-1)^{m} C_1 \kappa^{m-1} + (\text{terms containing $\epsilon$}).
  \end{equation*}
  Hence, the equation $f(\lambda,\epsilon)=0$,
  of which $\lambda_m(\epsilon)$ is the largest root,
  can be written in terms of the new variable
  $\kappa = \lambda \, \epsilon^2$
  as
  \begin{equation*}
    \begin{split}
      0 &=
      (-1)^{m+1} \epsilon^{2m-2} f(\lambda,\epsilon)
      \\ &=
      C_2 \epsilon^{2m} \lambda^m + \epsilon^{2m-2} (-1)^{m+1} p(\lambda,\epsilon)
      \\ &=
      C_2 \kappa^m +  \epsilon^{2m-2} (-1)^{m+1} p(\kappa \epsilon^{-2},\epsilon)
      \\ &=
      C_2 \kappa^m - C_1 \kappa^{m-1} + \epsilon \, q(\kappa,\epsilon)
      ,
    \end{split}
  \end{equation*}
  for some two-variable polynomial $q(\kappa,\epsilon)$.
  As before, the Implicit Function Theorem shows that this equation has an analytic solution
  $\kappa(\epsilon)$ with $\kappa(0)=C_1/C_2$,
  which corresponds to a meromorphic zero of $f(\lambda,\epsilon)$
  with a pole of order~$2$, as claimed:
  \begin{equation*}
    \lambda_m(\epsilon) = \frac{\kappa(\epsilon)}{\epsilon^2}
    = \frac{C_1/C_2 + \order{\epsilon}}{\epsilon^2}
    .
  \end{equation*}
  Finally, the corresponding residue is
  \begin{equation*}
    a_m(\epsilon) = \res_{\lambda=\lambda_m(\epsilon)} W(\lambda,\epsilon)
    = - \frac{S_{21}(\lambda_m(\epsilon),\epsilon)}{\dfrac{\partial S_{31}}{\partial \lambda}(\lambda_m(\epsilon),\epsilon)}
    .
  \end{equation*}
  The derivative of the polynomial $S_{31}$ at its largest zero has the same sign as
  the leading term of $S_{31}$, namely $(-1)^{m+1}$.
  As for the sign of $S_{21}$,
  we have from \eqref{eq:Sij-split-off-epsilon} that
  \begin{equation*}
    S_{21}(\lambda,\epsilon) = S_{21}(\lambda,0) + \epsilon s_{31}(\lambda)
    ,
  \end{equation*}
  where $S_{21}(\lambda,0)$ and $s_{31}(\lambda)$ have degrees $m-1$ and~$m$,
  respectively.
  When this is evaluated at $\lambda = \lambda_m(\epsilon) \sim \frac{C_1}{C_2}\epsilon^{-2}$,
  the two terms on the right-hand side are of order $\epsilon^{2m-2}$ and $\epsilon^{2m-1}$,
  respectively, so the dominant behavior as $\epsilon \to 0^+$
  comes from the leading term of $s_{31}(\lambda)$:
  \begin{equation*}
    S_{21}(\lambda_m(\epsilon),\epsilon) \sim \epsilon (-1)^m \frac{C_2}{l_{2m}} \left( \frac{C_1/C_2}{\epsilon^2} \right)^m.
  \end{equation*}
  In particular, the sign of $S_{21}(\lambda_m(\epsilon),\epsilon)$ is $(-1)^m$,
  and it follows that $a_m(\epsilon) > 0$, which is what we wanted to show.
  This concludes the proof of positivity for the residues~$a_i$.

  The proof for the residues $\{ b_j \}_{j=1}^{K-1}$ is similar.
  In the base case $K=1$ there is nothing to show.
  Assume that they are positive for $K=m-1$, and consider the case $K=m \ge 2$.
  We have from \eqref{eq:S-twin-discrete}
  \begin{equation*}
    \begin{pmatrix}
      \twin S_{11}(\lambda,\epsilon) \\ \twin S_{21}(\lambda,\epsilon) \\ \twin S_{31}(\lambda,\epsilon)
    \end{pmatrix}
    =
    L_{2m}(\lambda) \jumpmatrix{0}{\epsilon} L_{2m-1}(\lambda) \jumpmatrix{\epsilon}{0}  L_{2m-2}(\lambda) 
    \dotsm
    \jumpmatrix{0}{h_1} L_{1}(\lambda) \jumpmatrix{g_1}{0} L_{0}(\lambda)
    \begin{pmatrix}
      1 \\ 0 \\ 0
    \end{pmatrix}
    .
  \end{equation*}
  Splitting off the first four factors
  \begin{equation*}
    L_{2m}(\lambda) \jumpmatrix{0}{\epsilon} L_{2m-1}(\lambda) \jumpmatrix{\epsilon}{0}
    =
    \begin{pmatrix}
      1 & \epsilon & 0 \\
      - \epsilon l_{2m-1} \lambda & 1 - \epsilon^2 l_{2m-1} \lambda & \epsilon \\
      -(l_{2m} + l_{2m-1}) \lambda & - \epsilon (l_{2m} + l_{2m-1}) \lambda & 1
    \end{pmatrix}
    ,
  \end{equation*}
  we obtain
  \begin{equation}
    \label{eq:Sij-twin-split-off-epsilon}
    \begin{pmatrix}
      \twin S_{11}(\lambda,\epsilon) \\ \twin S_{21}(\lambda,\epsilon) \\ \twin S_{31}(\lambda,\epsilon)
    \end{pmatrix}
    =
    \begin{pmatrix}
      \twin S_{11}(\lambda,0) \\ \twin S_{21}(\lambda,0) \\ \twin S_{31}(\lambda,0)
    \end{pmatrix}
    +
    \begin{pmatrix}
      0 & \epsilon & 0 \\
      - \epsilon l_{2m-1} \lambda & - \epsilon^2 l_{2m-1} \lambda & \epsilon \\
      0 & - \epsilon (l_{2m} + l_{2m-1}) \lambda & 0
    \end{pmatrix}
    \begin{pmatrix}
      \twin s_{11}(\lambda) \\ \twin s_{21}(\lambda) \\ \twin s_{31}(\lambda)
    \end{pmatrix}
    ,
  \end{equation}
  where the degrees on the left-hand side are $(m-1,m,m)$,
  while both $3 \times 1$ matrices appearing on the right-hand side
  have degrees $(m-2,m-1,m-1)$ (cf. \autoref{prop:degrees-S-twin}).
  The eigenvalues $\{ \mu_j(\epsilon) \}_{j=1}^{m-1}$ are the zeros of the polynomial
  \begin{equation*}
    \twin f(\lambda,\epsilon) = \frac{\twin S_{31}(\lambda,\epsilon)}{\lambda}
    = \frac{\twin S_{31}(\lambda,0)}{\lambda} - \epsilon (l_{2m} + l_{2m-1}) \twin s_{21}(\lambda)
    .
  \end{equation*}
  As above,
  it follows easily that $\{ \mu_j(\epsilon) \}_{j=1}^{m-2}$ are analytic,
  and that the corresponding residues $\{ b_j(\epsilon) \}_{j=1}^{m-2}$ are positive.
  The largest zero $\mu_{m-1}(\epsilon)$ has a pole of order~$1$ at $\epsilon=0$,
  as we now show.
  By \autoref{prop:degrees-S-twin}, the leading terms of
  $\twin S_{31}(\lambda,0)$ and $(l_{2m} + l_{2m-1}) \twin s_{21}(\lambda)$ are
  $(-1)^{m-1} \twin C_1 \lambda^{m-1}$ and $(-1)^{m-1} \twin C_2 \lambda^{m-1}$, respectively,
  with some positive constants $\twin C_1$ and~$\twin C_2$.
  (For the record, these constants are
  \begin{gather*}
    \twin C_1 = \left( \sum_{a = 2m-3}^{2m} l_a \right) \left( \prod_{r=2}^{m-2} l_{2r-1} \right) (l_{0}+l_{1}) \left( \prod_{a=1}^{m-2} g_{a+1} h_a \right) > 0
    ,\\
    \twin C_2 = (l_{2m} + l_{2m-1}) \left( \prod_{r=2}^{m-1} l_{2r-1} \right) (l_{0}+l_{1}) \, h_{m-1} \left( \prod_{a=1}^{m-2} g_{a+1} h_a \right) > 0
    .
  \end{gather*}
  Special case: $\twin C_1 = 2$ if $m=2$. The empty product $\prod_{r=2}^{1} l_{2r-1}$
  is omitted in $\twin C_1$ when $m=3$ and in $\twin C_2$ when $m=2$.)
  Hence,
  \begin{equation*}
    \begin{split}
      \twin f(\lambda,\epsilon)
      &=
      \twin f(\lambda,0)
      - \epsilon (l_{2m} + l_{2m-1}) \twin s_{21}(\lambda)
      \\ &=
      (-1)^{m} \twin C_2 \epsilon \lambda^{m-1} + \twin p(\lambda,\epsilon)
      ,
    \end{split}
  \end{equation*}
  with a polynomial $\twin p(\lambda,\epsilon)$ of degree $m-2$ in $\lambda$,
  such that $\twin p(\lambda,0) = \twin f(\lambda,0)$ has leading term
  $(-1)^{m-1} \twin C_1 \lambda^{m-2}$,
  so that
  \begin{equation*}
    \epsilon^{m-2} \twin p(\twin\kappa \epsilon^{-1},\epsilon)
    = (-1)^{m-1} \twin C_1 \twin\kappa^{m-1} + (\text{terms containing $\epsilon$}).
  \end{equation*}
  The equation $\twin f(\lambda,\epsilon)=0$,
  of which $\mu_{m-1}(\epsilon)$ is the largest root,
  can therefore be written in terms of the new variable
  $\twin\kappa = \lambda \, \epsilon$
  as
  \begin{equation*}
    \begin{split}
      0 &=
      (-1)^{m} \epsilon^{m-2} \twin f(\lambda,\epsilon)
      \\ &=
      \twin C_2 \epsilon^{m-1} \lambda^{m-1} + \epsilon^{m-2} (-1)^{m} \twin p(\lambda,\epsilon)
      \\ &=
      \twin C_2 \twin\kappa^m +  \epsilon^{m-2} (-1)^{m} \twin p(\twin\kappa \epsilon^{-1},\epsilon)
      \\ &=
      \twin C_2 \twin\kappa^m - \twin C_1 \twin\kappa^{m-1} + \epsilon \, \twin q(\twin\kappa,\epsilon)
      ,
    \end{split}
  \end{equation*}
  for some two-variable polynomial $\twin q(\twin\kappa,\epsilon)$.
  The Implicit Function Theorem gives an analytic function $\twin\kappa(\epsilon)$
  with $\twin\kappa(0) = \twin C_1 / \twin C_2$, and, as claimed,
  \begin{equation*}
    \mu_{j-1}(\epsilon) = \frac{\twin\kappa(\epsilon)}{\epsilon}
    = \frac{\twin C_1 / \twin C_2 + \order{\epsilon}}{\epsilon}
    .
  \end{equation*}
  The corresponding residue is
  \begin{equation*}
    b_{m-1}(\epsilon) = \res_{\lambda=\mu_{m-1}(\epsilon)} \twin W(\lambda,\epsilon)
    = - \frac{\twin S_{21}(\mu_{m-1}(\epsilon),\epsilon)}{\dfrac{\partial \twin S_{31}}{\partial \lambda}(\mu_{m-1}(\epsilon),\epsilon)}
    .
  \end{equation*}
  The leading term of $S_{31}$ determines the sign of
  the derivative $\partial S_{31}/\partial \lambda$ at the largest zero, namely
  $(-1)^{m}$.
  From \eqref{eq:Sij-twin-split-off-epsilon},
  \begin{equation*}
    \twin S_{21}(\lambda,\epsilon)
    = \twin S_{21}(\lambda,0)
    - \epsilon l_{2m-1} \lambda \twin s_{11}(\lambda)
    - \epsilon^2 l_{2m-1} \lambda \twin s_{21}(\lambda)
    + \epsilon \twin s_{31}(\lambda)
    ,
  \end{equation*}
  and when evaluating this at
  $\lambda = \mu_{m-1}(\epsilon) \sim \frac{\twin C_1}{\twin C_2}\epsilon^{-1} $,
  the last three terms on the right-hand side are of order $\epsilon^{2-m}$,
  so the contribution of order $\epsilon^{1-m}$ from the first term $\twin S_{21}(\lambda,0)$
  is the dominant one as $\epsilon \to 0^+$, and it has the sign $(-1)^{m-1}$.
  It follows that $b_{m-1}(\epsilon) > 0$, and the proof is complete.
\end{proof}

\subsection{Symmetry}
\label{sec:symmetry}

For solutions of the differential equation \eqref{eq:dual-cubic-I},
$\frac{\partial \Phi}{\partial y} = \coeffmatrix(\lambda) \Phi$,
the transition matrix $S(\lambda)$ propagates initial values at the left endpoint to final values
at the right endpoint:
\begin{equation*}
  \Phi(+1) = S(\lambda) \Phi(-1).
\end{equation*}
The transition matrix depends of course not only on $\lambda$ but also on $g(y)$ and $h(y)$,
which in our discrete setup means the point masses
$g_j$ and $h_j$ interlacingly positioned at the sites $y_k$ with $l_k = y_{k+1} - y_k$;
let us write
\begin{equation*}
  \begin{split}
    S(\lambda)
    &= L_{2K}(\lambda)
    \jumpmatrix{h_K}{0} L_{2K-1}(\lambda) \jumpmatrix{0}{g_K}  L_{2K-2}(\lambda) 
    \dotsm
    \jumpmatrix{h_1}{0} L_{1}(\lambda) \jumpmatrix{0}{g_1} L_{0}(\lambda)
    \\[1ex]
    &= S(\lambda; l_0,\dots,l_{2K}; g_1, h_1, \dots, g_K, h_K)
  \end{split}
\end{equation*}
to indicate this.

For the adjoint equation \eqref{eq:adjoint-ODE},
$\frac{\partial \Omega}{\partial y} = \twin\coeffmatrix(-\lambda) \Omega$,
we saw in \autoref{prop:adjoint-transition-matrix} that the matrix
$\twin S(-\lambda)^{-1} = \J S(\lambda)^T \J$
propagates values in the opposite direction, from initial values at the right endpoint to final values
at the left endpoint.
If we denote this matrix by $S^*(\lambda)$, we thus have
\begin{equation*}
  \Omega(-1) = S^*(\lambda) \Omega(+1).
\end{equation*}
When going from right to left, one encounters the point masses in the opposite order compared to when going
from left to right,
and the following theorem shows that the solution $\Omega(y)$ reacts just like $\Phi(y)$ does when encountering
a mass, except for a difference in sign.

\begin{theorem}
  \label{thm:symmetry}
  The adjoint transition matrix is given by
  \begin{equation}
    \label{eq:symmetry}
    S^*(\lambda) = S(\lambda; l_{2K}, \dots, l_0; -h_K, -g_K, \dots, -h_1, -g_1).
  \end{equation}
  (And similarly with tildes for the twin problems.)
\end{theorem}

\begin{proof}
  Use $J = J^T = J^{-1}$ together with
  $\J L_k(\lambda)^T \J = L_k(\lambda)$ and $\J \jumpmatrix{x}{y}^T \J = \jumpmatrix{-y}{-x}$
  to obtain
  \begin{equation*}
    \begin{split}
      S^*(\lambda)
      &= \J S(\lambda)^T \J
      \\
      &= \J \biggl( L_{2K}(\lambda) \jumpmatrix{h_K}{0} \dotsm \jumpmatrix{0}{g_1} L_0(\lambda) \biggr)^T \J
      \\
      &= \biggl( \J L_0(\lambda)^T \J \biggr) \bigg( \J \jumpmatrix{0}{g_1}^T \J \biggr) \dotsm \bigg( \J \jumpmatrix{h_K}{0}^T \J \biggr) \bigg( \J L_{2K}(\lambda)^T \J \biggr)
      \\
      &= L_0(\lambda) \jumpmatrix{-g_1}{0} \dotsm \jumpmatrix{0}{-h_K} L_{2K}(\lambda).
    \end{split}
  \end{equation*}
\end{proof}

\begin{remark}
  \label{rem:S-star}
  The adjoint Weyl functions $W^*$ and $Z^*$ are defined from the first column in~$S^* = \J S^T \J$,
  \begin{equation*}
    (S^*_{11}, S^*_{11}, S^*_{11})^T = (S_{33},-S_{32},S_{31})^T,
  \end{equation*}
  in almost the same way as $W$ and $Z$ are defined from the first column in $S$,
  but there is a slight sign difference in $W^*$ since we have defined all Weyl functions
  so that they will have positive residues:
  \begin{equation*}
    W = -S_{21}/S_{31},
    \qquad
    Z = -S_{11}/S_{31},
  \end{equation*}
  but
  \begin{equation*}
    W^* = -S_{32}/S_{31} = +S^*_{21}/S^*_{31},
    \qquad
    Z^* = -S_{33}/S_{31} = -S^*_{11}/S^*_{31}.
  \end{equation*}
  As a consequence, we see for example that if
  \begin{equation*}
    a_k = F(l_0,\dots,l_{2K}; g_1, h_1, \dots, g_K, h_K)
  \end{equation*}
  indicates how the residue $a_k$ in $W$ depends on the configuration of the masses,
  then
  \begin{equation*}
    -a_k^* = F(l_{2K}, \dots, l_0; -h_K, -g_K, \dots, -h_1, -g_1),
  \end{equation*}
  with the same function $F$, will determine the corresponding residue in $W^*$.
\end{remark}

\begin{remark}
  \label{rem:epiphany}
  In \autoref{sec:recovery} we will use \autoref{thm:QPR-approximation} and \autoref{prop:recover-even}
  to derive formulas for recovering the weights $h_j$ and their positions $y_{2j}$ from the Weyl functions $W$ and $\twin W$.
  Because of the symmetry properties described here, the same formulas can then be used to recover
  the weights $g_j$ and their positions $y_{2j-1}$ from the adjoint Weyl functions $W^*$ and $\twin W^*$,
  by substituting
  \begin{equation}
    \label{eq:symmetry-substitution}
    \begin{aligned}
      a_i &\mapsto -a_i^*,
      \\
      b_j &\mapsto -b_j^*,
      \\
      b_{\infty} &\mapsto -b_{\infty}^*,
    \end{aligned}
    \qquad
    \begin{aligned}
      l_k &\mapsto l_{2K-k},
      \\
      g_j &\mapsto -h_{K+1-j},
      \\
      h_j &\mapsto -g_{K+1-j}.
    \end{aligned}
  \end{equation}
  Note that $1 - y_m = \sum_{k=m}^{2K} l_k$
  is to be replaced by
  $\sum_{k=m}^{2K} l_{2K-k} = \sum_{s=0}^{2K-m} l_{s} = 1 + y_{2K+1-m}$.
\end{remark}

\section{The inverse spectral problem}
\label{sec:inverse-spectral}

To summarize what we have seen so far, the Weyl functions
$W(\lambda)$, $Z(\lambda)$, $\twin W(\lambda)$, $\twin Z(\lambda)$
encode much of the information about our twin spectral problems.
In particular, in the discrete interlacing case with positive weights,
the Weyl functions are rational functions in the spectral variable~$\lambda$,
with poles at the (positive and simple) eigenvalues of the spectral problems,
and the functions $Z$ and~$\twin Z$ are completely determined by
$W$ and~$\twin W$
(which in turn are of course determined by the given discrete measures
$m$ and~$n$ that define the whole setup).

The measures depend on the $4K$ parameters
\begin{equation*}
  x_1,x_2,\dots,x_{2K-1},x_{2K},
  \qquad
  m_1,m_3,\dots,m_{2K-1},
  \qquad
  n_2,n_4,\dots,n_{2K}
\end{equation*}
(or equivalently $\{ y_k, g_{2a-1}, h_{2a} \}$),
while the Weyl function $W$ depends on the $2K$ parameters
\begin{equation*}
  \lambda_1,\dots,\lambda_{K},
  \qquad
  a_1,\dots,a_{K},
\end{equation*}
and its twin $\twin W$ on the $2K-1$ parameters
\begin{equation*}
  \mu_1,\dots,\mu_{K-1},
  \qquad
  b_1,\dots,b_{K-1},
  \qquad
  b_{\infty}.
\end{equation*}
To get an inverse spectral problem where the number of spectral data matches
the number of parameters to reconstruct, we therefore need to supplement
$W$ and $\twin W$ by one extra piece of information, and a suitable choice
turns out to be the coefficient $b_{\infty}^*$ defined by \eqref{eq:W-star-twin-parfrac}.
We will show in this section how to recover the discrete interlacing measures $m$ and~$n$
(or, equivalently, their counterparts $g$ and~$h$ on the finite interval)
from this set of spectral data $\{ \lambda_i, a_i, \mu_j, b_j, b_{\infty}, b_{\infty}^* \}$
that they give rise to.
Moreover, we will show that the \emph{necessary} constraints
($0 < \lambda_1 < \dots < \lambda_K$,
$0 < \mu_1 < \dots < \mu_{K-1}$,
and all $a_i$, $b_j$, $b_\infty$, $b^*_\infty$ positive)
are also \emph{sufficient} for such a set of numbers to
be the spectral data of a unique pair of interlacing discrete measures $m$ and~$n$.

\subsection{Approximation problem}
\label{sec:approx-solution}

As we mentioned in \autoref{sec:approx},
the properties in \autoref{thm:QPR-approximation}
are enough to determine the polynomials $Q$, $P$, $R$ uniquely,
and this fact will be proved here.

\begin{theorem}
  \label{thm:approximation-problem-solution}
  Let $b_{\infty}$ be a positive constant.
  Let $\alpha$ and $\beta$ be compactly supported measures on the positive real axis,
  with moments
  \begin{equation}
    \label{eq:alpha-beta-moments}
    \alpha_k = \int x^k \da
    , \qquad
    \beta_k = \int y^k \db
    ,
  \end{equation}
  and bimoments (with respect to the Cauchy kernel $\frac{1}{x+y}$)
  \begin{equation}
    \label{eq:alpha-beta-bimoments}
    I_{km} = \iint \frac{x^k y^m}{x+y} \, \da\db
    .
  \end{equation}
  Define $W$, $\twin W$, $Z$, $\twin Z$ by the formulas \eqref{eq:WZ-as-integrals}
  (repeated here for convenience):
  \begin{align*}
    W(\lambda) &= \int \frac{\da}{\lambda - x}
    ,
    \tag{\ref{eq:W-as-integral}}
    \\
    \twin W(\lambda) &= \int \frac{\db}{\lambda - y} - b_{\infty}
    ,
    \tag{\ref{eq:W-twin-as-integral}}
    \\
    Z(\lambda) &= \frac{1}{2\lambda} + \iint \frac{\da\db}{(\lambda - x)(x+y)} + b_{\infty} W(\lambda)
    ,
    \tag{\ref{eq:Z-as-integral}}
    \\
    \twin Z(\lambda) &= \frac{1}{2\lambda} + \iint \frac{\da\db}{(x+y)(\lambda - y)}
    .
    \tag{\ref{eq:Z-twin-as-integral}}
  \end{align*}
  Fix a positive integer~$j$.
  (If $\alpha$ and $\beta$ are supported at infinitely many points,
  then $j$ can be arbitary.
  In the discrete case with $\alpha = \sum_{i=1}^K a_i \delta_{\lambda_i}$ and
  $\beta = \sum_{i=1}^{K-1} b_i \delta_{\mu_i}$,
  we restrict $j$ to the interval $1 \le j \le K$.)

  Then there are unique polynomials $Q(\lambda) = Q_j(\lambda)$,
  $P(\lambda) = P_j(\lambda)$, $R(\lambda) = R_j(\lambda)$
  satisfying the conditions of \autoref{thm:QPR-approximation}
  (also repeated here for convenience):
  \begin{equation*}
    \deg Q = j,
    \qquad
    \deg P = j-1,
    \qquad
    \deg R = j-1,
    \tag{\ref{eq:QPR-degrees}}
  \end{equation*}
  \begin{equation*}
    Q(0) = 0,
    \qquad
    P(0) = 1,
    \tag{\ref{eq:QP-normalization}}
  \end{equation*}
  and, as $\lambda \to \infty$,
  \begin{gather*}
    W(\lambda) Q(\lambda) - P(\lambda) = \order{\frac{1}{\lambda}},
    \tag{\ref{eq:W-approx}}
    \\
    Z(\lambda) Q(\lambda) - R(\lambda) = \order{\frac{1}{\lambda}},
    \tag{\ref{eq:Z-approx}}
    \\
    R(\lambda) + P(\lambda) \twin W(-\lambda) + Q(\lambda) \twin Z(-\lambda) = \order{\frac{1}{\lambda^{j}}}.
    \tag{\ref{eq:ZW-relation-approx}}
  \end{gather*}
  These polynomials are given by
  \begin{subequations}
    \begin{align}
      Q(\lambda) &= \lambda \, p(\lambda),
      \label{eq:solution-Q}
      \\
      P(\lambda) &= \int \frac{Q(\lambda) - Q(x)}{\lambda - x} \, \da,
      \label{eq:solution-P}
      \\
      R(\lambda) &= \iint \frac{Q(\lambda) - Q(x)}{(\lambda - x)(x+y)} \, \da\db + \tfrac12 p(\lambda) + b_{\infty} P(\lambda),
      \label{eq:solution-R}
    \end{align}
    where
    \begin{equation}
      \label{eq:solution-p}
      p(\lambda) = \frac{\det
        \begin{pmatrix}
          1 & I_{10} & \dots & I_{1,j-2} \\
          \lambda & I_{20} & \dots & I_{2,j-2} \\
          \vdots & \vdots & & \vdots \\
          \lambda^{j-1} & I_{j0} & \dots & I_{j,j-2} \\
        \end{pmatrix}
      }{\det
        \begin{pmatrix}
          \alpha_0 & I_{10} & \dots & I_{1,j-2}  \\
          \alpha_1 & I_{20} & \dots & I_{2,j-2}  \\
          \vdots & \vdots & & \vdots \\
          \alpha_{j-1} & I_{j0} & \dots & I_{j,j-2} \\
        \end{pmatrix}
      }.
    \end{equation}
  \end{subequations}
  (If $j=1$, equation \eqref{eq:solution-p} should be read as $p(\lambda)=1/\alpha_0$.)

  In particular, we have (using notation from \autoref{sec:determinants})
  \begin{equation}
    \label{eq:Qprime-zero}
    Q'(0) = p(0)
    = \frac{\det
      \begin{pmatrix}
        I_{20} & \dots & I_{2,j-2} \\
        \vdots & & \vdots \\
        I_{j0} & \dots & I_{j,j-2} \\
      \end{pmatrix}
    }{\det
      \begin{pmatrix}
        \alpha_0 & I_{10} & \dots & I_{1,j-2}  \\
        \alpha_1 & I_{20} & \dots & I_{2,j-2}  \\
        \vdots & \vdots & & \vdots \\
        \alpha_{j-1} & I_{j0} & \dots & I_{j,j-2} \\
      \end{pmatrix}
    }
    = \frac{\heineintegral_{j-1,j-1}^{20}}{\heineintegral_{j,j-1}^{01}}
  \end{equation}
  (to be read as $Q'(0) = 1/\alpha_0$ if $j=1$), and
  \begin{equation}
    \label{eq:R-zero}
    \begin{split}
      R(0) &= \iint \frac{p(x)}{x+y} \da\db + \tfrac12 p(0) + b_{\infty}
      \\
      &=
      \frac{\det
        \begin{pmatrix}
          I_{00} + \tfrac12 & I_{10} & \dots & I_{1,j-2}  \\
          I_{10} & I_{20} & \dots & I_{2,j-2}  \\
          \vdots & \vdots & & \vdots \\
          I_{j-1,0} & I_{j0} & \dots & I_{j,j-2} \\
        \end{pmatrix}
      }{\det
        \begin{pmatrix}
          \alpha_0 & I_{10} & \dots & I_{1,j-2}  \\
          \alpha_1 & I_{20} & \dots & I_{2,j-2}  \\
          \vdots & \vdots & & \vdots \\
          \alpha_{j-1} & I_{j0} & \dots & I_{j,j-2} \\
        \end{pmatrix}
      }
      + b_{\infty}
      = \frac{\weirddeterminant_j}{\heineintegral_{j,j-1}^{01}} + b_{\infty}
    \end{split}
  \end{equation}
  (to be read as $R(0) = (I_{00} + \tfrac12)/\alpha_0 + b_{\infty}$ if $j=1$).
\end{theorem}

\begin{proof}
  A bit of notation first: define projection operators acting on
  (formal or convergent) Laurent series $f(\lambda) = \sum_{k \in \Z} c_k \lambda^k$
  as follows:
  \begin{equation}
    \label{eq:projection-operators}
    \Pi_{\ge 0} f = \sum_{k \ge 0} c_k \lambda^k, \qquad
    \Pi_{> 0} f = \sum_{k > 0} c_k \lambda^k, \qquad
    \Pi_{< 0} f = \sum_{k < 0} c_k \lambda^k.
  \end{equation}
  Note that we can expand $W(\lambda)$ in a Laurent series with negative powers,
  \begin{equation*}
    W(\lambda) = \int \frac{\da}{\lambda - x}
    = \frac{1}{\lambda} \int \sum_{k \ge 0} \left( \frac{x}{\lambda} \right)^k \da
    = \sum_{k \ge 0} \frac{\alpha_k}{\lambda^{k+1}},
  \end{equation*}
  and similarly for the other Weyl functions.

  We see at once that the conditions \eqref{eq:W-approx} and \eqref{eq:Z-approx}
  determine the polynomials $P$ and $R$ uniquely, by projection on nonnegative powers,
  if the polynomial $Q$ is known:
  \begin{equation}
    \label{eq:PR-projections}
    P = \Pi_{\ge 0} [QW],
    \qquad
    R = \Pi_{\ge 0} [QZ].
  \end{equation}
  Inserting this into \eqref{eq:ZW-relation-approx} gives
  \begin{equation*}
    \Pi_{\ge 0} [QZ](\lambda) + \Pi_{\ge 0}[QW](\lambda) \, \twin W(-\lambda) + Q(\lambda) \twin Z(-\lambda) = \order{\frac{1}{\lambda^{j}}}.
  \end{equation*}
  Writing $\Pi_{\ge 0} = \id{} - \Pi_{< 0}$ produces
  \begin{multline*}
    Q(\lambda) Z(\lambda) + Q(\lambda) W(\lambda) \twin W(-\lambda) + Q(\lambda) \twin Z(-\lambda)
    \\
    - \Pi_{< 0} [QZ](\lambda) - \Pi_{< 0}[QW](\lambda) \, \twin W(-\lambda) = \order{\frac{1}{\lambda^{j}}}, 
  \end{multline*}
  where the first three terms cancel thanks to the identity
  $Z(\lambda) + W(\lambda) \twin W(-\lambda) + \twin Z(-\lambda) = 0$
  which follows from the definitions \eqref{eq:WZ-as-integrals} by a
  short calculation (cf.\ also \eqref{eq:Weyl-relation}).
  This leaves
  \begin{equation}
    \label{eq:Q-projection}
    \Pi_{< 0} [QZ](\lambda) + \Pi_{< 0}[QW](\lambda) \, \twin W(-\lambda) = \order{\frac{1}{\lambda^{j}}}.    
  \end{equation}
  Next, note that
  \begin{equation*}
    Q(\lambda) W(\lambda)
    = Q(\lambda) \int \frac{\da}{\lambda - x}
    = \int \frac{Q(\lambda) - Q(x)}{\lambda - x} \, \da + \int \frac{Q(x)}{\lambda - x} \, \da,
  \end{equation*}
  where the first term is a polynomial in $\lambda$ (since
  $Q(\lambda)-Q(x)$ vanishes when $\lambda=x$ and therefore contains
  $\lambda-x$ as a factor), and the second term is $\order{1/\lambda}$
  as $\lambda \to \infty$.
  Thus, the first and second term are $\Pi_{\ge 0} [QW]$ and $\Pi_{< 0} [QW]$, respectively,
  which gives on the one hand the claimed integral representation for~$P$,
  \begin{equation}
    \label{eq:P-projection-as-integral}
    P(\lambda) = \Pi_{\ge 0} [QW](\lambda) =
    \int \frac{Q(\lambda) - Q(x)}{\lambda - x} \, \da,
  \end{equation}
  and on the other hand, multiplying the negative projection by $\twin W(-\lambda)$,
  \begin{multline}
    \label{eq:first-contribution}
    \Pi_{< 0}[QW](\lambda) \, \twin W(-\lambda)
    = \left( \int \frac{Q(x)}{\lambda - x} \, \da \right) \left( \int \frac{\db}{-\lambda - y} - b_{\infty} \right)
    \\
    = - b_{\infty} \int \frac{Q(x)}{\lambda - x} \, \da - \iint \frac{Q(x)}{(\lambda - x)(\lambda + y)} \, \da\db.
  \end{multline}
  Similarly,
  \begin{multline*}
    Q(\lambda) Z(\lambda)
    = \frac{Q(\lambda)}{2\lambda}
    + \iint \frac{Q(\lambda) - Q(x)}{(\lambda - x)(x+y)} \, \da\db
    \\
    + \iint \frac{Q(x)}{(\lambda - x)(x+y)} \, \da\db
    + b_{\infty} Q(\lambda) W(\lambda),
  \end{multline*}
  where the first term is a polynomial in $\lambda$ since we require $Q(0)=0$,
  likewise the second term is a polynomial (by the same argument as above),
  and the third term is $\order{1/\lambda}$ as $\lambda \to \infty$.
  Thus, we obtain from the first two terms, together with the contribution to nonnegative powers from the fourth term,
  the claimed integral representation for~$R$,
  \begin{equation}
    \label{eq:R-projection-as-integral}
    R(\lambda) = \Pi_{\ge 0} [QZ](\lambda) =
    \frac{Q(\lambda)}{2\lambda}
    + \iint \frac{Q(\lambda) - Q(x)}{(\lambda - x)(x+y)} \, \da\db
    + b_{\infty} P(\lambda),
  \end{equation}
  and from the third term, together with the contribution to negative powers from the fourth term,
  \begin{equation}
    \label{eq:second-contribution}
    \Pi_{< 0}[QZ](\lambda) = 
    \iint \frac{Q(x)}{(\lambda - x)(x+y)} \, \da\db
    + b_{\infty} \int \frac{Q(x)}{\lambda - x} \, \da.
  \end{equation}
  Inserting \eqref{eq:first-contribution} and  \eqref{eq:second-contribution}
  into \eqref{eq:Q-projection} gives
  \begin{equation*}
    \iint \frac{Q(x)}{(\lambda - x)(x+y)} \, \da\db
    - \iint \frac{Q(x)}{(\lambda - x)(\lambda + y)} \, \da\db
    = \order{\frac{1}{\lambda^{j}}},
  \end{equation*}
  which simplifies to
  \begin{equation}
    \label{eq:Q-projection-simplified}
    \iint \frac{Q(x)}{(x+y)(\lambda + y)} \, \da\db
    = \order{\frac{1}{\lambda^{j}}}.
  \end{equation}
  Since $Q(0)=0$, we write $Q(x) = x \, p(x)$, where $p$ is a polynomial of degree $j-1$.
  Upon expanding $1/(\lambda+y) = \sum_{k \ge 0} y^k \lambda^{-(k+1)}$,
  the condition \eqref{eq:Q-projection-simplified} takes the form
  \begin{equation}
    \label{eq:p-biorthogonality-condition}
    \iint \frac{p(x) \, x y^k}{x+y} \ \da\db = 0,
    \qquad
    0 \le k \le j-2.
  \end{equation}
  This imposes $j-1$ linear equations for the $j$ coefficients in~$p(x)=p_0 + p_1 x + \dots + p_{j-1} x^{j-1}$:
  \begin{equation*}
    (p_0, \dots, p_{j-1})
    \begin{pmatrix}
      I_{10} & \dots & I_{1,j-2}  \\
      I_{20} & \dots & I_{2,j-2} \\
      \vdots & & \vdots \\
      I_{j0} & \dots & I_{j,j-2} \\
    \end{pmatrix}
    = (0,\dots,0).
  \end{equation*}
  Adding an extra column,
  \begin{equation*}
    (p_0, \dots, p_{j-1})
    \begin{pmatrix}
      I_{10} & \dots & I_{1,j-2} & I_{1,j-1}  \\
      I_{20} & \dots & I_{2,j-2} & I_{2,j-1} \\
      \vdots & & \vdots & \vdots \\
      I_{j0} & \dots & I_{j,j-2} & I_{j,j-1} \\
    \end{pmatrix}
    = (0,\dots,0,*),
  \end{equation*}
  we see that the row vector $(p_0, \dots, p_{j-1})$ is proportional to
  the last row of the inverse of the bimoment matrix in question
  (which is invertible by the assumption about infinitely many points of support, or
  by the restriction on~$j$ in the discrete case).
  Hence, by Cramer's rule,
  \begin{equation*}
    p(x) = C \det
    \begin{pmatrix}
      I_{10} & \dots & I_{1,j-2} & 1 \\
      I_{20} & \dots & I_{2,j-2} & x \\
      \vdots & & \vdots & \vdots \\
      I_{j0} & \dots & I_{j,j-2} & x^{j-1} \\
    \end{pmatrix}.
  \end{equation*}
  The constant $C$ is determined by the remaining normalization condition $P(0)=1$;
  from \eqref{eq:P-projection-as-integral} we get
  \begin{equation*}
    \begin{split}
      1 &= P(0) = \int \frac{Q(0) - Q(x)}{0 - x} \, \da
      = \int p(x) \, \da
      \\[1ex] &=
      C \det
      \begin{pmatrix}
        I_{10} & \dots & I_{1,j-2} & \int 1 \,\da \\
        I_{20} & \dots & I_{2,j-2} & \int x \,\da \\
        \vdots & & \vdots & \vdots \\
        I_{j0} & \dots & I_{j,j-2} & \int x^{j-1} \,\da \\
      \end{pmatrix}
      \\[1ex] &=
      C \det
      \begin{pmatrix}
        I_{10} & \dots & I_{1,j-2} & \alpha_0 \\
        I_{20} & \dots & I_{2,j-2} & \alpha_1 \\
        \vdots & & \vdots & \vdots \\
        I_{j0} & \dots & I_{j,j-2} & \alpha_{j-1} \\
      \end{pmatrix}
      .
    \end{split}
  \end{equation*}

  Finally, the expression \eqref{eq:Qprime-zero} for $Q'(0)$ follows at once upon setting
  $\lambda = 0$ in $Q'(\lambda) = p(\lambda) + \lambda p'(\lambda)$ and using the determinantal
  expression \eqref{eq:solution-p} for $p(\lambda)$;
  the last term in \eqref{eq:Qprime-zero} represents an evaluation of the determinants in terms of certain integrals
  (which will be sums when the measures $\alpha$ and $\beta$ are discrete).
  This is explained in the appendix; see \autoref{sec:determinants},
  in particular equations \eqref{eq:bimoment-det-shifted} and
  \eqref{eq:p-polynomial-det}.
  The expression \eqref{eq:R-zero} for $R(0)$ is also immediate from the formula \eqref{eq:solution-R} for $R(\lambda)$,
  since $(Q(0)-Q(x))/(0-x) = Q(x)/x = p(x)$.
  (The symbol $\weirddeterminant_j$ is just notation for the determinant in the numerator;
  it doesn't seem to have a simple direct integral representation, but we will mainly be interested
  in the difference between $R_j(0)$ and $R_{j+1}(0)$, which equation
  \eqref{eq:oompa-loompa} takes care of.)
\end{proof}

\begin{remark}
  The polynomial $p(x)$ in \autoref{thm:approximation-problem-solution}
  is proportional to $p_{j-1}(x)$,
  where $\{ p_n(x), q_n(y) \}_{n \ge 0}$ are the normalized Cauchy biorthogonal polynomials
  with respect to the measures $x \, \da$ and $\db$.
  This can be seen either
  directly from the biorthogonality condition \eqref{eq:p-biorthogonality-condition},
  or by comparing the numerators in the formula \eqref{eq:solution-p} for~$p$
  and the formula \eqref{eq:biorth-pn} (with $I_{a+1,b}$ instead of $I_{ab}$) for~$p_n$.
\end{remark}

\subsection{Recovery formulas for the weights and their positions}
\label{sec:recovery}

Most of the work is now done, and we can at last state the
solution to the inverse problem of recovering the weights $g_j$ and
$h_j$ and their positions $y_k$ from the spectral data encoded in the
Weyl functions. The answer will be given in terms of the integrals
$\heineintegral_{nm}^{rs}$ defined by equation
\eqref{eq:heine-integral} in \autoref{sec:determinants} in the appendix.
Since $\alpha$ and $\beta$ are discrete measures here,
these integrals are in fact sums; see \eqref{eq:heine-integral-as-sum}
in \autoref{sec:heine-integral-discrete-case}.

\begin{theorem}
  \label{thm:even-recovery}
  The weights and positions of the even-numbered point masses are given by the formulas
  \begin{align}
    h_{K} &= \frac{I_{00} + \tfrac12}{\alpha_0} + b_\infty,
    \label{eq:recover-last-h}
    \\
    (1 - y_{2K}) h_K &= \frac{1}{\alpha_0},
    \label{eq:recover-last-y}
  \end{align}
  and, for $j = 2, \dots, K$,
  \begin{align}
    h_{K+1-j} &= \frac{\heineintegral_{j-1,j-1}^{10} \bigl( \heineintegral_{j,j-1}^{00} + \tfrac12 \heineintegral_{j-1,j-2}^{11} \bigr)}{\heineintegral_{j-1,j-2}^{01} \heineintegral_{j,j-1}^{01}},
    \label{eq:recover-other-h}
    \\
    (1 - y_{2(K+1-j)}) h_{K+1-j} &= \frac{\heineintegral_{j-1,j-2}^{11} \heineintegral_{j-1,j-1}^{10}}{\heineintegral_{j-1,j-2}^{01} \heineintegral_{j,j-1}^{01}}.
    \label{eq:recover-other-even-y}
  \end{align}
\end{theorem}

\begin{proof}
  We use \autoref{prop:recover-even} together with
  \eqref{eq:Qprime-zero} and \eqref{eq:R-zero} from \autoref{thm:approximation-problem-solution}.
  The rightmost mass is special; we have
  \begin{equation*}
    h_K = R_1(0) - R_0(0) = \left( \frac{I_{00} + \tfrac12}{\alpha_0} + b_\infty \right) - 0
  \end{equation*}
  and
  \begin{equation*}
    (1 - y_{2K}) h_K = Q_1'(0) - Q_0'(0) = \frac{1}{\alpha_0} - 0.
  \end{equation*}
  For the other masses we get
  \begin{equation*}
    h_{K+1-j} = R_{j}(0) - R_{j-1}(0)
    = \left( \frac{\weirddeterminant_{j}}{\heineintegral_{j,j-1}^{01}} + b_\infty \right)
    - \left( \frac{\weirddeterminant_{j-1}}{\heineintegral_{j-1,j-2}^{01}} + b_\infty \right),
  \end{equation*}
  which equals \eqref{eq:recover-other-h} according to \eqref{eq:oompa-loompa},
  and
  \begin{equation*}
    (1 - y_{2(K+1-j)}) h_{K+1-j}
    = Q_{j}'(0) - Q_{j-1}'(0)
    = \frac{\heineintegral_{j-1,j-1}^{20}}{\heineintegral_{j,j-1}^{01}}
    - \frac{\heineintegral_{j-2,j-2}^{20}}{\heineintegral_{j-1,j-2}^{01}},
  \end{equation*}
  which equals \eqref{eq:recover-other-even-y} according to \eqref{eq:heine-lewis-carroll-b}.
\end{proof}

The symmetry described in \autoref{rem:epiphany}
immediately provides formulas for the odd-numbered point masses.
We let $\starheineintegral_{nm}^{rs}$ denote the integral
$\heineintegral_{nm}^{rs}$ evaluated using the measures
\begin{equation}
  \label{eq:alpha-beta-star}
  \alpha^* = \sum_{i=1}^K a_i^* \delta_{\lambda_i}
  \qquad \text{and} \qquad
  \beta^* = \sum_{j=1}^{K-1} b_i^* \delta_{\mu_j}
\end{equation}
in place of $\alpha$ and~$\beta$, and similarly for the moments
$\alpha^*_r  = \starheineintegral_{10}^{rs}$
and~$\beta^*_s = \starheineintegral_{01}^{rs}$,
and the Cauchy bimoments~$I^*_{rs} = \starheineintegral_{11}^{rs}$.
Then the symmetry transformation \eqref{eq:symmetry-substitution}
also entails the substitution
\begin{equation}
  \label{eq:heineintegral-substitution}
  \heineintegral_{nm}^{rs} \mapsto (-1)^{n+m} \starheineintegral_{nm}^{rs}
\end{equation}
(including as special cases
$\alpha_k \mapsto -\alpha^*_k$,
$\beta_k \mapsto -\beta^*_k$,
and $I_{ab} \mapsto I^*_{ab}$).

\begin{corollary}
  \label{cor:odd-recovery}
  The weights and positions of the odd-numbered point masses are given by the formulas
  \begin{align}
    g_{1} &= \frac{I^*_{00} + \tfrac12}{\alpha^*_0} + b^*_\infty,
    \label{eq:recover-first-g}
    \\
    (1 + y_{1}) g_1 &= \frac{1}{\alpha^*_0},
    \label{eq:recover-first-y}
  \end{align}
  and, for $j = 2, \dots, K$,
  \begin{align}
    g_{j} &= \frac{\starheineintegral_{j-1,j-1}^{10} \bigl( \starheineintegral_{j,j-1}^{00} + \tfrac12 \starheineintegral_{j-1,j-2}^{11} \bigr)}{\starheineintegral_{j-1,j-2}^{01} \starheineintegral_{j,j-1}^{01}},
    \label{eq:recover-other-g}
    \\
    (1 + y_{2j-1}) g_{j} &= \frac{\starheineintegral_{j-1,j-2}^{11} \starheineintegral_{j-1,j-1}^{10}}{\starheineintegral_{j-1,j-2}^{01} \starheineintegral_{j,j-1}^{01}}.
    \label{eq:recover-other-odd-y}
  \end{align}
\end{corollary}

\begin{corollary}
  \label{cor:peakon-solution-formulas}
  The corresponding weights and positions on the real line are
  given by the formulas
  \begin{align}
    x_{2K} &= \tfrac12 \ln 2(I_{00} + b_{\infty} \alpha_0),
    \\
    n_{2K} &= \frac{1}{\alpha_0} \sqrt{\frac{I_{00} + b_{\infty} \alpha_0}{2}},
  \end{align}
  \begin{align}
    x_{1} &= -\tfrac12 \ln 2(I^*_{00} + b^*_{\infty} \alpha^*_0),
    \\
    m_{1} &= \frac{1}{\alpha^*_0} \sqrt{\frac{I^*_{00} + b^*_{\infty} \alpha^*_0}{2}},
  \end{align}
  and, for $j = 2, \dots, K$,
  \begin{align}
    x_{2(K+1-j)} &= \tfrac12 \ln \left( \frac{2 \, \heineintegral_{j,j-1}^{00}}{\heineintegral_{j-1,j-2}^{11}} \right),
    \\
    n_{2(K+1-j)} &=
    \frac{\heineintegral_{j-1,j-1}^{10}}{\heineintegral_{j-1,j-2}^{01} \heineintegral_{j,j-1}^{01}}
    \sqrt{\frac{\heineintegral_{j,j-1}^{00} \heineintegral_{j-1,j-2}^{11}}{2}}
    ,
  \end{align}
  \begin{align}
    x_{2j-1} &= -\tfrac12 \ln \left( \frac{2 \, \starheineintegral_{j,j-1}^{00}}{\starheineintegral_{j-1,j-2}^{11}} \right),
    \\
    m_{2j-1} &=
    \frac{\starheineintegral_{j-1,j-1}^{10}}{\starheineintegral_{j-1,j-2}^{01} \starheineintegral_{j,j-1}^{01}}
    \sqrt{\frac{\starheineintegral_{j,j-1}^{00} \starheineintegral_{j-1,j-2}^{11}}{2}}
    .
  \end{align}
  In terms of non-starred quantities (together with $b^*_{\infty}$),
  the odd-numbered variables take the form
  \begin{align}
    x_{2(K+1-j)-1} &= \tfrac12 \ln \left( \frac{2 \, \heineintegral_{jj}^{00}}{\heineintegral_{j-1,j-1}^{11}}\right),
    \\
    m_{2(K+1-j)-1} &=
    \frac{\heineintegral_{j,j-1}^{01}}{\heineintegral_{jj}^{10} \heineintegral_{j-1,j-1}^{10}}
    \sqrt{\frac{\heineintegral_{j-1,j-1}^{11} \heineintegral_{jj}^{00}}{2}}
    ,
  \end{align}
  for $j=1,\dots,K-1$, and
  \begin{align}
    x_1 &= \tfrac12 \ln \left( \frac{2 \, \heineintegral_{K,K-1}^{00}}{\heineintegral_{K-1,K-2}^{11} + \dfrac{\strut 2 b^*_{\infty} L}{M} \, \heineintegral_{K-1,K-1}^{10}} \right)
    \\
    m_1 &=
    \frac{M/L}{\heineintegral_{K-1,K-1}^{10}}
    \sqrt{\frac{\heineintegral_{K,K-1}^{00}}{2} \left( \heineintegral_{K-1,K-2}^{11} + \dfrac{\strut 2 b^*_{\infty} L}{M} \, \heineintegral_{K-1,K-1}^{10} \right)}
  ,
  \end{align}
  where $L = \prod_{i=1}^K \lambda_i$ and $M = \prod_{j=1}^{K-1} \mu_j$.
\end{corollary}

\begin{proof}
  Since $y_k = \tanh x_k = (e^{2x_k} - 1) / (e^{2x_k} + 1)$, we have
  \begin{equation*}
    \exp(2 x_k) = \frac{1+y_k}{1-y_k} = \frac{2}{1-y_k} - 1.
  \end{equation*}
  Moreover,
  $h_j = 2 n_{2j} \cosh x_{2j} = n_{2j} (e^{2 x_{2j}} + 1) e^{-x_{2j}} = n_{2j} \left( \frac{2}{1 - y_{2j}} \right) e^{-x_{2j}}$
  implies that
  \begin{equation*}
    2 n_{2j} \exp(-x_{2j}) = (1 - y_{2j}) h_j.
  \end{equation*}
  Now it is just a matter of plugging in the formulas from \autoref{thm:even-recovery}
  and solving for even-numbered $x_{2j}$ and $n_{2j}$.
  For example:
  \begin{equation*}
    \begin{split}
      \tfrac12 \exp(2 x_{2(K+1-j)})
      &= \frac{1}{1-y_{2(K+1-j)}} - \frac12
      = \frac{h_{K+1-j}}{(1-y_{2(K+1-j)}) h_{K+1-j}} - \frac12
      \\
      &= \frac{\heineintegral_{j-1j-1}^{10} \bigl( \heineintegral_{j,j-1}^{00} + \tfrac12 \heineintegral_{j-1,j-2}^{11} \bigr)}{\heineintegral_{j-1,j-2}^{11} \heineintegral_{j-1,j-1}^{10}} - \frac12
      = \frac{\heineintegral_{j,j-1}^{00}}{\heineintegral_{j-1,j-2}^{11}}.
    \end{split}
  \end{equation*}
  The odd-numbered $x_{2j-1}$ and $m_{2j-1}$ are dealt with similarly,
  using the formulas from \autoref{cor:odd-recovery} together with
  \begin{equation*}
    \exp(-2 x_k) = \frac{1-y_k}{1+y_k} = \frac{2}{1+y_k} - 1
  \end{equation*}
  and
  \begin{equation*}
    2 m_{2j-1} \exp(x_{2j-1}) = (1 + y_{2j-1}) g_j.
  \end{equation*}

  In order to translate starred to non-starred, we use
  \autoref{lem:heine-integral-symmetry}
  (with $A=K$ and $B=K-1$):
  \begin{equation*}
    \begin{split}
      \starheineintegral_{j,j-1}^{00}
      &= \frac{L^{2j-(j-1)+0-1} M^{2(j-1)-j+0-1} \heineintegral_{K-j,(K-1)-(j-1)}^{1-0,1-0}}{2^{j+(j-1)} \heineintegral_{K,K-1}^{00}}
      \\
      &= \frac{L^j M^{j-3} \heineintegral_{K-j,K-j}^{11}}{2^{2j-1} \heineintegral_{K,K-1}^{00}}
      ,
    \end{split}
  \end{equation*}
  and similarly for the other $\starheineintegral_{nm}^{rs}$ occurring in
  the formulas for $x_{2j-1}$ and $m_{2j-1}$.
  All the factors $L$, $M$ and $\heineintegral_{K,K-1}^{00}$ cancel in the quotients,
  except in the formulas for $x_1$ and $m_1$ where we have
  \begin{equation*}
    I^*_{00} = \starheineintegral_{11}^{00} = \frac{\heineintegral_{K-1,K-2}^{11}}{4 \, \heineintegral_{K,K-1}^{00}}
  \end{equation*}
  and
  \begin{equation*}
    \alpha^*_0 = \starheineintegral_{10}^{0s}
    = \starheineintegral_{10}^{01}
    = \frac{L^1 M^{-1} \heineintegral_{K-1,K-1}^{10}}{2 \, \heineintegral_{K,K-1}^{00}}.
  \end{equation*}
\end{proof}

\begin{remark}
  \label{rem:solution-alternative-form}
  A more compact way of writing the solution is to state the formulas in terms
  of the following quantities (where $r=K+1-j$ throughout):
  \begin{equation}
    \label{eq:solution-alternative-position}
    \begin{split}
      \tfrac12 \exp 2 x_{2K} &= I_{00} + b_{\infty} \alpha_0,\\
      \tfrac12 \exp 2 x_{2r} &= \frac{\heineintegral_{j,j-1}^{00}}{\heineintegral_{j-1,j-2}^{11}}, \qquad j = 2,\dots,K,\\
      \tfrac12 \exp 2 x_{2r-1} &= \frac{\heineintegral_{jj}^{00}}{\heineintegral_{j-1,j-1}^{11}}, \qquad j=1,\dots,K-1,\\
      \tfrac12 \exp 2 x_{1} &= \frac{\heineintegral_{K,K-1}^{00}}{\heineintegral_{K-1,K-2}^{11} + \dfrac{\strut 2 b^*_{\infty} L}{M} \, \heineintegral_{K-1,K-1}^{10}}
    \end{split}
  \end{equation}
  and
  \begin{equation}
    \label{eq:solution-alternative-amplitude}
    \begin{split}
      2 n_{2K} \exp(-x_{2K}) &= \frac{1}{\alpha_0},\\
      2 n_{2r} \exp(-x_{2r}) &= \frac{\heineintegral_{j-1,j-2}^{11} \heineintegral_{j-1,j-1}^{10}}{\heineintegral_{j-1,j-2}^{01} \heineintegral_{j,j-1}^{01}}, \qquad j = 2,\dots, K,\\
      2 m_{2r-1} \exp(-x_{2r-1}) &= \frac{\heineintegral_{j-1,j-1}^{11} \heineintegral_{j,j-1}^{01}}{\heineintegral_{jj}^{10} \heineintegral_{j-1,j-1}^{10}}, \qquad j = 1,\dots, K-1,\\
      2 m_{1} \exp(-x_{1}) &= \frac{M \, \heineintegral_{K-1,K-2}^{11}}{L \, \heineintegral_{K-1,K-1}^{10}} + 2 b^*_{\infty}.
    \end{split}
  \end{equation}
\end{remark}

We now know that the set of spectral data computed from the interlacing discrete measures
$m$ and~$n$ allows us to reconstruct these measures uniquely,
and we also know (\autoref{thm:simple-spectra}, \autoref{thm:weyl-parfrac},
equation \eqref{eq:b-infty-b-infty-star-product})
that the eigenvalues are positive and simple and that the residues are positive
(provided that the point masses in $m$ and~$n$ are positive).
Next we will show that there are no further constraints on the spectral data,
i.e., any set of such numbers are the spectral data of a unique pair of
interlacing discrete measures.
It will be convenient to introduce a bit of terminology first.

\begin{definition}
  \label{def:spectral-map}
  Let $\mathcal{P} \subset \R^{4K}$ (the ``pure peakon sector'') be the set of tuples
  \begin{equation*}
    \mathbf{p} =
    (x_1,\dots,x_{2K}; m_1,n_2, \dots, m_{2K-1},n_{2K})
  \end{equation*}
  satisfying
  \begin{equation*}
    x_1 < \dots < x_{2K}
    , \quad
    \text{all $m_{2a-1}>0$}
    , \quad
    \text{all $n_{2a}>0$}
    ,
  \end{equation*}
  and let $\mathcal{R} \subset \R^{4K}$ (the ``set of admissible spectral data'')
  be the set of tuples
  \begin{equation*}
    \mathbf{r} =
    (\lambda_1,\dots,\lambda_{K}; \mu_1, \dots, \mu_{K-1}; a_1, \dots, a_K; b_1, \dots, b_{K-1}; b_\infty, b^*_\infty)
  \end{equation*}
  satisfying
  \begin{equation*}
    0 < \lambda_1 < \dots < \lambda_K
    , \quad
    0 < \mu_1 < \dots < \mu_{K-1}
    , \quad
    \text{all $a_i, b_j, b_\infty, b^*_\infty > 0$}
    .
  \end{equation*}
  The \textbf{forward spectral map} taking a point $\mathbf{p} \in \mathcal{P}$,
  representing a pair of interlacing discrete measures
  \begin{equation*}
    m = 2 \sum_{a=1}^K  m_{2a-1} \, \delta_{x_{2a-1}}
    \quad\text{and}\quad
    n = 2 \sum_{a=1}^K  n_{2a} \, \delta_{x_{2a}}
    ,
  \end{equation*}
  to the corresponding spectral data $\mathbf{r} \in \mathcal{R}$
  (as described in \autoref{sec:discrete-on-the-interval})
  will be denoted by
  \begin{equation}
    \label{eq:forward-map-notation}
    \mathcal{S} \colon \mathcal{P} \to \mathcal{R}
    .
  \end{equation}
  The formulas in \autoref{cor:peakon-solution-formulas}
  (or \autoref{rem:solution-alternative-form})
  define a function
  \begin{equation}
    \label{eq:inverse-map-notation}
    \mathcal{T} \colon \mathcal{R} \to R^{4K}
  \end{equation}
  which we will call the \textbf{inverse spectral map}.
\end{definition}

\begin{theorem}
  \label{thm:spectral-map-bijection}
  For $K \ge 2$,
  the function $\mathcal{S}$ maps $\mathcal{P}$ bijectively onto~$\mathcal{R}$,
  and $\mathcal{T} \colon \mathcal{R} \to \mathcal{P}$ is the inverse map.
  (See \autoref{sec:K1} for the case $K=1$.)
\end{theorem}

\begin{proof}
  To begin with, $\mathcal{T}$ maps $\mathcal{R}$ into~$\mathcal{P}$;
  this is the content of \autoref{lem:maps-into-P} below.
  By \autoref{cor:peakon-solution-formulas}, $\mathcal{T} \circ \mathcal{S} = \id_\mathcal{P}$.
  Thus $\mathcal{S}$ is a homeomorphism onto its range.
  It remains to show that the range of~$\mathcal{S}$ is all of $\mathcal{R}$ and that
  $\mathcal{S} \circ \mathcal{T} = \id_\mathcal{R}$.
  For this, it is most convenient to use the alternative description of the forward spectral
  map given in \autoref{app:real-line},
  where the spectral data are defined directly in terms of $\{ x_k, m_{2a-1}, n_{2a} \}$
  without going via the transformation to the finite interval $[-1,1]$.
  (At the beginning of \autoref{app:real-line} there is a summary comparing the
  two descriptions.)
  The eigenvalues $\lambda_i$ and $\mu_j$,
  as well as the residues $a_i$, $b_j$ and $b_\infty$,
  are all uniquely determined by certain polynomials
  $A(\lambda)$, $\twin A(\lambda)$, $B(\lambda)$ and~$\twin B(\lambda)$
  with the property that their coefficients are polynomials in the variables
  $\{ m_{2a-1} \, e^{\pm x_{2a-1}}, n_{2a} \, e^{\pm x_{2a}} \}$;
  see \eqref{eq:jumpI-even-odd}, \eqref{eq:ABC-jump-product},
  \eqref{eq:jumpII-even-odd} and \eqref{eq:twin-ABC-jump-product}.
  Let us write $A(\lambda; \mathbf{p})$ (etc.) to indicate
  this dependence of the coefficients on the masses and positions.
  It is clear from the symmetry of the problem that $b^*_\infty$
  could also be definied similarly (although we have chosen not to name and
  write out the corresponding polynomials, instead using \eqref{eq:b-infty-star-real-line}
  as the definition).

  Now, since \autoref{rem:solution-alternative-form} exhibits the variables
  $\{ m_{2a-1} \, e^{\pm x_{2a-1}}, n_{2a} \, e^{\pm x_{2a}} \}$ as rational functions of the spectral
  variables~$\mathbf{r}$,
  the coefficients in the polynomial $A(\lambda; \mathcal{T}(\mathbf{r}))$ are also rational
  functions of~$\mathbf{r}$.
  Since $\mathcal{T} \circ \mathcal{S} = \id_\mathcal{P}$,
  we know that these coefficients agree with the coefficients of
  $\prod_{i=1}^K (1-\lambda/\lambda_k)$ (see \eqref{eq:A-twinA-factorization})
  for each $\mathbf{r}$ in the range of~$\mathcal{S}$
  (which is an open set in $\mathcal{R}$ since $\mathcal{S}$ is a homeomorphism).
  Hence $A(\lambda; \mathcal{T}(\mathbf{r})) = \prod_{i=1}^K (1-\lambda/\lambda_k)$
  identically as a rational function of~$\mathbf{r}$,
  and in particular this identity holds for any $\mathbf{r} \in \mathcal{R}$.
  The same argument works for the other polynomials involved in defining
  the spectral variables, and therefore 
  $\mathcal{S} \circ \mathcal{T} = \id_\mathcal{R}$,
  as desired.
\end{proof}

\begin{lemma}
  \label{lem:maps-into-P}
  The function $\mathcal{T}$ maps $\mathcal{R}$ into~$\mathcal{P}$, i.e.,
  the formulas in \autoref{cor:peakon-solution-formulas} give \textbf{positive} masses
  $m_{2a-1} > 0$ and~$n_{2a} > 0$, and \textbf{ordered} positions $x_1 < \dots < x_{2K}$, for any
  spectral data in the admissible set~$\mathcal{R}$.
\end{lemma}

\begin{proof}
  Positivity of $m_{2j-1}$ and~$n_{2j}$ is obvious.
  To show that the positions $x_k$ are ordered, we will use the formulas 
  \eqref{eq:solution-alternative-position}
  for $q_k = \tfrac12 \exp 2 x_{k}$ and show that $q_1 < \dots < q_{2K}$.

  The outermost intervals present no problems, since
  \begin{equation*}
    q_{2K} - q_{2K-1} = (I_{00} + b_{\infty} \alpha_0) -
    \frac{\heineintegral_{11}^{00}}{\heineintegral_{00}^{11}}
    = b_{\infty} \alpha_0 > 0
  \end{equation*}
  and
  \begin{equation*}
    \frac{1}{q_1} - \frac{1}{q_2} =
    \frac{\heineintegral_{K-1,K-2}^{11} + \dfrac{\strut 2 b^*_{\infty} L}{M} \, \heineintegral_{K-1,K-1}^{10}}{\heineintegral_{K,K-1}^{00}} -
    \frac{\heineintegral_{K-1,K-2}^{11}}{\heineintegral_{K,K-1}^{00}}
    = 
    \dfrac{\strut 2 b^*_{\infty} L}{M} \, \frac{\heineintegral_{K-1,K-1}^{10}}{\heineintegral_{K,K-1}^{00}}
    > 0
    .
  \end{equation*}
  As for the other distances, the differences
  \begin{equation*}
    q_{2r} - q_{2r-1} =
    \frac{\heineintegral_{j,j-1}^{00}}{\heineintegral_{j-1,j-2}^{11}}
    - \frac{\heineintegral_{jj}^{00}}{\heineintegral_{j-1,j-1}^{11}}
    , \qquad j=2,\dots,K-1
    ,
  \end{equation*}
  and
  \begin{equation*}
    q_{2r-1} - q_{2r-2} =
    \frac{\heineintegral_{jj}^{00}}{\heineintegral_{j-1,j-1}^{11}}
    - \frac{\heineintegral_{j+1,j}^{00}}{\heineintegral_{j,j-1}^{11}}
    , \qquad j=1,\dots,K-1
    ,
  \end{equation*}
  are all positive according to
  (the fairly technical)
  \autoref{lem:positive-distances}
  in \autoref{sec:heine-integral-discrete-case}.
\end{proof}

\begin{example}
  \label{ex:solution-K2}
  Let us explicitly write out the solution formulas for the inverse
  problem in the case $K=2$, by expanding the sums $\heineintegral_{nm}^{rs}$
  (including $I_{00} = \heineintegral_{11}^{00}$)
  as explained in \autoref{sec:heine-integral-discrete-case};
  recall that the lower indices $n$ and~$m$ give the number of factors $a_i$ and~$b_j$
  in each term,
  and also determine the dimensions of the accompanying Vandermonde-like factors $\Psi_{IJ}$
  (see \eqref{eq:PsiIJ}),
  while the upper indices $r$ and~$s$ are the powers to which the additional factors
  $\lambda_i$ and~$\mu_j$ appear.
  The spectral data are
  \begin{equation*}
    \lambda_1, \lambda_2, \mu_1, a_1, a_2, b_1, b_{\infty}, b_{\infty}^*,
  \end{equation*}
  and we want to recover
  \begin{equation*}
    x_1, x_2, x_3, x_4, m_1, n_2, m_3, n_4.
  \end{equation*}
  In terms of the quantities from \autoref{rem:solution-alternative-form},
  we get
  \begin{equation}
    \label{eq:solution-K2-position}
    \begin{split}
      \tfrac12 e^{2 x_4} &= I_{00} + b_{\infty} \alpha_0 = \frac{a_1 b_1}{\lambda_1+\mu_1} + \frac{a_2 b_1}{\lambda_2+\mu_1} + b_{\infty} (a_1 + a_2),\\
      \tfrac12 e^{2 x_3} &= \frac{\heineintegral_{11}^{00}}{\heineintegral_{00}^{11}} = \frac{I_{00}}{1} = \frac{a_1 b_1}{\lambda_1+\mu_1} + \frac{a_2 b_1}{\lambda_2+\mu_1},\\
      \tfrac12 e^{2 x_2} &= \frac{\heineintegral_{21}^{00}}{\heineintegral_{10}^{11}} = \frac{\dfrac{\bigl( \lambda_1-\lambda_2 \bigr)^2}{(\lambda_1+\mu_1)(\lambda_2+\mu_1)} a_1 a_2 b_1}{\lambda_1 a_1 + \lambda_2 a_2},\\
      \tfrac12 e^{2 x_1} &= \frac{\heineintegral_{21}^{00}}{\heineintegral_{10}^{11} + \dfrac{\strut 2 b^*_{\infty} L}{M} \, \heineintegral_{11}^{10}} \\ &= \frac{\dfrac{\bigl( \lambda_1-\lambda_2 \bigr)^2}{(\lambda_1+\mu_1)(\lambda_2+\mu_1)} a_1 a_2 b_1}{\lambda_1 a_1 + \lambda_2 a_2 + \dfrac{2 b_{\infty}^* \lambda_1 \lambda_2}{\mu_1} \left( \dfrac{\lambda_1 a_1 b_1}{\lambda_1+\mu_1} + \dfrac{\lambda_2 a_2 b_1}{\lambda_2+\mu_1} \right)}
    \end{split}
  \end{equation}
  and
  \begin{equation}
    \label{eq:solution-K2-amplitude}
    \begin{split}
      2 n_4 e^{-x_4} &= \frac{1}{\alpha_0} = \frac{1}{a_1 + a_2}, \\
      2 m_3 e^{-x_3} &= \frac{\heineintegral_{00}^{11} \heineintegral_{10}^{01}}{\heineintegral_{11}^{10} \heineintegral_{00}^{10}} = \frac{1 \cdot \heineintegral_{10}^{01}}{\heineintegral_{11}^{10} \cdot 1} = \frac{a_1 + a_2}{\dfrac{\lambda_1 a_1 b_1}{\lambda_1+\mu_1} + \dfrac{\lambda_2 a_2 b_1}{\lambda_2+\mu_1}}, \\
      2 n_2 e^{-x_2} &= \frac{\heineintegral_{10}^{11} \heineintegral_{11}^{10}}{\heineintegral_{10}^{01} \heineintegral_{21}^{01}} = \frac{\left( \lambda_1 a_1 + \lambda_2 a_2 \right) \left( \dfrac{\lambda_1 a_1 b_1}{\lambda_1+\mu_1} + \dfrac{\lambda_2 a_2 b_1}{\lambda_2+\mu_1} \right)}{\left( a_1 + a_2 \right) \dfrac{\mu_1 \bigl( \lambda_1-\lambda_2 \bigr)^2}{(\lambda_1+\mu_1)(\lambda_2+\mu_1)} a_1 a_2 b_1}, \\
      2 m_1 e^{-x_1} &= \frac{M \, \heineintegral_{10}^{11}}{L \, \heineintegral_{11}^{10}} + 2 b^*_{\infty} = \frac{\mu_1 \bigl( \lambda_1 a_1 + \lambda_2 a_2 \bigr)}{\lambda_1 \lambda_2 \left( \dfrac{\lambda_1 a_1 b_1}{\lambda_1+\mu_1} + \dfrac{\lambda_2 a_2 b_1}{\lambda_2+\mu_1} \right)} + 2 b_{\infty}^*.
    \end{split}
  \end{equation}
\end{example}

\begin{example}
  \label{ex:solution-K3}
  Similarly, in the case $K=3$
  the spectral data are
  \begin{equation*}
    \lambda_1, \lambda_2, \lambda_3, \mu_1, \mu_2, a_1, a_2, a_3, b_1, b_2, b_{\infty}, b_{\infty}^*,
  \end{equation*}
  and we want to recover
  \begin{equation*}
    x_1, x_2, x_3, x_4, x_5, x_6, m_1, n_2, m_3, n_4, m_5, n_6.
  \end{equation*}
  The solution is
  \begin{equation}
    \label{eq:solution-K3-position}
    \begin{split}
      \tfrac12 e^{2 x_6} &= I_{00} + b_{\infty} \alpha_0 = \sum_{i=1}^3 \sum_{j=1}^2 \frac{a_i b_j}{\lambda_i + \mu_j} + b_{\infty} (a_1 + a_2 + a_3),\\
      \tfrac12 e^{2 x_5} &= \frac{\heineintegral_{11}^{00}}{\heineintegral_{00}^{11}} = I_{00} = \sum_{i=1}^3 \sum_{j=1}^2 \frac{a_i b_j}{\lambda_i + \mu_j},\\
      \tfrac12 e^{2 x_4} &= \frac{\heineintegral_{21}^{00}}{\heineintegral_{10}^{11}} = \frac{\displaystyle\sum_{I=12,13,23} \sum_{j=1}^2 \dfrac{(\lambda_{i_1} - \lambda_{i_2})^2}{(\lambda_{i_1} + \mu_j)(\lambda_{i_2} + \mu_j)} \, a_{i_1} a_{i_2} b_j}{\lambda_1 a_1 + \lambda_2 a_2 + \lambda_3 a_3},\\
      \tfrac12 e^{2 x_3} &= \frac{\heineintegral_{22}^{00}}{\heineintegral_{11}^{11}} = \frac{\displaystyle\sum_{I=12,13,23} \dfrac{(\lambda_{i_1} - \lambda_{i_2})^2 (\mu_1 - \mu_2)^2}{\prod_{i \in I} \prod_{j=1}^2 (\lambda_i + \mu_j)} \, a_{i_1} a_{i_2} b_1 b_2}{\displaystyle \sum_{i=1}^3 \sum_{j=1}^2 \frac{\lambda_i \mu_j}{\lambda_i + \mu_j} \, a_i b_j},\\
      \tfrac12 e^{2 x_2} &= \frac{\heineintegral_{32}^{00}}{\heineintegral_{21}^{11}} = \frac{\dfrac{(\lambda_1 - \lambda_2)^2 (\lambda_1 - \lambda_3)^2 (\lambda_2 - \lambda_3)^2 (\mu_1 - \mu_2)^2}{\prod_{i=1}^3 \prod_{j=1}^2 (\lambda_i + \mu_j)} \, a_1 a_2 a_3 b_1 b_2}{\displaystyle\sum_{I=12,13,23} \sum_{j=1}^2 \dfrac{(\lambda_{i_1} - \lambda_{i_2})^2 \lambda_{i_1} \lambda_{i_2} \mu_j}{(\lambda_{i_1} + \mu_j)(\lambda_{i_2} + \mu_j)} \, a_{i_1} a_{i_2} b_j},\\
      \tfrac12 e^{2 x_1} &= \frac{\heineintegral_{32}^{00}}{\heineintegral_{21}^{11} + \dfrac{\strut 2 b^*_{\infty} L}{M} \, \heineintegral_{22}^{10}} \\ &= \frac{\heineintegral_{32}^{00}}{\heineintegral_{21}^{11} + \dfrac{\strut 2 b^*_{\infty} \lambda_1 \lambda_2 \lambda_3}{\mu_1 \mu_2} \!\!\!\! \displaystyle\sum_{I=12,13,23} \!\!\! \dfrac{(\lambda_{i_1} - \lambda_{i_2})^2 (\mu_1 - \mu_2)^2 \lambda_{i_1} \lambda_{i_2}}{\prod_{i \in I} \prod_{j=1}^2 (\lambda_i + \mu_j)} \, a_{i_1} a_{i_2} b_1 b_2}
    \end{split}
  \end{equation}
  and
  \begin{equation}
    \label{eq:solution-K3-amplitude}
    \begin{split}
      2 n_6 e^{-x_6} &= \frac{1}{\alpha_0} = \frac{1}{a_1 + a_2 + a_3},\\
      2 m_5 e^{-x_5} &= \frac{\heineintegral_{00}^{11} \heineintegral_{10}^{01}}{\heineintegral_{11}^{10} \heineintegral_{00}^{10}} = \frac{\heineintegral_{10}^{01}}{\heineintegral_{11}^{10}} = \frac{a_1 + a_2 + a_3}{\displaystyle \sum_{i=1}^3 \sum_{j=1}^2 \frac{\lambda_i}{\lambda_i + \mu_j} \, a_i b_j},\\
      2 n_4 e^{-x_4} &= \frac{\heineintegral_{10}^{11} \heineintegral_{11}^{10}}{\heineintegral_{10}^{01} \heineintegral_{21}^{01}},\\
      2 m_3 e^{-x_3} &= \frac{\heineintegral_{11}^{11} \heineintegral_{21}^{01}}{\heineintegral_{22}^{10} \heineintegral_{11}^{10}},\\
      2 n_2 e^{-x_2} &= \frac{\heineintegral_{21}^{11} \heineintegral_{22}^{10}}{\heineintegral_{21}^{01} \heineintegral_{32}^{01}},\\
      2 m_1 e^{-x_1} &= \frac{M \, \heineintegral_{21}^{11}}{L \, \heineintegral_{22}^{10}} + 2 b^*_{\infty} = \frac{\mu_1 \mu_2 \, \heineintegral_{21}^{11}}{\lambda_1 \lambda_2 \lambda_3 \, \heineintegral_{22}^{10}} + 2 b^*_{\infty}.
    \end{split}
  \end{equation}
  (The last few right-hand sides are too large to write in expanded form here,
  but explicit expressions for the sums $\heineintegral_{nm}^{rs}$ are written out
  in \autoref{ex:nonzero-heineintegrals}.)
\end{example}

\begin{remark}
  \label{rem:peakon-solutions}
  Using the Lax pairs for the Geng--Xue equation, it is not difficult to show
  (details will be published elsewhere)
  that the peakon ODEs \eqref{eq:GX-peakon-ode} induce the following time dependence
  for the spectral variables:
  \begin{equation}
    \label{eq:spectral-variables-ode}
    \dot \lambda_i = 0
    ,\quad
    \dot a_i = \frac{a_i}{\lambda_i}
    ,\quad
    \dot \mu_j = 0
    ,\quad
    \dot b_j = \frac{b_j}{\mu_j}
    ,\quad
    \dot b_\infty = 0
    ,\quad
    \dot b^*_\infty = 0
    .
  \end{equation}
  This means that the formulas in \autoref{cor:peakon-solution-formulas}
  give the solution to the peakon ODEs \eqref{eq:GX-peakon-ode} in the interlacing case,
  if we let the variables $\{ \lambda_i, \mu_j, b_\infty, b^*_\infty \}$ be constant,
  and let $\{ a_i, b_j \}$ have the time dependence
  \begin{equation}
    \label{eq:residues-time-dependence}
    a_i(t) = a_i(0) \, e^{t/\lambda_i}
    ,\qquad
    b_j(t) = b_j(0) \, e^{t/\mu_j}
    ,
  \end{equation}
  and the coefficients derived in \autoref{app:coeffs-via-planar-networks}
  are constants of motion.

  In particular, \eqref{eq:solution-K2-position} and \eqref{eq:solution-K2-amplitude}
  give the solution to the $2+2$ interlacing peakon ODEs
  \begin{equation}
    \label{eq:GX-peakon-ode-K2}
    \begin{split}
      \dot x_1 &= (m_1 + m_3 E_{13}) (n_2 E_{12} + n_4 E_{14})
      , \\
      \dot x_2 &= (m_1 E_{12} + m_3 E_{23}) (n_2 + n_4 E_{24})
      , \\
      \dot x_3 &= (m_1 E_{13} + m_3) (n_2 E_{23} + n_4 E_{34})
      , \\
      \dot x_4 &= (m_1 E_{14} + m_3 E_{34}) (n_2 E_{24} + n_4)
      , \\
      \frac{\dot m_1}{m_1} &= (m_1 + m_3 E_{13}) (n_2 E_{12} + n_4 E_{14}) - 2 m_3 E_{13} (n_2 E_{12} + n_4 E_{14})
      , \\
      \frac{\dot n_2}{n_2} &= (-m_1 E_{12} + m_3 E_{23}) (n_2 + n_4 E_{24}) - 2 (m_1 E_{12} + m_3 E_{23}) n_4 E_{24}
      , \\
      \frac{\dot m_3}{m_3} &= (m_1 E_{13} + m_3) (-n_2 E_{23} + n_4 E_{34}) + 2 m_1 E_{13} (n_2 E_{23} + n_4 E_{34})
      , \\
      \frac{\dot n_4}{n_4} &= (-m_1 E_{14} - m_3 E_{34}) (n_2 E_{24} + n_4) + 2 (m_1 E_{14} + m_3 E_{34}) n_2 E_{24}
      ,
    \end{split}
  \end{equation}
  where $E_{ij} = e^{-\abs{x_i-x_j}} = e^{x_i-x_j}$ for $i<j$,
  and \eqref{eq:solution-K3-position} and \eqref{eq:solution-K3-amplitude}
  give the solution to the corresponding ODEs for the case $K=3$.
  Likewise, \eqref{eq:peakon-recovery-K1} below gives the solution to the $1+1$ peakon ODEs
  \begin{equation}
    \label{eq:GX-peakon-ode-K1}
      \dot x_1 = \dot x_2
      = \frac{\dot m_1}{m_1} = - \frac{\dot n_2}{n_2}
      = m_1 n_2 E_{12}
      .
  \end{equation}
  However, in this last case the equations are rather trivial,
  and all the heavy machinery is not really required.
  Indeed, $m_1 n_2 E_{12}$ is a constant of motion,
  so direct integration gives
  $x_1(t) = x_1(0) + ct$,
  $x_2(t) = x_2(0) + ct$,
  $m_1(t) = m_1(0) e^{ct}$,
  $n_2(t) = n_2(0) e^{-ct}$,
  where $c = m_1(0) \, n_2(0) \, e^{x_1(0) - x_2(0)}$.
\end{remark}

\subsection{The case $K=1$}
\label{sec:K1}

As already mentioned in \autoref{rem:K1},
the case $K=1$ is degenerate. We have
\begin{equation*}
  \begin{split}
    S(\lambda)
    &= L_{2}(\lambda) \jumpmatrix{h_1}{0} L_{1}(\lambda) \jumpmatrix{0}{g_1}  L_{0}(\lambda)
    \\
    &= 
    \begin{pmatrix}
      1 - \lambda g_1 h_1 l_0 & h_1 & g_1 h_1 \\
      -\lambda g_1 l_0 & 1 & g_1 \\
      \lambda^2 g_1 h_1 l_0 l_2 - 2 \lambda & -\lambda h_1 l_2 & 1 - \lambda g_1 h_1 l_2
    \end{pmatrix}
  \end{split}
\end{equation*}
and
\begin{equation*}
  \begin{split}
    \twin S(\lambda)
    &= L_{2}(\lambda) \jumpmatrix{0}{h_1} L_{1}(\lambda) \jumpmatrix{g_1}{0}  L_{0}(\lambda)
    \\
    &=
    \begin{pmatrix}
      1 & g_1 & 0 \\
      -\lambda h_1 (l_0 + l_1) & 1 - \lambda g_1 h_1 l_1 & h_1 \\
      -2 \lambda & -\lambda g_1 (l_1 + l_2) & 1
    \end{pmatrix}
    ,
  \end{split}
\end{equation*}
and it follows (cf. \eqref{eq:weyl-parfrac}) that
\begin{equation*}
  W(\lambda) = -\frac{S_{21}(\lambda)}{S_{31}(\lambda)}
  = \frac{g_1 l_0}{\lambda g_1 h_1 l_0 l_2 - 2}
  = \frac{a_1}{\lambda - \lambda_1},
\end{equation*}
where
\begin{equation}
  \label{eq:lambda1-a1-K1}
  \lambda_1 = \frac{2}{g_1 h_1 l_0 l_2}
  ,\qquad
  a_1 = \frac{1}{h_1 l_2}
  ,
\end{equation}
and
\begin{equation*}
  \twin W(\lambda) = -\frac{\twin S_{21}(\lambda)}{\twin S_{31}(\lambda)}
  = -\frac{h_1 (l_0 + l_1)}{2}
  = - b_{\infty},
\end{equation*}
where
\begin{equation}
  \label{eq:binfty-K1}
  b_{\infty} = \frac{h_1 (l_0 + l_1)}{2} = \frac{h_1 (2 - l_2)}{2}
  .
\end{equation}
Moreover (cf. \eqref{eq:weyl-star-parfrac}),
\begin{equation*}
  \twin W^*(\lambda) = -\frac{\twin S_{32}(\lambda)}{\twin S_{31}(\lambda)}
  = -\frac{g_1 (l_1 + l_2)}{2}
  = - b_{\infty}^*,
\end{equation*}
where
\begin{equation}
  \label{eq:binftystar-K1}
  b_{\infty}^* = \frac{g_1 (l_1 + l_2)}{2} = \frac{g_1 (2 - l_0)}{2}
  .
\end{equation}
We also have
\begin{equation*}
  W^*(\lambda) = -\frac{S_{32}(\lambda)}{S_{31}(\lambda)}
  = \frac{h_1 l_2}{\lambda g_1 h_1 l_0 l_2 - 2}
  = \frac{a^*_1}{\lambda - \lambda_1},
\end{equation*}
where
\begin{equation}
  a^*_1 = \frac{1}{g_1 l_0}
  ,
\end{equation}
so that (cf.~\autoref{thm:residue-relation-W-Wstar})
\begin{equation}
  \label{eq:residue-relation-W-Wstar-K1}
  a_1 a^*_1 = \frac{1}{h_1 l_2} \cdot \frac{1}{g_1 l_0} = \frac{\lambda_1}{2}
  .
\end{equation}
From these equations it follows that
\begin{equation}
  \label{eq:string-recovery-K1}
  \begin{split}
    \frac{1}{l_0} &= a^*_1 b^*_{\infty} + \frac12
    ,\\
    \frac{1}{l_2} &= a_1 b_{\infty} + \frac12
    ,\\
    g_1 &= b_{\infty}^* + \dfrac{1}{2 a^*_1}
    ,\\
    h_1 &= b_{\infty} + \dfrac{1}{2 a_1}
    .
  \end{split}
\end{equation}
Mapping back to the real line using
$-1 + l_0 = y_1 = \tanh x_1$,
$1 - l_2 = y_2 = \tanh x_2$,
$g_1 = 2 m_1 \cosh x_1$
and
$h_1 = 2 n_2 \cosh x_2$,
we get
\begin{equation}
  \label{eq:peakon-recovery-K1}
  \begin{split}
    \tfrac12 e^{2 x_2} &= a_1 b_{\infty},\\
    \tfrac12 e^{-2 x_1} &= a^*_1 b^*_{\infty},\\
    2 n_2 e^{-x_2} &= \frac{1}{a_1}, \\
    2 m_1 e^{x_1} &= \frac{1}{a_1^*}, \\
  \end{split}
\end{equation}
which recovers $x_1$, $x_2$, $m_1$, $n_2$ from the spectral data
$\lambda_1$, $a_1$, $b_{\infty}$, $b_{\infty}^*$ (and $a^*_1 = (2 a_1)^{-1} \lambda_1$).

It is clear that all the spectral variables are positive if $m_1$ and~$n_2$ are positive,
but there is an additional constraint (not present for $K \ge 2$):
the ordering requirement $x_1 < x_2$ is fulfilled if and only if
\begin{equation}
  \label{eq:constraint-K1}
  1 < e^{2(x_2-x_1)} = 4 a_1 a^*_1 b_\infty b^*_\infty = 2 \lambda_1 b_\infty b^*_\infty
  .
\end{equation}
In the terminology of \autoref{def:spectral-map},
the range of the forward spectral map~$\mathcal{S}$ for $K=1$
is not all of~$\mathcal{R}$, but only the subset where $\lambda_1 b_\infty b^*_\infty > \frac12$.

\section{Concluding remarks}
\label{sec:concluding}

In this paper we have studied a third order non-selfadjoint
boundary value problem coming from the Lax pair(s) of the nonlinear integrable PDE
\eqref{eq:GX} put forward by Geng and Xue \cite{geng-xue:cubic-nonlinearity}.
We have given a complete solution of the forward and
inverse spectral problems in the case of two positive interlacing discrete measures.
The main motivation for this is the explicit construction of peakons,
a special class of weak solutions to the PDE; more details about this will be given
in a separate paper.
This inverse problem is closely related to the inverse problems for the
\emph{discrete cubic string}
appearing in connection with peakon solutions to the
Degasperis--Procesi equation \eqref{eq:DP},
and for the \emph{discrete dual cubic string}
playing the corresponding role for Novikov's equation \eqref{eq:Novikov}
(see \cite{lundmark-szmigielski:DPlong}
and \cite{hone-lundmark-szmigielski:novikov}, respectively),
but it has the interesting new feature of involving two Lax pairs
and two independent spectral measures.

\appendix

\section{Cauchy biorthogonal polynomials}
\label{sec:biorth}

The theory of Cauchy biorthogonal polynomials,
developed by Bertola, Gekhtman and Szmigielski
\cite{bertola-gekhtman-szmigielski:cauchy,bertola-gekhtman-szmigielski:cubicstring,bertola-gekhtman-szmigielski:twomatrix,bertola-gekhtman-szmigielski:meijerG},
provides a conceptual framework for understanding the
approximation problems and determinants that appear
in this paper. 
In \hyperref[sec:biorth-def]{Sections \ref*{sec:biorth-def}} and~\ref{sec:four-term} below,
we recall a few of the basic definitions and properties,
just to give a flavour of the theory and put our results in a wider context.
\autoref{sec:determinants} is the crucial one for the purpose of this paper;
it contains determinant evaluations (and also defines notation) used in the main text.

\subsection{Definitions}
\label{sec:biorth-def}

Let $\alpha$ and $\beta$ be measures on the positive real axis,
with finite moments
\begin{equation}
  \label{eq:alpha-beta-moments-appendix}
  \alpha_k = \int x^k \da,
  \qquad
  \beta_k = \int y^k \db,
\end{equation}
and finite bimoments with respect to the Cauchy kernel $1/(x+y)$,
\begin{equation}
  \label{eq:alpha-beta-bimoments-appendix}
  I_{ab} = \iint \frac{x^a y^b}{x+y} \da\db.
\end{equation}
According to \eqref{eq:bimoment-det} below,
the matrix $(I_{ab})_{a,b=0}^{n-1}$ has positive determinant $D_n$ for every~$n$,
provided that $\alpha$ and $\beta$ have infinitely many points of support.
Then there are unique polynomials $(p_n(x))_{n=0}^{\infty}$
and $(q_n(y))_{n=0}^{\infty}$ such that
\begin{enumerate}
\item[(i)] $\deg p_n = \deg q_n = n$ for all~$n$,
\item[(ii)] the biorthogonality condition
  \begin{equation}
    \label{eq:biorth}
    \int \frac{p_i(x) \, q_j(y)}{x+y} \da \db = \delta_{ij}
  \end{equation}
  holds for all $i$ and~$j$ (where $\delta_{ij}$ is the Kronecker delta),
\item[(iii)] for each~$n$, the leading coefficient of~$p_n$ is positive and equal to the leading coefficient of~$q_n$.
\end{enumerate}
These polynomials are given by the determinantal formulas
\begin{equation}
  \label{eq:biorth-pn}
  p_n(x) =
  \frac{1}{\sqrt{D_n D_{n+1}}}
  \begin{vmatrix}
    I_{00} & I_{01} & \cdots & I_{0,n-1} & 1 \\
    I_{10} & I_{11} & \cdots & I_{1,n-1} & x \\
    \vdots & \vdots & & \vdots & \vdots \\
    I_{n-1,0} & I_{n-1,1} & \cdots & I_{n-1,n-1} & x^{n-1} \\
    I_{n0} & I_{n1} & \cdots & I_{n,n-1} & x^n
  \end{vmatrix},
\end{equation}
\begin{equation}
  \label{eq:biorth-qn}
  q_n(y) =
  \frac{1}{\sqrt{D_n D_{n+1}}}
  \begin{vmatrix}
    I_{00} & I_{01} & \cdots & I_{0,n-1} & I_{0n} \\
    I_{10} & I_{11} & \cdots & I_{1,n-1} & I_{1n} \\
    \vdots & \vdots & & \vdots & \vdots \\
    I_{n-1,0} & I_{n-1,1} & \cdots & I_{n-1,n-1} & I_{n-1,n} \\
    1 & y & \cdots & y^{n-1} & y^n
  \end{vmatrix}.
\end{equation}
If either $\alpha$ or $\beta$ (or both) is a discrete measure, the determinant $D_n$ will be zero for all sufficiently
large~$n$, and then there will only be finitely many biorthogonal polynomials;
cf. \autoref{sec:heine-integral-discrete-case}.

\subsection{Four-term recursion}
\label{sec:four-term}

One basic property of the Cauchy kernel which underlies much of the theory is the following:
\begin{multline}
  \label{eq:bimoments-shift-relation}
  I_{a+1,b} + I_{a,b+1} = \iint \frac{x^{a+1} y^{b}}{x+y} \da\db + \iint \frac{x^{a} y^{b+1}}{x+y} \da\db
  \\
  = \iint x^a y^b \da\db = \int x^a \da \int y^b \db
  = \alpha_a \beta_b.
\end{multline}
For example, if $X = (X_{ij})_{i,j \ge 0}$ and $Y = (Y_{ij})_{i,j \ge 0}$
are the semi-infinite Hessenberg matrices (lower triangular plus an extra diagonal
above the main one) defined by
\begin{equation*}
  x
  \begin{pmatrix}
    p_0(x) \\ p_1(x) \\ \vdots
  \end{pmatrix}
  =
  X
  \begin{pmatrix}
    p_0(x) \\ p_1(x) \\ \vdots
  \end{pmatrix},
  \qquad
  y
  \begin{pmatrix}
    q_0(y) \\ q_1(y) \\ \vdots
  \end{pmatrix}
  =
  Y
  \begin{pmatrix}
    q_0(y) \\ q_1(y) \\ \vdots
  \end{pmatrix},
\end{equation*}
then it is straightforward to show that \eqref{eq:bimoments-shift-relation} implies
\begin{equation}
  \label{eq:hessenberg-shift-relation}
  X_{nm} + Y_{mn} = \pi_n \eta_m,
\end{equation}
where the numbers
\begin{equation}
  \pi_n = \int p_n(x) \da =
    \frac{1}{\sqrt{D_n D_{n+1}}}
    \begin{vmatrix}
      I_{00} & I_{01} & \cdots & \alpha_0 \\
      I_{10} & I_{11} & \cdots & \alpha_1 \\
      \vdots & \vdots & & \vdots \\
      I_{n0} & I_{n1} & \cdots & \alpha_n
    \end{vmatrix}
\end{equation}
are positive by \eqref{eq:pi-vector-det} below,
and similarly for $\eta_m = \int q_m(y) \db$.
Since the matrix
\begin{equation}
  L =
  \begin{pmatrix}
    -\pi_0^{-1} & \pi_1^{-1} & 0 & 0 & \cdots \\
    0 & -\pi_1^{-1} & \pi_2^{-1} & 0 & \\
    0 & 0 & -\pi_2^{-1} & \pi_3^{-1} & \\
    \vdots &&&\ddots& \ddots
  \end{pmatrix}
\end{equation}
kills the vector $\boldsymbol{\pi} = (\pi_0,\pi_1,\dots)^T$,
we have $0 = L\boldsymbol{\pi}\boldsymbol{\eta}^T = L(X+Y^T) = LX + LY^T$.
If $M_{[a,b]}$ denotes the set of matrices which are zero outside of
the band of diagonals number $a$ to~$b$ inclusive
(with the main diagonal as number zero, and
subdiagonals labelled by negative numbers),
then $LX \in M_{[0,1]} \cdot M_{[-\infty,1]} = M_{[-\infty,2]}$
and $LY^T \in M_{[0,1]} \cdot M_{[-1,\infty]} = M_{[-1,\infty]}$.
But since their sum is zero, it follows that $LX$ and $LY^T$
are both in $M_{[-1,2]}$ (four-banded).
Thus
\begin{equation}
  x \, L
  \begin{pmatrix}
    p_0(x) \\ p_1(x) \\ \vdots
  \end{pmatrix}
  = LX 
  \begin{pmatrix}
    p_0(x) \\ p_1(x) \\ \vdots
  \end{pmatrix},
\end{equation}
with $ L \in M_{[0,1]}$ and $LX \in M_{[-1,2]}$,
is a four-term recurrence satisfied by the polynomials~$p_n$.

The same argument applied to
$X^T+Y = \boldsymbol{\eta} \boldsymbol{\pi}^T$
shows that the polynomials~$q_n$ satisfy a corresponding four-term recurrence
\begin{equation}
  y \, \widetilde{L}
  \begin{pmatrix}
    q_0(y) \\ q_1(y) \\ \vdots
  \end{pmatrix}
  = \widetilde{L} Y
  \begin{pmatrix}
    q_0(y) \\ q_1(y) \\ \vdots
  \end{pmatrix}
  ,
\end{equation}
where $\widetilde{L}$ is like $L$ except for $\boldsymbol{\pi}$ being replaced by~$\boldsymbol{\eta}$.

Further generalizations of familiar properties from the theory of classical orthogonal polynomials
include Christoffel--Darboux-like identities, interlacing of zeros, characterization by
Hermite--Pad{\'e} and Riemann--Hilbert problems,
and connections to random matrix models;
see \cite{bertola-gekhtman-szmigielski:cauchy,bertola-gekhtman-szmigielski:cubicstring,bertola-gekhtman-szmigielski:twomatrix,bertola-gekhtman-szmigielski:meijerG}
for more information.

\subsection{Determinant identities}
\label{sec:determinants}

For $x = (x_1,\dots,x_n) \in \R^n$, let
\begin{equation}
  \label{eq:def-Delta}
  \Delta(x) = \Delta(x_1,\dots,x_n) = \prod_{i < j} (x_i - x_j)
\end{equation}
and
\begin{equation}
  \label{eq:def-Gamma}
  \Gamma(x) = \Gamma(x_1,\dots,x_n) = \prod_{i < j} (x_i + x_j),
\end{equation}
where the right-hand sides are interpreted as $1$ (empty products) if $n=0$
or $n=1$.
Moreover, for $x \in \R^n$ and $y \in \R^m$, let
\begin{equation}
  \label{eq:def-Gamma-mixed}
  \Gamma(x;y) = \prod_{i=1}^{n} \prod_{j=1}^{m} (x_i + y_j)
  .
\end{equation}
Finally, let $\sigma_n$ be the sector in $\R_+^n$ defined by the inequalities $0 < x_1 < \dots < x_n$.
With this notation in place, we define
\begin{equation}
  \label{eq:heine-integral}
  \heineintegral_{nm}^{rs} =
  \int_{\sigma_{n} \times \sigma_{m}}
  \frac{\Delta(x)^2 \Delta(y)^2 \Bigl( \prod_{i=1}^{n} x_i \Bigr)^r \Bigl( \prod_{j=1}^{m} y_j \Bigr)^s}{\Gamma(x;y)}  \dA{n} \dB{m},
\end{equation}
for $n$ and~$m$ positive. We also consider the degenerate cases
\begin{equation}
  \label{eq:heine-integral-one-index-zero}
  \begin{split}
    \heineintegral_{n0}^{rs} &= \int_{\sigma_{n}} \Delta(x)^2 \Bigl( \prod_{i=1}^{n} x_i \Bigr)^r \dA{n} \qquad (n > 0),
    \\
    \heineintegral_{0m}^{rs} &= \int_{\sigma_{m}} \Delta(y)^2 \Bigl( \prod_{j=1}^{m} y_j \Bigr)^s \dB{m} \qquad (m > 0),
    \\[1ex]
    \heineintegral_{00}^{rs} &= 1.
  \end{split}
\end{equation}
Note that $\heineintegral_{10}^{rs} = \alpha_r$ and $\heineintegral_{01}^{rs} = \beta_s$.

When $\alpha$ and~$\beta$ are discrete measures, the integrals
$\heineintegral_{nm}^{rs}$ reduce to sums; see
\autoref{sec:heine-integral-discrete-case} below.

Many types of determinants involving the bimoments
$I_{ab} = \iint \frac{x^a y^b}{x+y} \da \db$ can be evaluated in terms
of such integrals; these formulas are similar in spirit to Heine's
formula for Hankel determinants of moments $\alpha_k = \int x^k \da$,
encountered in the theory of orthogonal polynomials:
\begin{equation*}
  \det(\alpha_{i+j})_{i,j=0}^{n-1}
  = \int_{\sigma_n} \Delta(x)^2 \dA{n}
  = \heineintegral_{n0}^{00}.
\end{equation*}
Here we have collected a few such formulas
(all formulated for determinants of size $n \times n$).
To begin with, specializing Theorem~2.1 in \cite{bertola-gekhtman-szmigielski:cauchy}
to the case of the Cauchy kernel $K(x,y)=1/(x+y)$,
we get the most basic bimoment determinant identity,
\begin{equation}
  \label{eq:bimoment-det}
  D_n = \det(I_{ij})_{i,j=0}^{n-1}
  = \begin{vmatrix}
    I_{00} & \dots & I_{0,n-1}  \\
    \vdots & & \vdots \\
    I_{n-1,0} & \dots & I_{n-1,n-1} \\
  \end{vmatrix}
 = \heineintegral_{nn}^{00}.
\end{equation}
Applying \eqref{eq:bimoment-det} with the measure $x^r \, \da$ in place of $\da$
and $y^s \, \db$ in place of $\db$ gives
\begin{equation}
  \label{eq:bimoment-det-shifted}
  \det(I_{r+i,s+j})_{i,j=0}^{n-1} = \heineintegral_{nn}^{rs}.
\end{equation}
Proposition~3.1 in \cite{bertola-gekhtman-szmigielski:cauchy} says that
\begin{equation}
  \label{eq:pi-vector-det}
  \begin{vmatrix}
    I_{00} & \dots & I_{0,n-2} & \alpha_0 \\
    \vdots & & \vdots & \vdots \\
    I_{n-1,0} & \dots & I_{n-1,n-2} & \alpha_{n-1} \\
  \end{vmatrix}
  = \heineintegral_{n,n-1}^{00}.
\end{equation}
(For $n=1$, the left-hand side should be read as the $1 \times 1$ determinant with the single entry $\alpha_0$;
this agrees with $\heineintegral_{10}^{00} = \alpha_0$.)
By the same trick, we find from this that
\begin{equation}
  \label{eq:pi-vector-det-shifted}
  \begin{vmatrix}
    I_{rs} & \dots & I_{r,s+n-2} & \alpha_r \\
    \vdots & & \vdots & \vdots \\
    I_{r+n-1,s} & \dots & I_{r+n-1,s+n-2} & \alpha_{r+n-1} \\
  \end{vmatrix}
  = \heineintegral_{n,n-1}^{rs}.
\end{equation}
We also need the following identity:
\begin{equation}
  \label{eq:p-polynomial-det}
  \begin{vmatrix}
    \alpha_0 & I_{10} & \dots & I_{1,n-2} \\
    \vdots & \vdots & & \vdots \\
    \alpha_{n-1} & I_{n0} & \dots & I_{n,n-2} \\
  \end{vmatrix}
  = \heineintegral_{n,n-1}^{01}
\end{equation}

\begin{proof}[Proof of \eqref{eq:p-polynomial-det}]
  From \eqref{eq:pi-vector-det-shifted} we have
  \begin{equation*}
    \heineintegral_{n,n-1}^{01} =
    \begin{vmatrix}
      I_{01} & \dots & I_{0,n-1} & \alpha_0 \\
      \vdots & & \vdots & \vdots \\
      I_{n-1,1} & \dots & I_{n-1,n-1} & \alpha_{n-1} \\
    \end{vmatrix}.
  \end{equation*}
  We rewrite the bimoments as
  $I_{jk} = \alpha_j \beta_{k-1} - I_{j+1,k-1}$
  (using \eqref{eq:bimoments-shift-relation}),
  and then subtract $\beta_{k-1}$ times the last column from the other columns $k=1,\dots,n-1$.
  This transforms the determinant into
  \begin{equation*}
    \begin{vmatrix}
      -I_{10} & \dots & -I_{1,n-2} & \alpha_0 \\
      \vdots & & \vdots & \vdots \\
      -I_{n0} & \dots & -I_{n,n-2} & \alpha_{n-1} \\
    \end{vmatrix}
  \end{equation*}
  without changing its value.
  Now move the column of $\alpha$'s to the left; on its way, it passes
  each of the other columns, thereby cancelling all the minus signs.
\end{proof}

Other useful formulas follow from the Desnanot--Jacobi identity, also known as
Lewis Carroll's identity (or as a special case of Sylvester's identity \cite[Section II.3]{gantmacher:matrixtheoryI}):
if $X$ is an $n \times n$ determinant (with $n \ge 2$), then
\begin{equation}
  \label{eq:LewisCarroll}
  X Y = X^{nn} X^{11} - X^{1n} X^{n1},
\end{equation}
where $X^{ij}$ is the subdeterminant of~$X$ obtained by removing row~$i$ and column~$j$,
and where $Y=(X^{nn})^{11}$ is the ``central'' subdeterminant of~$X$ obtained by removing the first and last row
as well as the first and last column
(for $n=2$, we take $Y=1$ by definition).
For example, applying this identity to the bimoment determinant \eqref{eq:bimoment-det-shifted}
(of size $n+1$ instead of~$n$) gives
\begin{equation}
  \label{eq:heine-lewis-carroll-a}
  \heineintegral_{n+1,n+1}^{rs} \heineintegral_{n-1,n-1}^{r+1,s+1} = \heineintegral_{nn}^{rs} \heineintegral_{nn}^{r+1,s+1} - \heineintegral_{nn}^{r+1,s} \heineintegral_{nn}^{r,s+1},
\end{equation}
and from \eqref{eq:p-polynomial-det} we get
\begin{equation}
  \label{eq:heine-lewis-carroll-b}
  \heineintegral_{n+1,n}^{01} \heineintegral_{n-1,n-1}^{20} = \heineintegral_{n,n-1}^{01} \heineintegral_{nn}^{20} - \heineintegral_{n,n-1}^{11} \heineintegral_{nn}^{10}
\end{equation}
(for $n \ge 1$, in both cases).

Another identity, which holds for arbitrary $z_i$, $w_i$, $X_{ij}$, is
\begin{multline}
  \label{eq:LewisCarroll-on-steroids}
  \begin{vmatrix}
    z_1 & X_{11} & \dots & X_{1,n-1} \\
    z_2 & X_{21} & \dots & X_{2,n-1} \\
    \vdots & \vdots & & \vdots \\
    z_n & X_{n1} & \dots & X_{n,n-1} \\
  \end{vmatrix}
  \begin{vmatrix}
    w_1 & X_{11} & \dots & X_{1,n-1} & X_{1n} \\
    w_2 & X_{21} & \dots & X_{2,n-1} & X_{2n} \\
    \vdots & \vdots & & \vdots & \vdots \\
    w_n & X_{n1} & \dots & X_{n,n-1} & X_{nn} \\
    w_{n+1} & X_{n+1,1} & \dots & X_{n+1,n-1} & X_{n+1,n+1} \\
  \end{vmatrix}
  \\
  =
  \begin{vmatrix}
    w_1 & X_{11} & \dots & X_{1,n-1} \\
    w_2 & X_{21} & \dots & X_{2,n-1} \\
    \vdots & \vdots & & \vdots \\
    w_n & X_{n1} & \dots & X_{n,n-1} \\
  \end{vmatrix}
  \begin{vmatrix}
    z_1 & X_{11} & \dots & X_{1,n-1} & X_{1n} \\
    z_2 & X_{21} & \dots & X_{2,n-1} & X_{2n} \\
    \vdots & \vdots & & \vdots & \vdots \\
    z_n & X_{n1} & \dots & X_{n,n-1} & X_{nn} \\
    z_{n+1} & X_{n+1,1} & \dots & X_{n+1,n-1} & X_{n+1,n+1} \\
  \end{vmatrix}
  -
  \\
  \begin{vmatrix}
    X_{11} & \dots & X_{1,n-1} & X_{1n} \\
    X_{21} & \dots & X_{2,n-1} & X_{2n} \\
    \vdots & & \vdots & \vdots \\
    X_{n1} & \dots & X_{n,n-1} & X_{nn} \\
  \end{vmatrix}
  \begin{vmatrix}
    z_1 & w_1 & X_{11} & \dots & X_{1,n-1} \\
    z_2 & w_2 & X_{21} & \dots & X_{2,n-1} \\
    \vdots & \vdots & \vdots & & \vdots \\
    z_n & w_n & X_{n1} & \dots & X_{n,n-1} \\
    z_{n+1} & w_{n+1} & X_{n+1,1} & \dots & X_{n+1,n-1} \\
  \end{vmatrix}
  .
\end{multline}
Indeed, the coefficients of $z_{n+1}$ on the right-hand side cancel,
and the coefficients of the other variables $z_j$ on both sides agree, which can be seen
by applying the Desnanot--Jacobi identity to the second determinant on
the left-hand side with its $j$th row moved to the top.

In the text -- see equation \eqref{eq:R-zero} -- we encounter the $n \times n$ determinant
\begin{equation}
  \label{eq:weirddeterminant}
  \weirddeterminant_n =
  \begin{vmatrix}
    I_{00} + \tfrac12 & I_{10} & \dots & I_{1,n-2} \\
    I_{10} & I_{20} & \dots & I_{2,n-2} \\
    \vdots & \vdots & & \vdots \\
    I_{n-1,0} & I_{n0} & \dots & I_{n,n-2} \\
  \end{vmatrix},
\end{equation}
which satisfies the recurrence
\begin{equation}
  \label{eq:oompa-loompa}
  \frac{\weirddeterminant_{n+1}}{\heineintegral_{n+1,n}^{01}}
  = \frac{\weirddeterminant_{n} }{\heineintegral_{n,n-1}^{01}}
  + \frac{\heineintegral_{nn}^{10} ( \heineintegral_{n+1,n}^{00} + \tfrac12 \heineintegral_{n,n-1}^{11})}{\heineintegral_{n,n-1}^{01} \heineintegral_{n+1,n}^{01}}.
\end{equation}

\begin{proof}[Proof of \eqref{eq:oompa-loompa}]
  By taking $z_i = \alpha_{i-1}$, $w_i = I_{i0} + \tfrac12 \delta_{i0}$ and $X_{ij} = I_{i,j-1}$ in
  \eqref{eq:LewisCarroll-on-steroids},
  and using \eqref{eq:bimoment-det-shifted} and \eqref{eq:p-polynomial-det},
  we find that
  $\heineintegral_{n,n-1}^{01} \weirddeterminant_{n+1} = \weirddeterminant_{n} \heineintegral_{n+1,n}^{01} - \heineintegral_{nn}^{10} Z$,
  where
  \begin{equation*}
    Z =
    \begin{vmatrix}
      \alpha_0 & I_{00} + \tfrac12 & I_{10} & I_{11} & \dots & I_{1,n-2} \\
      \alpha_1 & I_{10}            & I_{20} & I_{21} & \dots & I_{2,n-2} \\
      \vdots & \vdots & \vdots & \vdots & & \vdots \\
      \alpha_n & I_{n0}            & I_{n+1,0} & I_{n+1,1} & \dots & I_{n+1,n-2} \\
    \end{vmatrix}.
  \end{equation*}
  Using again the trick of rewriting the bimoments (except in column~$2$) as
  $I_{j+1,k} = \alpha_j \beta_k - I_{j,k+1}$,
  subtracting $\beta_{k-3}$ times the first column from column~$k$ (for $k=3,\dots,n$),
  and moving the first column to the right,
  we see that
  \begin{equation*}
    \begin{split}
      Z &= -
      \begin{vmatrix}
        I_{00} + \tfrac12 & I_{01} & I_{02} & \dots & I_{0,n-1} & \alpha_0 \\
        I_{10}            & I_{11} & I_{12} & \dots & I_{1,n-1} & \alpha_1 \\
        \vdots & \vdots & \vdots & & \vdots & \vdots \\
        I_{n0}            & I_{n,1} & I_{n,2} & \dots & I_{n,n-1} & \alpha_n \\
      \end{vmatrix}
      \\[1ex]
      &= -( \heineintegral_{n+1,n}^{00} + \tfrac12 \heineintegral_{n,n-1}^{11}).
    \end{split}
  \end{equation*}
  (The last equality follows from \eqref{eq:pi-vector-det-shifted}.)
  Consequently,
  \begin{equation*}
    \heineintegral_{n,n-1}^{01} \weirddeterminant_{n+1}
    = \weirddeterminant_{n} \heineintegral_{n+1,n}^{01}
    + \heineintegral_{nn}^{10} ( \heineintegral_{n+1,n}^{00} + \tfrac12 \heineintegral_{n,n-1}^{11}),
  \end{equation*}
  which is equivalent to \eqref{eq:oompa-loompa}.
\end{proof}

\begin{remark}
  It was shown in \cite[Lemma~4.10]{lundmark-szmigielski:DPlong} that when $\alpha=\beta$,
  the factorization
  \begin{equation}
    \heineintegral_{nn}^{10}
    = \frac{1}{2^n} \left( \int_{\sigma_n} \frac{\Delta(x)^2}{\Gamma(x)} \, \dA{n} \right)^2
  \end{equation}
  holds. We are not aware of anything similar in the general case with $\alpha \neq \beta$.
\end{remark}

\subsection{The discrete case}
\label{sec:heine-integral-discrete-case}
Consider next the bimoments $I_{ab}$ and the
integrals $\heineintegral_{nm}^{rs}$ defined by \eqref{eq:heine-integral} in the
case when $\alpha$ and~$\beta$ are discrete measures, say
\begin{equation}
  \label{eq:discrete-measures-alpha-beta}
  \alpha = \sum_{i=1}^A a_i \delta_{\lambda_i},
  \qquad
  \beta = \sum_{j=1}^B b_j \delta_{\mu_j}.
\end{equation}
(In the setup in the main text, we have $A = K$ and $B = K-1$.)
Then the bimoments become
\begin{equation}
  \label{eq:bimoment-as-sum}
  I_{ab} = \iint \frac{x^a y^b}{x+y} \da\db
  = \sum_{i=1}^A \sum_{j=1}^B \frac{\lambda_i^a \mu_j^b}{\lambda_i+\mu_j} \, a_i b_j,
\end{equation}
and likewise the integrals $\heineintegral_{nm}^{rs}$ turn into sums:
\begin{equation}
  \label{eq:heine-integral-as-sum}
  \heineintegral_{nm}^{rs} =
  \sum_{I \in \binom{[A]}{n}} \sum_{J \in \binom{[B]}{m}}
  \Psi_{IJ} \, \lambda_I^r a_I \, \mu_J^s b_J.
\end{equation}
Here $\binom{[A]}{n}$ denotes the set of $n$-element subsets $I = \{ i_1 < i_2 < \dots < i_n \}$ of the integer interval
$[A] = \{ 1,2, \dots, A \}$, and similarly for $\binom{[B]}{m}$.
Morever,
\begin{equation}
  \label{eq:product-notation-explanation}
  \lambda_I^r a_I \, \mu_J^s b_J =
  \Bigl( \prod_{i \in I} \lambda_i^r a_i \Bigr)
  \Bigl( \prod_{j \in J} \mu_j^s b_j \Bigr)
\end{equation}
and
\begin{equation}
  \label{eq:PsiIJ}
  \Psi_{IJ} = \frac{\Delta_I^2 \twin\Delta_J^2}{\Gamma_{IJ}},
\end{equation}
where we use the shorthand notation
\begin{equation}
  \label{eq:DeltaI}
  \begin{split}
    \Delta_I^2 &= \Delta(\lambda_{i_1},\dots,\lambda_{i_n})^2 = \prod_{\substack{a,b \in I \\ a < b}} (\lambda_{a} - \lambda_{b})^2
    , \\
    \twin\Delta_J^2 &= \Delta(\mu_{j_1},\dots,\mu_{j_m})^2 = \prod_{\substack{a,b \in J \\ a < b}} (\mu_{a} - \mu_{b})^2
    , \\
    \Gamma_{IJ} &= \Gamma(\lambda_{i_1},\dots,\lambda_{i_n};\mu_{j_1},\dots,\mu_{j_m}) = \prod_{i \in I, \, j \in J} (\lambda_i + \mu_j)
    .
  \end{split}
\end{equation}
For later use, we also introduce the symbol
\begin{equation}
  \label{eq:DeltaI1I2}
  \Delta_{I_1 I_2}^2 = \prod_{i_1 \in I_1, \, i_2 \in I_2} (\lambda_{i_1} - \lambda_{i_2})^2
  ,
\end{equation}
and similarly for $\twin \Delta_{J_1 J_2}$.
Empty products (as in $\Delta_{I}^2$ when $I$ is a singleton or the empty set)
are taken to be~$1$ by definition.
When needed for the sake of clarity,
we will write $\Psi_{I,J}$ instead of $\Psi_{IJ}$, etc.

For (positive) measures on the positive real line ($a_i$, $\lambda_i$, $b_j$, $\mu_j$ positive),
we thus have $\heineintegral_{nm}^{rs} > 0$ for $0 \le n \le A$ and $0 \le m \le B$, otherwise
$\heineintegral_{nm}^{rs} = 0$.

\begin{example}
  \label{ex:nonzero-heineintegrals}
  Below we have listed the nonzero $\heineintegral_{nm}^{00}$ in the case $A=3$ and $B=2$
  (the more general sum $\heineintegral_{nm}^{rs}$ is obtained by replacing each $a_i$
  and~$b_j$ in $\heineintegral_{nm}^{00}$ by $\lambda_i^r a_i$ and $\mu_j^s b_j$,
  respectively):
  \begin{equation*}
    \begin{split}
      \heineintegral_{00}^{00} &= 1, \\
      \heineintegral_{10}^{00} &= a_1 + a_2 + a_3, \\
      \heineintegral_{20}^{00} &=
      (\lambda_1 - \lambda_2)^2 a_1 a_2 + (\lambda_1 - \lambda_3)^2 a_1 a_3 + (\lambda_2 - \lambda_3)^2 a_2 a_3, \\
      \heineintegral_{30}^{00} &= (\lambda_1 - \lambda_2)^2 (\lambda_1 - \lambda_3)^2 (\lambda_2 - \lambda_3)^2 a_1 a_2 a_3, \\
    \end{split}
  \end{equation*}
  \begin{equation*}
    \begin{split}
      \heineintegral_{01}^{00} &= b_1 + b_2, \\
      \heineintegral_{11}^{00} &= I_{00} \\ &= 
      \frac{1}{\lambda_1 + \mu_1} a_1 b_1 
      + \frac{1}{\lambda_2 + \mu_1} a_2 b_1
      + \frac{1}{\lambda_3 + \mu_1} a_3 b_1
      \\ & \quad
      + \frac{1}{\lambda_1 + \mu_2} a_1 b_2
      + \frac{1}{\lambda_2 + \mu_2} a_2 b_2
      + \frac{1}{\lambda_3 + \mu_2} a_3 b_2
      ,\\
      \heineintegral_{21}^{00} &= 
        \frac{(\lambda_1 - \lambda_2)^2}{(\lambda_1 + \mu_1)(\lambda_2 + \mu_1)} a_1 a_2 b_1 
      + \frac{(\lambda_1 - \lambda_2)^2}{(\lambda_1 + \mu_2)(\lambda_2 + \mu_2)} a_1 a_2 b_2 
      \\ & \quad
      + \frac{(\lambda_1 - \lambda_3)^2}{(\lambda_1 + \mu_1)(\lambda_3 + \mu_1)} a_1 a_3 b_1 
      + \frac{(\lambda_1 - \lambda_3)^2}{(\lambda_1 + \mu_2)(\lambda_3 + \mu_2)} a_1 a_3 b_2 
      \\ & \quad
      + \frac{(\lambda_2 - \lambda_3)^2}{(\lambda_2 + \mu_1)(\lambda_2 + \mu_1)} a_2 a_3 b_1 
      + \frac{(\lambda_2 - \lambda_3)^2}{(\lambda_2 + \mu_2)(\lambda_2 + \mu_2)} a_2 a_3 b_2 
      ,\\
      \heineintegral_{31}^{00} &=
      \frac{(\lambda_1 - \lambda_2)^2 (\lambda_1 - \lambda_3)^2 (\lambda_2 - \lambda_3)^2}{(\lambda_1 + \mu_1)(\lambda_2 + \mu_1)(\lambda_3 + \mu_1)} a_1 a_2 a_3 b_1
      \\ & \quad
      + \frac{(\lambda_1 - \lambda_2)^2 (\lambda_1 - \lambda_3)^2 (\lambda_2 - \lambda_3)^2}{(\lambda_1 + \mu_2)(\lambda_2 + \mu_2)(\lambda_3 + \mu_2)} a_1 a_2 a_3 b_2
      ,\\
    \end{split}
  \end{equation*}
  \begin{equation*}
    \begin{split}
      \heineintegral_{02}^{00} &= (\mu_1 - \mu_2)^2 b_1 b_2
      ,\\
      \heineintegral_{12}^{00} &=
      \frac{(\mu_1 - \mu_2)^2}{(\lambda_1 + \mu_1)(\lambda_1 + \mu_2)} a_1 b_1 b_2
      + \frac{(\mu_1 - \mu_2)^2}{(\lambda_2 + \mu_1)(\lambda_2 + \mu_2)} a_2 b_1 b_2
      \\ & \quad
      + \frac{(\mu_1 - \mu_2)^2}{(\lambda_3 + \mu_1)(\lambda_3 + \mu_2)} a_3 b_1 b_2
      ,\\
      \heineintegral_{22}^{00} &=
      \frac{(\lambda_1 - \lambda_2)^2 (\mu_1 - \mu_2)^2}{(\lambda_1 + \mu_1)(\lambda_2 + \mu_1) (\lambda_1 + \mu_2)(\lambda_2 + \mu_2)} a_1 a_2 b_1 b_2
      \\ & \quad
      + \frac{(\lambda_1 - \lambda_3)^2 (\mu_1 - \mu_2)^2}{(\lambda_1 + \mu_1)(\lambda_3 + \mu_1) (\lambda_1 + \mu_2)(\lambda_3 + \mu_2)} a_1 a_3 b_1 b_2
      \\ & \quad
      + \frac{(\lambda_2 - \lambda_3)^2 (\mu_1 - \mu_2)^2}{(\lambda_2 + \mu_1)(\lambda_3 + \mu_1) (\lambda_2 + \mu_2)(\lambda_3 + \mu_2)} a_2 a_3 b_1 b_2
      ,\\
      \heineintegral_{32}^{00} &=
      \frac{(\lambda_1 - \lambda_2)^2 (\lambda_1 - \lambda_3)^2 (\lambda_2 - \lambda_3)^2 (\mu_1 - \mu_2)^2}{(\lambda_1 + \mu_1)(\lambda_2 + \mu_1)(\lambda_3 + \mu_1) (\lambda_1 + \mu_2)(\lambda_2 + \mu_2)(\lambda_3 + \mu_2)} a_1 a_2 a_3 b_1 b_2
      .
    \end{split}
  \end{equation*}
\end{example}

\begin{lemma}
  \label{lem:heine-integral-symmetry}
  Let $\starheineintegral_{nm}^{rs}$ denote the integral
  $\heineintegral_{nm}^{rs}$ evaluated using the measures
  \begin{equation}
    \label{eq:alpha-beta-star-appendix}
    \alpha^* = \sum_{i=1}^A a_i^* \delta_{\lambda_i}
    \qquad \text{and} \qquad
    \beta^* = \sum_{j=1}^B b_i^* \delta_{\mu_j}
  \end{equation}
  in place of $\alpha$ and~$\beta$,
  and suppose that the four measures are related as in
  \autoref{thm:residue-relation-W-Wstar}:
  \begin{equation}
    \label{eq:a-astar-b-bstar}
    a_k a_k^*
    = \frac{\displaystyle \lambda_k \prod_{j=1}^{B} \left( 1 + \frac{\lambda_k}{\mu_j} \right) }{\displaystyle 2 \prod_{\substack{i=1 \\ i \neq k}}^{A} \left( 1 - \frac{\lambda_k}{\lambda_i} \right)^2},
    \qquad
    b_k b_k^*
    = \frac{\displaystyle \mu_k \prod_{i=1}^{A} \left( 1 + \frac{\mu_k}{\lambda_i} \right) }{\displaystyle 2 \prod_{\substack{j=1 \\ j \neq k}}^{B} \left( 1 - \frac{\mu_k}{\mu_j} \right)^2}.
  \end{equation}
  Then
  \begin{equation}
    \label{eq:heine-integral-symmetry}
    \begin{split}
      \starheineintegral_{nm}^{rs}
      &= 
      \frac{1}{2^{n+m}}
      \left( \prod_{i=1}^A \lambda_i \right)^{2n-m+r-1} \left( \prod_{j=1}^B \mu_j \right)^{2m-n+s-1}
      \frac{\heineintegral_{A-n,B-m}^{1-r,1-s}}{\heineintegral_{AB}^{00}}
      \\
      &= \frac{\heineintegral_{A-n,B-m}^{1-r,1-s}}{2^{n+m} \heineintegral_{AB}^{m-2n+1-r,n-2m+1-s}}
      .
    \end{split}
  \end{equation}
\end{lemma}

\begin{proof}
  This is a fairly straightforward computation.
  To begin with, let $L=\prod_{i=1}^A \lambda_i$ and $M=\prod_{j=1}^B \mu_j$,
  and write \eqref{eq:a-astar-b-bstar} as
  \begin{equation*}
    a_k a_k^* = \frac{L^2 \, \Gamma_{\{ k \},[B]}}{2 \lambda_k M \, \displaystyle\prod_{\substack{i=1 \\ i \neq k}}^{A} (\lambda_i - \lambda_k)^2},
    \qquad
    b_k b_k^* = \frac{M^2 \, \Gamma_{[A],\{ k \}}}{2 \mu_k L \, \displaystyle\prod_{\substack{j=1 \\ j \neq k}}^{B} (\mu_j - \mu_k)^2}.
  \end{equation*}
  If $I \in \binom{[A]}{n}$, we therefore have
  \begin{equation*}
    a_I a^*_I = \prod_{i \in I} a_i a_i^*
    = \frac{1}{\lambda_I} \left( \frac{L^2}{2M}\right)^n
    \frac{\Gamma_{I,[B]}}{\displaystyle \prod_{i \in I} \biggl( \, \prod_{\substack{t=1 \\ t \neq i}}^{A} (\lambda_t - \lambda_i)^2 \biggr)}
    .
  \end{equation*}
  (The factor in front is $\lambda_I = \prod_{i \in I} \lambda_i$.)
  In the denominator, the factor $(\lambda_p - \lambda_q)^2$ will appear twice if $p$ and $q$ both belong to~$I$,
  once if one of them does, and not at all if both belong to $[A] \setminus I$.
  Thus we can write
  \begin{equation*}
    a_I a^*_I =
    \frac{1}{\lambda_I} \left( \frac{L^2}{2M} \right)^n
    \frac{\Gamma_{I,[B]} \, \Delta_{[A] \setminus I}^2}{\Delta_I^2 \, \Delta_{[A]}^2}.
  \end{equation*}
  Similarly,
  \begin{equation*}
    b_J b^*_J =
    \frac{1}{\mu_J} \left( \frac{M^2}{2L}\right)^m
    \frac{\Gamma_{[A],J} \, \twin\Delta_{[B] \setminus J}^2}{\twin\Delta_J^2 \, \twin\Delta_{[B]}^2}.
  \end{equation*}
  Putting this into the definition of $\starheineintegral_{nm}^{rs}$ gives
  \begin{equation*}
    \begin{split}
      \starheineintegral_{nm}^{rs}
      &=
      \sum_{I \in \binom{[A]}{n}} \sum_{J \in \binom{[B]}{m}}
      \Psi_{IJ} \, \lambda_I^r a^*_I \, \mu_J^s b^*_J
      \\
      &= \sum_I \sum_J \biggl(
      \frac{\Delta_I^2 \twin\Delta_J^2}{\Gamma_{I,J}} \,
      \lambda_I^r \,
      \frac{1}{\lambda_I a_I} \left( \frac{L^2}{2M}\right)^n
      \frac{\Gamma_{I,[B]} \, \Delta_{[A] \setminus I}^2}{\Delta_I^2 \, \Delta_{[A]}^2} \times
      \\ &\qquad\qquad
      \mu_J^s \,
      \frac{1}{\mu_J b_J} \left( \frac{M^2}{2L}\right)^m
      \frac{\Gamma_{[A],J} \, \twin\Delta_{[B] \setminus J}^2}{\twin\Delta_J^2 \, \twin\Delta_{[B]}^2}
      \biggr)
    \end{split}
  \end{equation*}
  Now note that
  \begin{equation*}
    \begin{split}
      \Gamma_{I,[B]} &= \Gamma_{I,J} \, \Gamma_{I,[B] \setminus J}, \\
      \Gamma_{[A],J} &= \Gamma_{I,J} \, \Gamma_{[A] \setminus I,J}, \\
      \Gamma_{[A],[B]} &= \Gamma_{I,J} \, \Gamma_{[A] \setminus I,J} \, \Gamma_{I,[B] \setminus J} \, \Gamma_{[A] \setminus I,[B] \setminus J},
    \end{split}
  \end{equation*}
  which implies
  \begin{equation*}
    \frac{\Gamma_{I,[B]} \, \Gamma_{[A],J}}{\Gamma_{I,J}}
    = \frac{\Gamma_{[A],[B]}}{\Gamma_{[A] \setminus I,[B] \setminus J}}.
  \end{equation*}
  Thus,
  \begin{equation*}
    \begin{split}
      \starheineintegral_{nm}^{rs}
      &= \sum_I \sum_J \biggl(
      \frac{\lambda_I^{r-1} \mu_J^{s-1} L^{2n-m} M^{2m-n}}{2^{n+m} a_I b_J}
      \times
      \\
      & \qquad \qquad
      \frac{\Gamma_{[A],[B]}}{\Delta_{[A]}^2 \twin\Delta_{[B]}^2}
      \times
      \frac{\Delta_{[A] \setminus I}^2 \twin\Delta_{[B] \setminus J}^2}{\Gamma_{[A] \setminus I,[B] \setminus J}}
      \biggr)
      \\
      &= \sum_I \sum_J \biggl(
      \left( \frac{L}{\lambda_{[A] \setminus I}} \right)^{r-1}
      \left( \frac{M}{\mu_{[B] \setminus J}} \right)^{s-1}
      \frac{L^{2n-m} M^{2m-n}}{2^{n+m}}
      \times
      \\
      & \qquad \qquad
      \frac{a_{[A] \setminus I} b_{[B] \setminus J}}{a_{[A]} b_{[B]}}
      \times
      \frac{\Psi_{[A] \setminus I, [B] \setminus J}}{\Psi_{[A],[B]}}
      \biggr)
      \\
      &=
      \frac{L^{2n-m+r-1} M^{2m-n+s-1}}{2^{n+m} \Psi_{[A],[B]} a_{[A]} b_{[B]}}
      \times
      \\
      & \qquad
      \sum_I \sum_J \Psi_{[A] \setminus I, [B] \setminus J} \, (\lambda_{[A] \setminus I})^{1-r} \, a_{[A] \setminus I} \, (\mu_{[B] \setminus J})^{1-s} \, b_{[B] \setminus J}
      \\
      &=
      \frac{L^{2n-m+r-1} M^{2m-n+s-1}}{2^{n+m} \heineintegral_{AB}^{00}}
      \sum_{U \in \binom{[A]}{A-n}} \sum_{V \in \binom{[B]}{B-m}}
      \Psi_{UV} \, \lambda_{U}^{1-r} a_{U} \, \mu_{V}^{1-s} b_{V}
      \\
      &=
      \frac{L^{2n-m+r-1} M^{2m-n+s-1}}{2^{n+m} \heineintegral_{AB}^{00}} \, \heineintegral_{A-n,B-m}^{1-r,1-s}
      ,
    \end{split}
  \end{equation*}
  as claimed.
\end{proof}

The following lemma is needed in the proof of \autoref{lem:positive-distances}.
\begin{lemma}
  \label{lem:positive-distances}
  Suppose (as in the main text) that the number of point masses in $\alpha$ and $\beta$
  are $K$ and~$K-1$:
  \begin{equation*}
    \alpha = \sum_{i=1}^K a_i \delta_{\lambda_i},
    \qquad
    \beta = \sum_{j=1}^{K-1} b_j \delta_{\mu_j}.
  \end{equation*}
  Then the quantities $\heineintegral_{nm}^{rs}$ satisfy
  \begin{equation}
    \label{eq:distances1}
    \heineintegral_{j,j-1}^{00} \heineintegral_{j-1,j-1}^{11} - \heineintegral_{jj}^{00} \heineintegral_{j-1,j-2}^{11} > 0
    , \qquad j=2,\dots,K-1
    ,
  \end{equation}
  and
  \begin{equation}
    \label{eq:distances2}
    \heineintegral_{jj}^{00} \heineintegral_{j,j-1}^{11} - \heineintegral_{j+1,j}^{00} \heineintegral_{j-1,j-1}^{11} > 0
    , \qquad j=1,\dots,K-1
    .
  \end{equation}
\end{lemma}

\begin{proof}
  In \eqref{eq:distances1} we let $j=m+1$ for convenience ($1 \le m \le K-2$),
  and expand the left-hand side using \eqref{eq:heine-integral-as-sum}:
  \begin{equation}
    \label{eq:expand-JJ-JJ}
    \begin{split}
      & \heineintegral_{m+1,m}^{00} \heineintegral_{mm}^{11} - \heineintegral_{m+1,m+1}^{00} \heineintegral_{m,m-1}^{11}
      \\
      &=
      \Biggl(
      \sum_{A \in \binom{[K]}{m+1}} \sum_{C \in \binom{[K-1]}{m}}
      \Psi_{AC}^{} \, a_A^{} \, b_C^{}
      \Biggr)
      \Biggl(
      \sum_{B \in \binom{[K]}{m}} \sum_{D \in \binom{[K-1]}{m}}
      \Psi_{BD}^{} \, \lambda_B^{} \, \mu_D^{} \, a_B^{} \, b_D^{}
      \Biggr)
      \\
      &\quad -
      \Biggl(
      \sum_{A \in \binom{[K]}{m+1}} \sum_{C \in \binom{[K-1]}{m+1}}
      \Psi_{AC}^{} \, a_A^{} \, b_C^{}
      \Biggr)
      \Biggl(
      \sum_{B \in \binom{[K]}{m}} \sum_{D \in \binom{[K-1]}{m-1}}
      \Psi_{BD}^{} \, \lambda_B^{} \, \mu_D^{} \, a_B^{} \, b_D^{}
      \Biggr)
      \\
      &=
      \sum_{A \in \binom{[K]}{m+1}} \sum_{B \in \binom{[K]}{m}} \sum_{C \in \binom{[K-1]}{m}} \sum_{D \in \binom{[K-1]}{m}}
      \Psi_{AC}^{} \Psi_{BD}^{} \, \lambda_B^{} \, \mu_D^{} \, a_A^{} \, a_B^{} \, b_C^{} \, b_D^{}
      \\
      &\quad -
      \sum_{A \in \binom{[K]}{m+1}} \sum_{B \in \binom{[K]}{m}} \sum_{C \in \binom{[K-1]}{m+1}} \sum_{D \in \binom{[K-1]}{m-1}}
      \Psi_{AC}^{} \Psi_{BD}^{} \, \lambda_B^{} \, \mu_D^{} \, a_A^{} \, a_B^{} \, b_C^{} \, b_D^{}
      .
    \end{split}
  \end{equation}
  Let us denote the summand by
  \begin{equation*}
    f(A,B,C,D) = \Psi_{AC}^{} \Psi_{BD}^{} \, \lambda_B^{} \, \mu_D^{} \, a_A^{} \, a_B^{} \, b_C^{} \, b_D^{}
  \end{equation*}
  for simplicity.

  Choosing two subsets $A$ and~$B$ of a set is equivalent to first choosing $R = A \cap B$ and then
  choosing two disjoint sets 
  $X = A \setminus (A \cap B)$ and $Y = B \setminus (A \cap B)$
  among the remaining elements.
  If $\abs{A}=m+1$ and $\abs{B}=m$, then
  \begin{equation*}
    \abs{A \cap B} = m-k
    ,\qquad
    \abs{A \setminus (A \cap B)}=k+1
    ,\qquad
    \abs{B \setminus (A \cap B)}=k
    ,
  \end{equation*}
  for some $k \in \{ 0,1,\dots,m \}$.
  Thus
  \begin{equation*}
    \sum_{A \in \binom{[K]}{m+1}} \sum_{B \in \binom{[K]}{m}} f(A,B,C,D)
    = \sum_{k=0}^m \sum_{R \in \binom{[K]}{m-k}} \sum_{\substack{X \in \binom{[K] \setminus R}{k+1} \\ Y \in \binom{[K] \setminus R}{k} \\ X \cap Y = \emptyset}} f(R+X, R+Y, C, D),
  \end{equation*}
  where we write $R+X$ rather than $R \cup X$, in order to indicate that it is a union of
  \emph{disjoint} sets.
  (If $K - (m-k) < 2k+1$, then the innermost sum is empty.)
  With similar rewriting for $C$ and~$D$, \eqref{eq:expand-JJ-JJ} becomes
  \begin{multline}
    \label{eq:expand-JJ-JJ-2}
    \sum_{k=0}^m
    \sum_{l=0}^{m-1}
    \sum_{R \in \binom{[K]}{m-k}}
    \sum_{S \in \binom{[K-1]}{m-l}}
    \sum_{\substack{X \in \binom{[K] \setminus R}{k+1} \\ Y \in \binom{[K] \setminus R}{k} \\ X \cap Y = \emptyset}}
    \sum_{\substack{Z \in \binom{[K-1] \setminus S}{l} \\ W \in \binom{[K-1] \setminus S}{l} \\ Z \cap W = \emptyset}}
    f(R+X,R+Y,S+Z,S+W)
    \\
    -
    \sum_{k=0}^m
    \sum_{l=0}^{m-1}
    \sum_{R \in \binom{[K]}{m-k}}
    \sum_{S \in \binom{[K-1]}{m-l}}
    \sum_{\substack{X \in \binom{[K] \setminus R}{k+1} \\ Y \in \binom{[K] \setminus R}{k} \\ X \cap Y = \emptyset}}
    \sum_{\substack{Z \in \binom{[K-1] \setminus S}{l+1} \\ W \in \binom{[K-1] \setminus S}{l-1} \\ Z \cap W = \emptyset}}
    f(R+X,R+Y,S+Z,S+W)
    .
  \end{multline}
  (The innermost sum on the second line is empty when $l=0$.)
  Now, since
  \begin{equation*}
    \begin{split}
      \Psi_{R+X,S+Z}^{} \Psi_{R+Y,S+W}^{}
      &=
      \frac{\Delta_{R+X}^2 \twin\Delta_{S+Z}^2}{\Gamma_{R+X,S+Z}^{}} \, \frac{\Delta_{R+Y}^2 \twin\Delta_{S+W}^2}{\Gamma_{R+Y,S+W}^{}}
      \\
      &=
      \frac{\Delta_{R}^2 \Delta_{X}^2 \Delta_{RX}^2 \twin\Delta_{S}^2 \twin\Delta_{Z}^2 \twin\Delta_{SZ}^2}{\Gamma_{RS}^{} \Gamma_{RZ}^{} \Gamma_{XS}^{} \Gamma_{XZ}^{}}
      \, \frac{\Delta_{R}^2 \Delta_{Y}^2 \Delta_{RY}^2 \twin\Delta_{S}^2 \twin\Delta_{W}^2 \twin\Delta_{SW}^2}{\Gamma_{RS}^{} \Gamma_{RW}^{} \Gamma_{YS}^{} \Gamma_{YW}^{}}
      \\
      &=
      \Delta_{X}^2 \Delta_{Y}^2 \twin\Delta_{Z}^2 \twin\Delta_{W}^2 \Gamma_{XW}^{} \Gamma_{YZ}^{}
      \, \frac{\Delta_{R}^4 \Delta_{R,X+Y}^2 \twin\Delta_{S}^4 \twin\Delta_{S,Z+W}^2}{\Gamma_{R+X+Y,S+Z+W}^{} \Gamma_{RS}^{}}
      ,
    \end{split}
  \end{equation*}
  we have
  \begin{equation*}
    \begin{split}
      & f(R+X,R+Y,S+Z,S+W)
      \\
      &=
      \Psi_{R+X,S+Z}^{} \Psi_{R+Y,S+W}^{} \, \lambda_{R+Y}^{} \, \mu_{S+W}^{} \, a_{R+X}^{} \, a_{R+Y}^{} \, b_{S+Z}^{} \, b_{S+W}^{}
      \\
      &=
      \Bigl(
      \Delta_{X}^2 \Delta_{Y}^2 \twin\Delta_{Z}^2 \twin\Delta_{W}^2 \Gamma_{XW}^{} \Gamma_{YZ}^{} \, \lambda_{Y}^{} \, \mu_{W}^{}
      \Bigr)
      \times
      \\
      & \quad
      \left(
        \frac{\Delta_{R}^4 \Delta_{R,X+Y}^2 \twin\Delta_{S}^4 \twin\Delta_{S,Z+W}^2}{\Gamma_{R+X+Y,S+Z+W}^{} \Gamma_{RS}^{}}
        \, \lambda_{R}^{} \, \mu_{S}^{}  \, a_{R}^{2} \, a_{X+Y}^{} \, b_{S}^{2} \, b_{Z+W}^{}
    \right)
    ,
    \end{split}
  \end{equation*}
  where the first factor depends on the sets $X$, $Y$, $Z$ and $W$ individually,
  while the second factor involves only their unions $U=X+Y$ and $V=Z+W$.
  Therefore we can write \eqref{eq:expand-JJ-JJ-2} as
  \begin{equation}
    \label{eq:expand-JJ-JJ-3}
    \begin{split}
      &
      \sum_{k=0}^m
      \sum_{l=0}^{m-1}
      \sum_{R \in \binom{[K]}{m-k}}
      \sum_{S \in \binom{[K-1]}{m-l}}
      \sum_{U \in \binom{[K] \setminus R}{2k+1}}
      \sum_{V \in \binom{[K-1] \setminus S}{2l}}
      \Biggl(
      \frac{\Delta_{R}^4 \Delta_{RU}^2 \twin\Delta_{S}^4 \twin\Delta_{SV}^2}{\Gamma_{R+U,S+V}^{} \Gamma_{RS}^{}}
      \times
      \\
      & \quad
      \lambda_{R}^{} \, \mu_{S}^{}  \, a_{R}^{2} \, a_{U}^{} \, b_{S}^{2} \, b_{V}^{}
      \Biggl(
      \sum_{\substack{X + Y = U \\ \abs{X}=k+1 \\ \abs{Y}=k}} \sum_{\substack{Z + W = V \\ \abs{Z}=l \\ \abs{W}=l}} \Delta_X^2 \Delta_Y^2 \twin\Delta_Z^2 \twin\Delta_W^2 \Gamma_{XW}^{} \Gamma_{YZ}^{} \lambda_Y^{} \mu_W^{}
      \\ & \qquad\qquad\qquad -
      \sum_{\substack{X + Y = U \\ \abs{X}=k+1 \\ \abs{Y}=k}} \sum_{\substack{Z + W = V \\ \abs{Z}=l+1 \\ \abs{W}=l-1}} \Delta_X^2 \Delta_Y^2 \twin\Delta_Z^2 \twin\Delta_W^2 \Gamma_{XW}^{} \Gamma_{YZ}^{} \lambda_Y^{} \mu_W^{}
      \Biggr)
      \Biggr)
    \end{split}
  \end{equation}
  This is in fact positive (which is what we wanted to prove),
  because of the following identity for the expression in brackets:
  if $\abs{U} = 2k+1$ and $\abs{V} = 2l$, then
  \begin{equation}
    \label{eq:formidable-identity}
    \begin{split}
      &
      \sum_{\substack{X + Y = U \\ \abs{X}=k+1 \\ \abs{Y}=k}} \sum_{\substack{Z + W = V \\ \abs{Z}=l \\ \abs{W}=l}} \Delta_X^2 \Delta_Y^2 \twin\Delta_Z^2 \twin\Delta_W^2 \Gamma_{XW}^{} \Gamma_{YZ}^{} \lambda_Y^{} \mu_W^{}
      \\ & \qquad -
      \sum_{\substack{X + Y = U \\ \abs{X}=k+1 \\ \abs{Y}=k}} \sum_{\substack{Z + W = V \\ \abs{Z}=l+1 \\ \abs{W}=l-1}} \Delta_X^2 \Delta_Y^2 \twin\Delta_Z^2 \twin\Delta_W^2 \Gamma_{XW}^{} \Gamma_{YZ}^{} \lambda_Y^{} \mu_W^{}
      \\ & =
      \sum_{\substack{X + Y = U \\ \abs{X}=k+1 \\ \abs{Y}=k}} \sum_{\substack{Z + W = V \\ \abs{Z}=l \\ \abs{W}=l}} \Delta_X^2 \Delta_Y^2 \twin\Delta_Z^2 \twin\Delta_W^2 \Gamma_{XZ}^{} \Gamma_{YW}^{} \lambda_Y^{} \mu_W^{}
      .
    \end{split}
  \end{equation}
  (Note the change from $\Gamma_{XW}^{} \Gamma_{YZ}^{}$ on the left to $\Gamma_{XZ}^{} \Gamma_{YW}^{}$
  on the right.)

  To prove \eqref{eq:formidable-identity},
  we can take $U=[2k+1]$ and $V=[2l]$ without loss of generality.
  If $l=0$, the identity is trivial, since both sides reduce to
  $\sum_{X+Y=U} \Delta_X^2 \Delta_Y^2 \lambda_Y$.
  The rest of the proof concerns the case $l \ge 1$.

  When $k=0$ and $l=1$, both sides reduce to $\lambda_1 (\mu_1+\mu_2) + 2 \mu_1 \mu_2$.
  For fixed $(k,l)$ with $k \ge 1$, evaluation at $\lambda_{2k} = \lambda_{2k+1} = c$
  gives on both sides $2c \prod_{i=1}^{2k-1} (\lambda_i - c)^2 \prod_{j=1}^{2l} (c + \mu_j)$
  times the corresponding expression with $k-1$ instead of~$k$.
  Provided that the identity for $(k-1,l)$ is true,
  our $(k,l)$ identity therefore holds when $\lambda_{2k} = \lambda_{2k+1}$,
  and in fact (because of the symmetry) whenever any two $\lambda_i$ are equal.
  This implies that the difference between the left-hand side
  and the right-hand side is divisible by $\Delta_U^2$.
  (Any polynomial $p$ in the variables $\lambda_i$ which vanishes whenever two
  $\lambda_i$ are equal is divisible by $\Delta_U$. If in addition $p$ is a \emph{symmetric}
  polynomial,
  then $p/\Delta_U$ is antisymmetric and therefore again vanishes whenever two
  $\lambda_i$ are equal; hence $p/\Delta_U$ is divisible by $\Delta_U$.)
  Considered as polynomials in $\lambda_1$, the difference has degree $2k+l$
  and $\Delta_U^2$ has degree~$4k$.
  If $l < 2k$, then $2k+l < 4k$, and in this case we can conclude that the difference
  must be identically zero. To summarize: if the $(k-1,l)$ identity is true and $l < 2k$,
  then the $(k,l)$ identity is also true.
  
  Similarly, for fixed $(k,l)$ with $l \ge 2$, evaluation at $\mu_{2l-1} = \mu_{2l} = c$
  gives $2c \prod_{i=1}^{2k+1} (\lambda_i + c) \prod_{j=1}^{2l-2} (\mu_j - c)^2$
  times the corresponding identity with $l-1$ instead of~$l$.
  As polynomials in~$\mu_1$, the difference between the left-hand side
  and the right-hand side has degree $2l+k$ and $\twin\Delta_V^2$ has degree $4l-2$.
  The same argument as above shows that
  if the $(k,l-1)$ identity is true and $k < 2l-2$, then the $(k,l)$ identity is also true.
  
  Since any integer pair $(k,l)$ with $k \ge 0$ and $l \ge 1$, except $(k,l)=(0,1)$,
  satisfies at least one of the inequalities $l < 2k$ or $k < 2l-2$, we can work our
  way down to the already proved base case $(0,1)$ from any other $(k,l)$ by decreasing either $k$
  or~$l$ by one in each step. This concludes the proof of \eqref{eq:formidable-identity},
  and thereby \eqref{eq:distances1} is also proved.
  
  The proof of \eqref{eq:distances2} is similar: the left-hand side expands to
  \begin{multline}
    \label{eq:replay-expand-JJ-JJ-2}
    \sum_{k=0}^j
    \sum_{l=1}^{j}
    \sum_{R \in \binom{[K]}{j-k}}
    \sum_{S \in \binom{[K-1]}{j-l}}
    \sum_{\substack{X \in \binom{[K] \setminus R}{k} \\ Y \in \binom{[K] \setminus R}{k} \\ X \cap Y = \emptyset}}
    \sum_{\substack{Z \in \binom{[K-1] \setminus S}{l} \\ W \in \binom{[K-1] \setminus S}{l-1} \\ Z \cap W = \emptyset}}
    f(R+X,R+Y,S+Z,S+W)
    \\
    -
    \sum_{k=0}^j
    \sum_{l=1}^{j}
    \sum_{R \in \binom{[K]}{j-k}}
    \sum_{S \in \binom{[K-1]}{j-l}}
    \sum_{\substack{X \in \binom{[K] \setminus R}{k+1} \\ Y \in \binom{[K] \setminus R}{k-1} \\ X \cap Y = \emptyset}}
    \sum_{\substack{Z \in \binom{[K-1] \setminus S}{l} \\ W \in \binom{[K-1] \setminus S}{l-1} \\ Z \cap W = \emptyset}}
    f(R+X,R+Y,S+Z,S+W)
    ,
  \end{multline}
  which equals
  \begin{equation}
    \label{eq:replay-expand-JJ-JJ-3}
    \begin{split}
      &
      \sum_{k=0}^j
      \sum_{l=1}^{j}
      \sum_{R \in \binom{[K]}{j-k}}
      \sum_{S \in \binom{[K-1]}{j-l}}
      \sum_{U \in \binom{[K] \setminus R}{2k}}
      \sum_{V \in \binom{[K-1] \setminus S}{2l-1}}
      \Biggl(
      \frac{\Delta_{R}^4 \Delta_{RU}^2 \twin\Delta_{S}^4 \twin\Delta_{SV}^2}{\Gamma_{R+U,S+V}^{} \Gamma_{RS}^{}}
      \times
      \\
      & \quad
      \lambda_{R}^{} \, \mu_{S}^{}  \, a_{R}^{2} \, a_{U}^{} \, b_{S}^{2} \, b_{V}^{}
      \Biggl(
      \sum_{\substack{X + Y = U \\ \abs{X}=k \\ \abs{Y}=k}} \sum_{\substack{Z + W = V \\ \abs{Z}=l \\ \abs{W}=l-1}} \Delta_X^2 \Delta_Y^2 \twin\Delta_Z^2 \twin\Delta_W^2 \Gamma_{XW}^{} \Gamma_{YZ}^{} \lambda_Y^{} \mu_W^{}
      \\ & \qquad\qquad\qquad -
      \sum_{\substack{X + Y = U \\ \abs{X}=k+1 \\ \abs{Y}=k-1}} \sum_{\substack{Z + W = V \\ \abs{Z}=l \\ \abs{W}=l-1}} \Delta_X^2 \Delta_Y^2 \twin\Delta_Z^2 \twin\Delta_W^2 \Gamma_{XW}^{} \Gamma_{YZ}^{} \lambda_Y^{} \mu_W^{}
      \Biggr)
      \Biggr)
      ,
    \end{split}
  \end{equation}
  which is positive, since for $\abs{U} = 2k$ and $\abs{V} = 2l-1$ the identity
  \begin{equation}
    \label{eq:replay-formidable-identity}
    \begin{split}
      &
      \sum_{\substack{X + Y = U \\ \abs{X}=k \\ \abs{Y}=k}} \sum_{\substack{Z + W = V \\ \abs{Z}=l \\ \abs{W}=l-1}} \Delta_X^2 \Delta_Y^2 \twin\Delta_Z^2 \twin\Delta_W^2 \Gamma_{XW}^{} \Gamma_{YZ}^{} \lambda_Y^{} \mu_W^{}
      \\ & \qquad -
      \sum_{\substack{X + Y = U \\ \abs{X}=k+1 \\ \abs{Y}=k-1}} \sum_{\substack{Z + W = V \\ \abs{Z}=l \\ \abs{W}=l-1}} \Delta_X^2 \Delta_Y^2 \twin\Delta_Z^2 \twin\Delta_W^2 \Gamma_{XW}^{} \Gamma_{YZ}^{} \lambda_Y^{} \mu_W^{}
      \\ & =
      \sum_{\substack{X + Y = U \\ \abs{X}=k \\ \abs{Y}=k}} \sum_{\substack{Z + W = V \\ \abs{Z}=l \\ \abs{W}=l-1}} \Delta_X^2 \Delta_Y^2 \twin\Delta_Z^2 \twin\Delta_W^2 \Gamma_{XZ}^{} \Gamma_{YW}^{} \lambda_Y^{} \mu_W^{}
    \end{split}
  \end{equation}
  holds; it is proved using the same technique as above.
\end{proof}

\section{The forward spectral problem on the real line}
\label{app:real-line}

Consider the Lax equations \eqref{eq:laxI-x} and \eqref{eq:laxII-x}
in the interlacing discrete case \eqref{eq:mn-interlacing}.
In what follows, we will start from scratch 
and analyze these equations directly on real line,
without passing to the finite interval $(-1,1)$
via the transformation \eqref{eq:liouville-trf}.
As we will see, this leads in a natural way to the definition of certain polynomials
$\{ A_k(\lambda), B_k(\lambda), C_k(\lambda) \}_{k=0}^N$.
Since the two approaches are equivalent (cf.\ \autoref{rem:transformation}),
these polynomials are of course related to the quantitites defined
in \autoref{sec:discrete-on-the-interval}.
Before delving into the details, let us just state these relations,
for the sake of comparison.

In the interval $y_k < y < y_{k+1}$,
the wave function $\Phi(y;\lambda)$ is given by
\begin{equation}
  \label{eq:phi-inbetween}
  \begin{pmatrix} \phi_1(y;\lambda) \\ \phi_2(y;\lambda) \\ \phi_3(y;\lambda) \end{pmatrix} =
  \begin{pmatrix}
    A_k(\lambda) - \lambda C_k(\lambda) \\ -2 \lambda B_k(\lambda) \\ -\lambda (1+y) A_k(\lambda) + \lambda^2 (1-y) C_k(\lambda)
  \end{pmatrix}
  .
\end{equation}
Hence, letting $y \to 1^-$ we obtain (with $(A,B,C)$ as synonyms for $(A_N,B_N,C_N)$)
\begin{equation}
  \label{eq:Phi-at-1}
  \begin{pmatrix} \phi_1(1;\lambda) \\ \phi_2(1;\lambda) \\ \phi_3(1;\lambda) \end{pmatrix} =
  \begin{pmatrix}
     A(\lambda) - \lambda C(\lambda) \\
     -2 \lambda B(\lambda) \\
     -2 \lambda A(\lambda)
  \end{pmatrix}
  .
\end{equation}
As for the spectra and the Weyl functions, we have
\begin{equation}
  \label{eq:A-twinA-factorization}
  A(\lambda) = \prod_{k=1}^{K} \left( 1 - \frac{\lambda}{\lambda_k}\right)
  ,
  \qquad
  \twin A(\lambda) = \prod_{k=1}^{K-1} \left( 1 - \frac{\lambda}{\mu_k}\right)
  ,
\end{equation}
\begin{equation}
  \label{eq:WZ-ABC}
  W(\lambda) = -\frac{\phi_2(1;\lambda)}{\phi_3(1;\lambda)}
  = - \frac{B(\lambda)}{A(\lambda)}
  ,
  \qquad
  Z(\lambda) = -\frac{\phi_1(1;\lambda)}{\phi_3(1;\lambda)}
  = \frac{1}{2\lambda} - \frac{C(\lambda)}{A(\lambda)}
  ,
\end{equation}
\begin{equation}
  \label{eq:WZ-ABC-twin}
  \twin W(\lambda) =
  -\frac{\twin \phi_2(1;\lambda)}{\twin \phi_3(1;\lambda)} =
  - \frac{\twin B(\lambda)}{\twin A(\lambda)}
  ,
  \qquad
  \twin Z(\lambda) =
  -\frac{\twin \phi_1(1;\lambda)}{\twin \phi_3(1;\lambda)} =
  \frac{1}{2\lambda} - \frac{\twin C(\lambda)}{\twin A(\lambda)}
  .
\end{equation}
The residues $a_i$, $b_j$ and $b_\infty$ are defined from the Weyl functions
as before; see \autoref{thm:weyl-parfrac}.
One can also define similar polynomials corresponding to the adjoint Weyl functions,
in order to define $b^*_\infty$ as in \eqref{eq:weyl-star-parfrac},
but it is perhaps more convenient to define $b^*_\infty$ using the relations
\eqref{eq:b-infty-b-infty-star-product} and
\eqref{eq:b-infty-b-infty-star-product-K1},
which in this setting take the form
\begin{equation}
  \label{eq:b-infty-star-real-line}
  b_{\infty} b^*_{\infty}
  =
  \begin{cases}
    \displaystyle
    \frac12 \biggl( \prod_{i=1}^K \frac{(1-E_{2i-1,2i}^2)}{\lambda_i \, E_{2i-1,2i}^2} \biggr)
    \biggl( \prod_{j=1}^{K-1} \frac{\mu_j}{(1-E_{2j,2j+1}^2)} \biggr)
    , &
    K \ge 2
    ,\\[1.5em]
    \displaystyle
    \frac{1}{2 \lambda_1 E_{12}^2}
    , &
    K = 1
    ,
  \end{cases}
\end{equation}
where $E_{ij} = e^{-\abs{x_i-x_j}} = e^{x_i-x_j}$ for $i<j$.

\subsection{Setup}
\label{app:setup}

To begin with, recall equation \eqref{eq:laxI-x},
which determines $\Psi(x;z)$:
\begin{equation}
  \label{eq:laxI-x-components}
  \begin{split}
    \partial_x \psi_1(x;z) &= z n(x) \psi_2(x;z) + \psi_3(x;z), \\
    \partial_x \psi_2(x;z) &= z m(x) \psi_3(x;z), \\
    \partial_x \psi_3(x;z) &= \psi_1(x;z),
  \end{split}
\end{equation}
Away from the points $x_k$ where the distributions $m$ and~$n$ are supported,
this reduces to
\begin{equation*}
  \partial_x \psi_1 = \psi_3,
  \qquad
  \partial_x \psi_2 = 0,
  \qquad
  \partial_x \psi_3 = \psi_1,
\end{equation*}
so $\psi_2(x;z)$ is piecewise constant, and
$\psi_1(x;z)$ and $\psi_3(x;z)$ are piecewise linear combinations
of $e^{x}$ and~$e^{-x}$.
It is convenient to write this as
\begin{equation}
  \label{eq:psi-inbetween}
    \begin{pmatrix} \psi_1(x;z) \\ \psi_2(x;z) \\ \psi_3(x;z) \end{pmatrix} =
    \begin{pmatrix}
      A_k e^x + z^2 C_k e^{-x} \\ 2z B_k \\ A_k e^x - z^2 C_k e^{-x}
    \end{pmatrix},
    \qquad
    x_k < x < x_{k+1},
\end{equation}
where the coefficients $\{ A_k, B_k, C_k \}_{k=0}^{N}$ may depend on~$z$ but not on~$x$.
(Here we set $x_0 = -\infty$ and $x_{N+1} = +\infty$,
so that the $x$ axis splits into $N+1$ intervals
$x_k < x < x_{k+1}$ numbered by $k=0, 1, \dots, N$.)
Then the conditions \eqref{eq:boundary-conditions-x} and \eqref{eq:initial-condition-x}
for $\Psi(x;z)$ at $\pm\infty$ translate into
\begin{equation}
  \label{eq:A0B0C0}
  B_0 = C_0 = 0 = A_N
  , \qquad
  A_0 = 1,
\end{equation}
respectively.
So we impose $(A_0,B_0,C_0) = (1,0,0)$ (i.e., $\Psi(x;z) = (e^x,0,e^x)^T$ for $x < x_1$)
and investigate for which $z$ the condition $A_N(z) = 0$ is satisfied;
the corresponding values $\lambda = -z^2$ will be the eigenvalues
considered in the main text (cf.\ \autoref{rem:lambda-vs-z}).

The pieces \eqref{eq:psi-inbetween} are stitched together by
evaluating equations \eqref{eq:laxI-x-components} at the sites $x=x_{k}$.
Since the Dirac delta is the distributional derivative of the Heaviside
step function,
a jump in $\psi_i$ at $x_k$ will give rise to a Dirac delta term $\delta_{x_k}$
in~$\partial_x \psi_i$,
whose coefficient must match that of the corresponding Dirac delta coming from
$m$ or~$n$ on the right-hand side of \eqref{eq:laxI-x-components}.
Denoting jumps by $\jump{f(x_k)} = f(x_k^+) - f(x_k^-)$,
we find at the odd-numbered sites $x=x_k=x_{2a-1}$ (where $m$ is supported)
the jump conditions
\begin{equation*}
  \begin{split}
    \jump{\psi_1(x_k;z)} &= 0, \\
    \jump{\psi_2(x_k;z)} &= 2z m_k \psi_3(x_k), \\
    \jump{\psi_3(x_k;z)} &= 0,
  \end{split}
\end{equation*}
while at the even-numbered sites $x=x_k=x_{2a}$ (where $n$ is supported) we get
\begin{equation*}
  \begin{split}
    \jump{\psi_1(x_k;z)} &= 2z n_k \psi_2(x_k), \\
    \jump{\psi_2(x_k;z)} &= 0, \\
    \jump{\psi_3(x_k;z)} &= 0.
  \end{split}
\end{equation*}
Upon expressing the left and right limits $\psi_i(x_k^{\pm})$ using \eqref{eq:psi-inbetween},
these jump conditions translate into linear equations relating $(A_k,B_k,C_k)$
to $(A_{k-1},B_{k-1},C_{k-1})$. Solving for $(A_k,B_k,C_k)$ yields
\begin{equation}
  \label{eq:jumpI-ABC}
  \begin{pmatrix} A_k \\ B_k \\ C_k \end{pmatrix} =
  S_k(-z^2)
  \begin{pmatrix} A_{k-1} \\ B_{k-1} \\ C_{k-1} \end{pmatrix},
\end{equation}
with the jump matrix $S_k$ (not to be confused with the transition matrix
$S(\lambda)$ defined by \eqref{eq:transition-matrices} and used in the main text)
defined by
\begin{equation}
  \label{eq:jumpI-even-odd}
  S_{k}(\lambda) =
  \begin{cases}
    \begin{pmatrix}
      1 & 0 & 0 \\
      m_{k} e^{x_{k}} & 1 & \lambda m_{k} e^{-x_{k}} \\
      0 & 0 & 1
    \end{pmatrix},
    & k = 2a-1
    ,
    \\[5ex]
    \begin{pmatrix}
      1 & -2\lambda n_{k} e^{-x_{k}} & 0 \\
      0 & 1 & 0 \\
      0 &  2 n_{k} e^{x_{k}} & 1
    \end{pmatrix},
    & k = 2a
    ,
  \end{cases}
\end{equation}
for $a=1,\dots,K$.
Starting with $(A_0,B_0,C_0)=(1,0,0)$ we obtain in the rightmost
interval $x > x_N$ polynomials $(A_N,B_N,C_N)=(A(\lambda),B(\lambda),C(\lambda))$
in the variable $\lambda = -z^2$:
\begin{equation}
  \label{eq:ABC-jump-product}
  \begin{smallpmatrix} A(\lambda) \\ B(\lambda) \\ C(\lambda) \end{smallpmatrix}
  = S_{2K}(\lambda) S_{2K-1}(\lambda) \dotsm S_{2}(\lambda) S_{1}(\lambda)
  \begin{smallpmatrix} 1 \\ 0 \\ 0 \end{smallpmatrix}.
\end{equation}
Between the factors $S_{2K}(\lambda)$ and $S_{1}(\lambda)$ in the matrix
product there are $K-1$ pairs of factors of the form
\begin{multline}
  S_{2a+1}(\lambda) S_{2a}(\lambda) =
  \begin{pmatrix}
    1 & 0 & 0 \\
    m_{2a+1} e^{x_{2a+1}} & 1 & 0 \\
    0 & 2 n_{2a} e^{x_{2a}} & 1
  \end{pmatrix}
  +
  \\
  \lambda
  \begin{pmatrix}
    0 & -2 n_{2a} e^{-x_{2a}} & 0 \\
    0 & 2 m_{2a+1} n_{2a} (e^{x_{2a}-x_{2a+1}} - e^{x_{2a+1}-x_{2a}}) & m_{2a+1} e^{-x_{2a+1}} \\
    0 & 0 & 0
  \end{pmatrix}
  ,
\end{multline}
each such pair depending linearly on~$\lambda$.
The factor $S_1(\lambda) \, (1,0,0)^T$
does not depend on~$\lambda$, and therefore the vector
$\bigl( A(\lambda),B(\lambda),C(\lambda) \bigr)^T$
equals $S_{2K}(\lambda)$ times a vector whose entries have degree $K-1$ in~$\lambda$.
Since $\lambda$ only appears in the top row of $S_{2K}(\lambda)$,
we see that $B(\lambda)$ and $C(\lambda)$ are polynomials of degree $K-1$,
while $A(\lambda)$ is of degree~$K$.
We will name the coefficients in these polynomials as follows:
\begin{equation}
  \label{eq:ABC}
  \begin{split}
    A(\lambda) &= 1 - 2\lambda [A]_1 + \dots + (-2\lambda)^{K} [A]_{K},\\
    B(\lambda) &= [B]_0 - 2\lambda [B]_1 + \dots + (-2\lambda)^{K-1} [B]_{K-1},\\
    C(\lambda) &= [C]_0 - 2\lambda [C]_1 + \dots + (-2\lambda)^{K-1} [C]_{K-1}.\\
  \end{split}
\end{equation}
These coefficients can be computed explicitly in terms of the
positions $x_k$ and the weights $m_k$ and~$n_k$ by carefully studying
what happens when multiplying out the matrix product
$S_{2K}(\lambda) S_{2K-1}(\lambda) \dotsm S_{2}(\lambda) S_{1}(\lambda) \, (1,0,0)^T$.
For example, using the abbreviation
\begin{equation}
  \label{eq:Eab}
  E_{ab}=e^{-\abs{x_a-x_b}}
  \qquad\text{($= e^{x_a-x_b}$ when $a<b$)}
\end{equation}
we have
\begin{equation}
  \label{eq:ABC-somecoeffs}
  \begin{split}
    [A]_1 &= \sum_{1 \le i<j \le N} m_i n_j E_{ij}
    ,\\
    [A]_{K} &= m_1 n_2 E_{12} \, (1-E_{23}^2) \, m_3 n_4 E_{34} \, (1-E_{45}^2) \, m_5 n_6 E_{56} \dotsm \\
    & \qquad \dotsm (1-E_{N-2,N-1}^2) \, m_{N-1} n_{N} E_{N-1,N},\\
    [B]_0 &= B(0) = \sum_{1 \le i < N} m_i e^{x_i}
    .
  \end{split}
\end{equation}
(Recall that $N=2K$. Note also that since $m_{2a}$ and $n_{2a-1}$ are zero,
only the terms with $i$~odd and $j$~even contribute to the sums.)
Later we will show a simpler way to read off all the coefficients
in~$A(\lambda)$; see \eqref{eq:Ak-as-sum-of-minors-I} in
\autoref{app:coeffs-via-planar-networks}.

For the second Lax equation \eqref{eq:laxII-x} things are similar,
except that the roles of $m$ and~$n$ are swapped. This leads to
\begin{equation}
  \label{eq:psi-twin-inbetween}
    \begin{pmatrix} \twin\psi_1(x;z) \\ \twin\psi_2(x;z) \\ \twin\psi_3(x;z) \end{pmatrix} =
    \begin{pmatrix}
      \twin A_k e^x + z^2 \twin C_k e^{-x} \\ 2z \twin B_k \\ \twin A_k e^x - z^2 \twin C_k e^{-x}
    \end{pmatrix},
    \qquad
    x_k < x < x_{k+1},
\end{equation}
and
\begin{equation}
  \label{eq:jumpII-ABC}
  \begin{pmatrix} \twin A_k \\ \twin B_k \\ \twin C_k \end{pmatrix} =
  \twin S_k(-z^2)
  \begin{pmatrix} \twin A_{k-1} \\ \twin B_{k-1} \\ \twin C_{k-1} \end{pmatrix},
\end{equation}
where
\begin{equation}
  \label{eq:jumpII-even-odd}
  \twin S_{k}(\lambda) =
  \begin{cases}
    \begin{pmatrix}
      1 & -2\lambda m_{k} e^{-x_{k}} & 0 \\
      0 & 1 & 0 \\
      0 &  2 m_{k} e^{x_{k}} & 1
    \end{pmatrix},
    & k = 2a-1,
    \\
    & \\
    \begin{pmatrix}
      1 & 0 & 0 \\
      n_{k} e^{x_{k}} & 1 & \lambda n_{k} e^{-x_{k}} \\
      0 & 0 & 1
    \end{pmatrix},
    & k = 2a.
  \end{cases}
\end{equation}
Starting again with $(\twin A_0, \twin B_0, \twin C_0)=(1,0,0)$,
we have in the rightmost interval
$(\twin A_N, \twin B_N, \twin C_N)=(\twin A(\lambda), \twin B(\lambda), \twin C(\lambda))$, where
\begin{equation}
  \label{eq:twin-ABC-jump-product}
  \begin{smallpmatrix} \twin A(\lambda) \\ \twin B(\lambda) \\ \twin C(\lambda) \end{smallpmatrix}
  = \twin S_{2K}(\lambda) \twin S_{2K-1}(\lambda) \dotsm \twin S_{2}(\lambda) \twin S_{1}(\lambda)
  \begin{smallpmatrix} 1 \\ 0 \\ 0 \end{smallpmatrix}.
\end{equation}
Because of the asymmetry between $m$ and~$n$ coming from the interlacing,
the variable $\lambda$ appears in a slightly different way here;
in this case we have $K$ pairs of factors $\twin S_{2a} \twin S_{2a-1}$,
each of a similar form as the pairs $S_{2a+1} S_{2a}$ that we computed earlier:
\begin{multline}
  \twin S_{2a}(\lambda) \twin S_{2a-1}(\lambda) =
  \begin{pmatrix}
    1 & 0 & 0 \\
    n_{2a} e^{x_{2a}} & 1 & 0 \\
    0 & 2 m_{2a-1} e^{x_{2a-1}} & 1
  \end{pmatrix}
  +
  \\
  \lambda
  \begin{pmatrix}
    0 & -2 m_{2a-1} e^{-x_{2a-1}} & 0 \\
    0 & 2 n_{2a} m_{2a-1} (e^{x_{2a-1}-x_{2a}} - e^{x_{2a}-x_{2a-1}}) & n_{2a} e^{-x_{2a}} \\
    0 & 0 & 0
  \end{pmatrix}
  .
\end{multline}
From this we see that
$\twin S_{2} \twin S_{1} \, (1,0,0)^T$
is independent of~$\lambda$,
so that the degrees of $\twin A(\lambda)$, $\twin B(\lambda)$ and $\twin C(\lambda)$
are at most $K-1$.
The leftmost pair $\twin S_{2K} \twin S_{2K-1}$ has no $\lambda$ in its
bottom row, so $\twin C(\lambda)$ is in fact only of degree $K-2$.
Naming the coefficients as
\begin{equation}
  \label{eq:ABC-twin}
  \begin{split}
    \twin A(\lambda) &= 1 - 2\lambda [\twin A]_1 + \dots + (-2\lambda)^{K-1} [\twin A]_{K-1},\\
    \twin B(\lambda) &= [B]_0 - 2\lambda [\twin B]_1 + \dots + (-2\lambda)^{K-1} [\twin B]_{K-1},\\
    \twin C(\lambda) &= [C]_0 - 2\lambda [\twin C]_1 + \dots + (-2\lambda)^{K-2} [\twin C]_{K-2},\\
  \end{split}
\end{equation}
we have for example
\begin{equation}
  \label{eq:ABC-twin-somecoeffs}
  \begin{split}
    [\twin A]_1 &= \sum_{1 < j < i < N} n_j m_i E_{ji},\\
    [\twin A]_{K-1} &= n_2 m_3 E_{23} \, (1-E_{34}^2) \, n_4 m_5 E_{45} \, (1-E_{56}^2) \, n_6 m_7 E_{67} \dotsm \\
    & \qquad \dotsm (1-E_{N-3,N-2}^2) \, n_{N-2} m_{N-1} E_{N-2,N-1},\\
    [\twin B]_0 &= \twin B(0) = \sum_{1 < j \le N} n_j e^{x_j},\\
    [\twin B]_{K-1} &= [\twin A]_{K-1} \, n_N e^{x_N} (1-E_{N-1,N}^2)
    ,
  \end{split}
\end{equation}
where (as in \eqref{eq:ABC-somecoeffs})
only terms with $i$~odd and $j$~even contribute to the sums.
See \eqref{eq:Ak-as-sum-of-minors-II} in \autoref{app:coeffs-via-planar-networks} for
an easy way to read off all the coefficients of~$\twin A(\lambda)$.

\subsection{Positivity and simplicity of the spectra}
\label{app:spectrum}

By construction, the zeros of $A(\lambda)$ and $\twin A(\lambda)$
are exactly the nonzero eigenvalues $\lambda_1, \dots, \lambda_K$ and~$\mu_1, \dots, \mu_{K-1}$
treated in the main text,
so the following theorem implies \autoref{thm:simple-spectra}:

\begin{theorem}
  \label{thm:positive-simple-spectrum}
  If all nonzero weights $m_{2a-1}$ and $n_{2a}$ are positive,
  then the polynomials $A(\lambda)$ and $\twin A(\lambda)$
  have positive simple zeros
  $\lambda_1,\dots,\lambda_K$ and $\mu_1,\dots,\mu_{K-1}$,
  respectively.
\end{theorem}

\begin{proof}
  We will rewrite the two spectral problems,
  \eqref{eq:laxI-x} with boundary conditions
  $B_0 = C_0 = 0 = A_N$,
  and its twin 
  \eqref{eq:laxII-x} with boundary conditions
  $\twin B_0 = \twin C_0 = 0 = \twin A_N$,
  as matrix eigenvalue problems.
  (Recall that $A$ and $\twin A$ are just aliases for $A_N$ and $\twin A_N$,
  respectively.)

  For the first problem,
  elimination of $\psi_1$ from \eqref{eq:laxI-x} gives
  $\partial_x \psi_2 = zm\psi_3$ and $(\partial_x^2 - 1)\psi_3 = zn\psi_2$,
  which, considering the boundary conditions in the form \eqref{eq:boundary-conditions-x},
  we can write as
  \begin{equation}
    \label{eq:integral-eqn-I}
    \psi_2(x) = z \int_{-\infty}^x \psi_3(y) \, dm(y),
    \quad
    \psi_3(x) = -z \int_{-\infty}^{\infty} \frac12 e^{-\abs{x-y}} \psi_2(y) \, dn(y).
  \end{equation}
  Evaluating the first equation at the even-numbered~$x_k$
  and the second equation at the odd-numbered~$x_k$,
  we find the $2K \times 2K$ eigenvalue problem
  \begin{equation}
    \label{eq:eigenvalue-problem}
    \begin{pmatrix} \psi_{2,\text{even}} \\ \psi_{3,\text{odd}} \end{pmatrix}
    = z
    \begin{pmatrix}
      0 & 2 (I+\mathcal{L}) \mathcal{M} \\
      -\mathcal{E} \mathcal{N} & 0
    \end{pmatrix}
    \begin{pmatrix} \psi_{2,\text{even}} \\ \psi_{3,\text{odd}} \end{pmatrix},
  \end{equation}
  where
  \begin{equation*}
    \begin{split}
      \psi_{2,\text{even}} &= \bigl( \psi_2(x_2), \psi_2(x_4), \dots, \psi_2(x_{2K}) \bigr)^T,\\
      \psi_{3,\text{odd}} &= \bigl( \psi_3(x_1), \psi_3(x_3), \dots, \psi_3(x_{2K-1}) \bigr)^T,\\
      I &= \text{$K \times K$ identity matrix}, \\
      \mathcal{L} &= \text{strictly lower triangular $K \times K$ matrix with $\mathcal{L}_{ij}=1$ for $i>j$}, \\
      \mathcal{E} &= (e^{-\abs{x_{2i-1}-x_{2j}}})_{i,j=1}^K = (E_{2i-1,2j})_{i,j=1}^K \quad \text{(using the notation of \eqref{eq:Eab})},\\
      \mathcal{M} &= \diag(m_1,m_3,\dots,m_{2K-1}),\\
      \mathcal{N} &= \diag(n_2,n_4,\dots,n_{2K}).
    \end{split}
  \end{equation*}
  Eliminating $\psi_{3,\text{odd}}$ we can write this as a $K \times K$ eigenvalue problem
  in terms of $\psi_{2,\text{even}}$ alone:
  \begin{equation}
    \label{eq:matrixI}
    \psi_{2,\text{even}} =
    2 \lambda \, (I+\mathcal{L}) \mathcal{M} \mathcal{E} \mathcal{N} \, \psi_{2,\text{even}}
    \qquad (\lambda = -z^2).
  \end{equation}
  As we've seen earlier, the eigenvalues are given precisely by the zeros
  of $A(\lambda)$, and since $A(0)=1$ we must therefore have
  \begin{equation}
    \label{eq:A-as-determinant-I}
    A(\lambda) = \det(I - 2 \lambda \, (I+\mathcal{L}) \mathcal{M} \mathcal{E} \mathcal{N} ).
  \end{equation}
  Now, for positive numbers $\{ m_{2k-1}, n_{2k} \}_{k=1}^K$, the matrix
  $(I+\mathcal{L}) \mathcal{M} \mathcal{E} \mathcal{N}$
  is oscillatory, since $I+\mathcal{L}$ is nonsingular and
  totally nonnegative (being the path matrix for the planar network
  illustrated in \autoref{fig:network-for-I+L}),
  and since $\mathcal{M} \mathcal{E} \mathcal{N}$
  is totally positive ($\mathcal{E}$ being a submatrix of the totally
  positive matrix $(E_{ij})_{i,j=1}^{2K}$).
  This implies that its eigenvalues, which up to an unimportant
  factor of~$2$ are the zeros of~$A$, are positive and simple.
  (See, for example, our earlier papers
  \cite{lundmark-szmigielski:DPlong,hone-lundmark-szmigielski:novikov}
  for a summary of the relevant results from the theory of total
  positivity used here, and for further references.)

  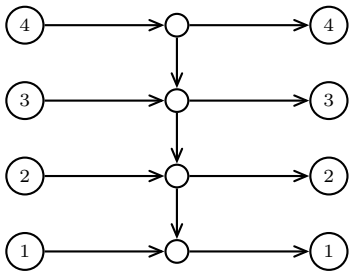
\begin{figure}
    \centering
    \ifthenelse{\isundefined{\draft}}{%
      \begin{tikzpicture}
        [
        source/.style={circle,draw=black,thick,inner sep=1mm},
        sink/.style={circle,draw=black,thick,inner sep=1mm},
        middle/.style={circle,draw=black,thick,inner sep=0mm, minimum size=3mm},
        arrow/.style={->,thick,>=angle 45}
        ]
        \foreach \y in {1,2,3,4} {
          \node(a\y) at (0,\y) [source] {\scriptsize \y};
          \node(d\y) at (4,\y) [sink] {\scriptsize \y};
          \node(b\y) at (2,\y) [middle] {};
        }
        \foreach \y in {1,2,3,4} {
          \draw[arrow] (a\y) -- (b\y);
          \draw[arrow] (b\y) -- (d\y);
        }
        
        \foreach \y/\z in {b2/b1,b3/b2,b4/b3}
        \draw[arrow] (\y.south) -- (\z);
        \pgfsetbaseline{12mm}
      \end{tikzpicture}
    }{\textbf{[Draft mode; no picture]}}
    \caption{A planar network (illustrated in the case $K=4$)
      for which $I + \mathcal{L}$ is the path matrix.
      What this means is that matrix entry $(i,j)$ equals the number of paths from
      source node $i$ on the left to sink node $j$ on the right;
      in this case there is one path if $i \ge j$ and none if $i<j$.
    }
    \label{fig:network-for-I+L}
  \end{figure}

  For the second spectral problem we swap $m$ and~$n$ and obtain
  \begin{equation}
    \label{eq:integral-eqn-II}
    \twin\psi_2(x) = z \int_{-\infty}^x \twin\psi_3(y) \, dn(y),
    \quad
    \twin\psi_3(x) = -z \int_{-\infty}^{\infty} \frac12 e^{-\abs{x-y}} \twin\psi_2(y) \, dm(y).
  \end{equation}
  These integral equations are evaluated the other way around (at
  odd-numbered and even-numbered $x_k$, respectively); this yields
  \begin{equation}
    \label{eq:big-matrix-equation}
    \begin{pmatrix} \twin\psi_{2,\text{odd}} \\ \twin\psi_{3,\text{even}} \end{pmatrix}
    = z
    \begin{pmatrix}
      0 & 2 \mathcal{L} \mathcal{N} \\
      - \mathcal{E}^T \mathcal{M} & 0
    \end{pmatrix}
    \begin{pmatrix} \twin\psi_{2,\text{odd}} \\ \twin\psi_{3,\text{even}} \end{pmatrix},
  \end{equation}
  which in terms of $\twin\psi_{2,\text{odd}}$ alone becomes
  \begin{equation}
    \label{eq:matrixII}
    \twin\psi_{2,\text{odd}} =
    2 \lambda \, \mathcal{L} \mathcal{N} \mathcal{E}^T \mathcal{M} \, \twin\psi_{2,\text{odd}}
    \qquad (\lambda = -z^2).
  \end{equation}
  Thus,
  \begin{equation}
    \label{eq:A-as-determinant-II}
    \twin A(\lambda) = \det(I - 2 \lambda \, \mathcal{L} \mathcal{N} \mathcal{E}^T \mathcal{M} ).
  \end{equation}
  The previous argument doesn't quite work for dealing with the zeros of $\twin A(\lambda)$,
  since $\mathcal{L}$ is singular and one therefore cannot draw the conclusion that
  the matrix $\mathcal{L} \mathcal{N} \mathcal{E}^T \mathcal{M}$ appearing in
  \eqref{eq:A-as-determinant-II} is oscillatory
  (only totally nonnegative, which is not enough to show simplicity of the zeros).
  However, a slightly modified argument does the trick.
  Note that the first row and the last column of the $2K \times 2K$ matrix
  in \eqref{eq:big-matrix-equation} are zero.
  Thus $\twin\psi_2(x_1)=0$ if \eqref{eq:big-matrix-equation} is satisfied,
  and the value of $\twin\psi_3(x_{2K})$ doesn't really enter into the problem either
  (it appears only in the left-hand side, and is automatically determined by all
  the other quantities in the equation).
  Therefore \eqref{eq:big-matrix-equation} has nontrivial solutions
  if and only if there are nontrivial solutions to
  the following truncated $(2K-2) \times (2K-2)$ problem
  obtained by removing the masses $m_1$ and~$n_{2K}$
  (i.e., by deleting the first and last row and the first and last column):
  \begin{equation}
    \label{eq:truncated-big-matrix-equation}
    \begin{pmatrix} \twin\psi'_{2,\text{odd}} \\ \twin\psi'_{3,\text{even}} \end{pmatrix}
    = z
    \begin{pmatrix}
      0 & 2 (I'+\mathcal{L}') \mathcal{N}' \\
      - (\mathcal{E}')^T \mathcal{M}' & 0
    \end{pmatrix}
    \begin{pmatrix} \twin\psi'_{2,\text{odd}} \\ \twin\psi'_{3,\text{even}} \end{pmatrix},
  \end{equation}
  where
  \begin{equation*}
    \begin{split}
      \twin\psi'_{2,\text{odd}} &= \bigl( \twin\psi_2(x_3), \twin\psi_2(x_5), \dots, \twin\psi_2(x_{2K-1}) \bigr)^T,\\
      \twin\psi'_{3,\text{even}} &= \bigl( \twin\psi_3(x_2), \twin\psi_3(x_4), \dots, \twin\psi_3(x_{2K-2}) \bigr)^T,\\
      I' &= \text{$(K-1) \times (K-1)$ identity matrix}, \\
      \mathcal{L}' &= \text{strictly lower triangular $(K-1) \times (K-1)$ with $\mathcal{L}'_{ij}=1$ for $i>j$}, \\
      \mathcal{E}' &= \text{$\mathcal{E}$ with its first row and last column removed}, \\
      \mathcal{M}' &= \diag(m_3,m_5,\dots,m_{2K-1}),\\
      \mathcal{N}' &= \diag(n_2,n_4,\dots,n_{2K-2}).
    \end{split}
  \end{equation*}
  In terms of $\twin\psi'_{2,\text{odd}}$ alone, this becomes
  \begin{equation*}
    \twin\psi'_{2,\text{odd}} = 2\lambda \, (I'+\mathcal{L}') \mathcal{N}' (\mathcal{E}')^T \mathcal{M}' \twin\psi'_{2,\text{odd}},
  \end{equation*}
  and the conclusion is that
  \begin{equation}
    \label{eq:A-as-determinant-II-but-smaller}
    A(\lambda) = \det(I' - 2\lambda \, (I+\mathcal{L}') \mathcal{N}' (\mathcal{E}')^T \mathcal{M}' \twin\psi'_{2,\text{odd}}),
  \end{equation}
  where $(I'+\mathcal{L}') \mathcal{N}' (\mathcal{E}')^T \mathcal{M}'$ is an oscillatory
  $(K-1) \times (K-1)$ matrix (by the previous argument).
  This shows that $\twin A(\lambda)$ has positive simple zeros too.
\end{proof}

\subsection{Expressions for the coefficients of $A$ and $\twin A$}
\label{app:coeffs-via-planar-networks}

From equations \eqref{eq:A-as-determinant-I} and \eqref{eq:A-as-determinant-II}
we can extract nice and fairly explicit representations of the coefficients
of the polynomials $A(\lambda)$ and $\twin A(\lambda)$.
These coefficients are of particular interest, since they turn out to be
constants of motion for the peakon solutions to the Geng--Xue equation.
(It is not hard to show, using the Lax pairs, that
$A(\lambda)$ and $\twin A(\lambda)$ are independent of time;
the details will be published in a separate paper about peakons.)

First a bit of notation: $\binom{S}{k}$ will denote the set of $k$-element subsets of a set~$S$,
and $[K]$ is the set $\{ 1,2,\dots,K \}$.
For a matrix~$X$ and index sets $I = \{ i_1 < \dots < i_m \}$ and
$J = \{ j_1 < \dots < j_n \}$,
we write $X_{IJ}$ for the submatrix obtained from~$X$ by taking elements from
the rows indexed by~$I$ and the columns indexed by~$J$;
in other words, $X_{IJ} = \bigl( X_{i_a j_b} \bigl)_{\substack{a=1,\dots,m \\ b=1,\dots,n}}$.

To begin with, \eqref{eq:A-as-determinant-I} says that
$A(\lambda) = \det(I - 2 \lambda \, (I + \mathcal{L}) \mathcal{M} \mathcal{E} \mathcal{N})$,
which shows that the quantity $[A]_k$ from \eqref{eq:ABC}
(the coefficient of $(-2\lambda)^k$ in $A(\lambda)$) equals
the sum of the principal $k \times k$ minors in
$(I + \mathcal{L}) \mathcal{M} \mathcal{E} \mathcal{N}$:
\begin{equation}
  \label{eq:Ak-as-preliminary-sum-of-minors}
  [A]_k = \sum_{J \in \binom{[K]}{k}} \det \bigl( (I + \mathcal{L}) \mathcal{M} \mathcal{E} \mathcal{N} \bigr)_{JJ}.
\end{equation}
A general fact is that for any $K \times K$ matrix~$X$
and for any fixed $J \in \binom{[K]}{k}$,
we have the identity
\begin{equation}
  \label{eq:minorsum-I}
  \det\bigl( (I+\mathcal{L}) X \bigr)_{JJ} =
  \sum_{\substack{I \in \binom{[K]}{k} \\ I \lel J }} \det X_{IJ},
\end{equation}
with summation over all index sets~$I$ of size~$k$
that are ``half-strictly interlacing'' with~$J$:
\begin{equation}
  \label{eq:half-interlacing-I}
  I \lel J
  \quad
  \Longleftrightarrow
  \quad
  i_1 \le j_1 < i_2 \le j_2 < \dots < i_k \le j_k.
\end{equation}
(This is similar to, but much simpler than, the ``Canada Day
Theorem'' about certain sums of minors of \emph{symmetric} matrices,
which appeared in the context of Novikov peakons
\cite{hone-lundmark-szmigielski:novikov,gomez-lundmark-szmigielski:CDT}.)
Equation \eqref{eq:minorsum-I}
can be proved by expanding $\det((I + \mathcal{L}) X)_{JJ}$ with the Cauchy--Binet formula
and computing the minors of $I + \mathcal{L}$ by applying the Lindström--Gessel--Viennot Lemma
to the planar network for $I + \mathcal{L}$  in \autoref{fig:network-for-I+L}.
(We briefly recall the statement of this lemma: if $X$ is the (weighted) path matrix of a planar
network $G$, then the minor $\det X_{IJ}$ equals the number of vertex-disjoint path families
(or the weighted sum over such families) connecting the sources indexed by~$I$
to the sinks indexed by~$J$.)
Alternatively, one can do row operations directly, as follows:
\begin{equation*}
  \begin{split}
    \det((I + \mathcal{L}) X)_{JJ}
    &= \det\Bigl( ((I + \mathcal{L}) X)_{j_r j_s} \Bigr)_{r,s=1}^k
    \\
    &= \det\Bigl( \sum_{m=1}^{j_r} X_{mj_s} \Bigr)_{r,s=1}^k
    \\
    &= \det\Bigl( \sum_{m=j_{r-1} + 1}^{j_r} X_{mj_s} \Bigr)_{r,s=1}^k
    \\
    &= \sum_{i_1=1}^{j_1} \sum_{i_2=j_1+1}^{j_2} \dots \sum_{i_k=j_{k-1} + 1}^{j_k} \det\Bigl( X_{i_r j_s} \Bigr)_{r,s=1}^k \\
    &= \sum_{\substack{\phantom{j_0 <} i_1 \le j_1 \\ j_1 < i_2 \le j_2 \\ j_2 < i_3 \le j_3 \\[1ex] \cdots \\[1ex] j_{k-1} < i_k \le j_k}} \det X_{IJ}.
  \end{split}
\end{equation*}
(In the second line, we used the definition of $\mathcal{L}$.
In the third line, we have subtracted from each row the row above it;
$j_0=0$ by definition.
Next, the summation index is renamed from $m$ to $i_r$ in row~$r$;
this lets us use multilinearity to bring the sums outside of the determinant.)
Applying this fact to \eqref{eq:Ak-as-preliminary-sum-of-minors},
we obtain the desired representation
\begin{equation}
  \label{eq:Ak-as-sum-of-minors-I}
  [A]_k = \sum_{\substack{I,J \in \binom{[K]}{k} \\ I \lel J }} \det (\mathcal{N} \mathcal{E} \mathcal{M})_{IJ},
\end{equation}
This is useful, since these determinants can be evaluated using the Lindström--Gessel--Viennot Lemma
on the planar network shown in \autoref{fig:network-for-A-coeffs};
see \autoref{ex:integrals-of-motion} below.

\begin{figure}
  \centering
  \ifthenelse{\isundefined{\draft}}{%
    \begin{tikzpicture}
      [
      source/.style={circle,draw=black,thick,inner sep=1mm},
      sink/.style={circle,draw=black,thick,inner sep=1mm},
      middle/.style={circle,draw=black,thick,inner sep=0mm, minimum size=3mm},
      arrow/.style={->,thick,>=angle 45}
      ]
      \foreach \y in {1,2,3,4} {
        \node(a\y) at (0,2*\y-1) [source] {\small $\y$};
        \node(d\y) at (8,2*\y) [sink] {\small $\y$};
      }
      \foreach \y in {1,2,...,8} {
        \node(b\y) at (3,\y) [middle] {};
        \node(c\y) at (5,\y) [middle] {};
      }
      \small
      \foreach \y in {1,2,3,4} {
        \pgfmathparse{int(2*\y-1)};
        \let\z\pgfmathresult;
        \draw[arrow] (a\y) -- node[above] {$m_{\z}$} (b\z);
        \pgfmathparse{int(2*\y)};
        \let\z\pgfmathresult;
        \draw[arrow] (c\z) -- node[above] {$n_{\z}$} (d\y);
      }
      \foreach \y in {1,2,...,7} {
        \pgfmathparse{int(\y+1)};
        \let\z\pgfmathresult;
        \draw[arrow] (b\z) -- node[left] {$E_{\y\z}$} (b\y);
        \draw[arrow] (c\y) -- node[right] {$E_{\y\z}$} (c\z);
        \draw[arrow] (b\z) -- node[above] {$1-E_{\y\z}^2$} (c\z);
      }
      \draw[arrow] (b1) -- node[above] {$1$} (c1);
    \end{tikzpicture}
  }{\textbf{[Draft mode; no picture]}}
  \caption{A weighted planar network (illustrated in the case $K=4$)
    for which $\mathcal{M} \mathcal{E} \mathcal{N}$ is the weighted
    path matrix; this means that the $(i,j)$ entry is the weighted sum
    of all paths from source~$i$ to sink~$j$, each path being counted
    with a weight equal to the product of its edge weights. }
  \label{fig:network-for-A-coeffs}
\end{figure}
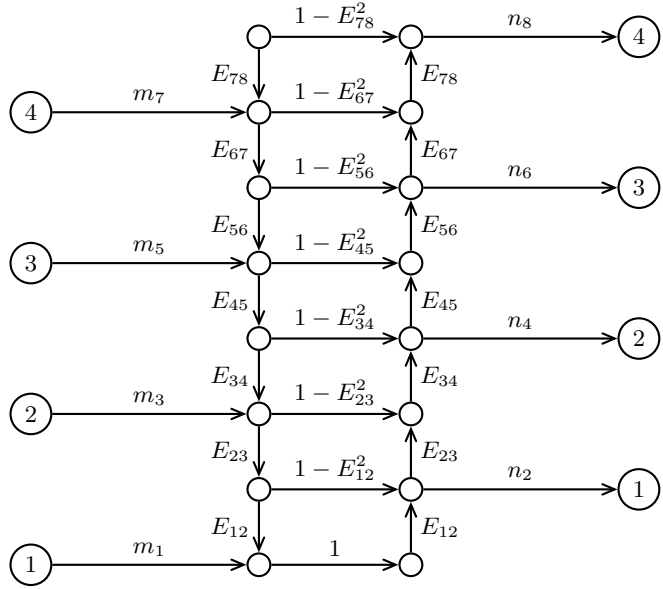

In an entirely similar way one derives the identity
\begin{equation}
  \label{eq:minorsum-II}
  \det(\mathcal{L}X)_{JJ} =
  \sum_{\substack{I \in \binom{[K]}{k} \\ I \lle J }} \det X_{IJ},
\end{equation}
with the other type of ``half-strictly interlacing'' relation
\begin{equation}
  \label{eq:half-interlacing-II}
  I \lle J
  \quad
  \Longleftrightarrow
  \quad
  i_1 < j_1 \le i_2 < j_2 \le \dots \le i_k < j_k.
\end{equation}
Indeed, we can use the network for $\mathcal{L}$ in
\autoref{fig:network-for-L}, or do row operations:
\begin{equation*}
  \begin{split}
    \det(\mathcal{L}X)_{JJ}
    &= \det\Bigl( (\mathcal{L}X)_{j_r j_s} \Bigr)_{r,s=1}^k
    \\
    &= \det\Bigl( \sum_{m=1}^{j_r-1} X_{mj_s} \Bigr)_{r,s=1}^k 
    \\
    &= \det\Bigl( \sum_{m=j_{r-1}}^{j_r-1} X_{mj_s} \Bigr)_{r,s=1}^k
    \\
    &= \sum_{i_1=1}^{j_1-1} \sum_{i_2=j_1}^{j_2-1} \dots \sum_{i_k=j_{k-1}}^{j_k-1} \det\Bigl( X_{i_r j_s} \Bigr)_{r,s=1}^k \\
    &= \sum_{\substack{\phantom{j_0 <} i_1 < j_1 \\ j_1 \le i_2 < j_2 \\ j_2 \le i_3 < j_3 \\[1ex] \cdots \\[1ex] j_{k-1} \le i_k < j_k}} \det X_{IJ}.
  \end{split}
\end{equation*}
From \eqref{eq:A-as-determinant-II} we then see that the
coefficients defined by \eqref{eq:ABC-twin} are given by
\begin{equation*}
  [\twin A]_k = \sum_{J \in \binom{[K]}{k}} \det \bigl( \mathcal{L} \mathcal{N} \mathcal{E}^T \mathcal{M} \bigr)_{JJ},
\end{equation*}
which by the identity above amounts to
\begin{equation}
  \label{eq:Ak-as-sum-of-minors-II}
  [\twin A]_k
  = \sum_{\substack{I,J \in \binom{[K]}{k} \\ I \lle J }} \det ( \mathcal{N} \mathcal{E}^T \mathcal{M} )_{IJ}
  = \sum_{\substack{I,J \in \binom{[K]}{k} \\ I \lle J }} \det ( \mathcal{M} \mathcal{E} \mathcal{N} )_{JI}
  .
\end{equation}
Again, we can use the network for
$\mathcal{M} \mathcal{E} \mathcal{N}$ in \autoref{fig:network-for-A-coeffs}
to read off these determinants
(but note that the transposition in the last step of \eqref{eq:Ak-as-sum-of-minors-II}
has the effect that now the sources are index by~$J$ and the sinks by~$I$).

\begin{figure}
  \centering
  \ifthenelse{\isundefined{\draft}}{%
    \begin{tikzpicture}
      [
      source/.style={circle,draw=black,thick,inner sep=1mm},
      sink/.style={circle,draw=black,thick,inner sep=1mm},
      middle/.style={circle,draw=black,thick,inner sep=0mm, minimum size=3mm},
      arrow/.style={->,thick,>=angle 45}
      ]
      \foreach \y in {1,2,3,4} {
        \node(a\y) at (0,\y) [source] {\scriptsize \y};
        \node(d\y) at (4,\y) [sink] {\scriptsize \y};
       }
      \foreach \y in {2,3,4} {
        \node(b\y) at (2,\y-0.5) [middle] {};
       }
      \foreach \y in {2,3,4} {
        \draw[arrow] (a\y) -- (b\y);
      }
      \foreach \y/\z in {b3/b2,b4/b3}
      \draw[arrow] (\y.south) -- (\z);
      \foreach \y/\z in {b2/d1,b3/d2,b4/d3}
      \draw[arrow] (\y) -- (\z);
    \end{tikzpicture}
  }{\textbf{[Draft mode; no picture]}}
  \caption{A planar network (illustrated in the case $K=4$)
    for which $\mathcal{L}$ is the path matrix.
    There is one path from source node~$i$ to sink node~$j$ if $i>j$
    and none otherwise.
  }
  \label{fig:network-for-L}
\end{figure}
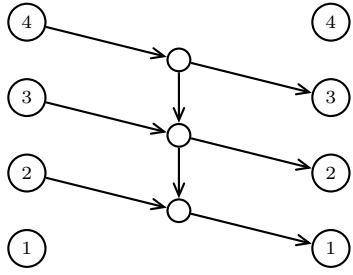

\begin{example}
  \label{ex:integrals-of-motion}
  Consider the case $K=4$.
  We compute the coefficients $[A]_k$ using \eqref{eq:Ak-as-sum-of-minors-I}
  and the planar network in \autoref{fig:network-for-A-coeffs}.

  Note identities such as
  $E_{12} E_{23} E_{34} = E_{14}$
  and
  $(1-E_{23}^2) + E_{23} \cdot (1-E_{12}^2) \cdot E_{23} = 1-E_{13}^2$,
  which are used repeatedly when computing path weights.
  For example, the determinant corresponding to $I = \{ 1,3 \}$ and $J = \{ 1,4 \}$
  is found by locating all pairs of path connecting source~1 to sink~1
  and source~3 to sink~4 in \autoref{fig:network-for-A-coeffs},
  and having no vertices in common.
  There is only one path $1 \to 1$,
  and its weight is
  $m_1 \cdot 1 \cdot E_{12} \cdot n_2$.
  Then there are three paths $3 \to 4$ not touching this first path,
  and their weights are
  \begin{gather*}
    m_5 \cdot (1-E_{45}^2) \cdot E_{56} \cdot E_{67} \cdot E_{78} \cdot n_8
    , \\
    m_5 \cdot E_{45} \cdot (1-E_{34}^2) \cdot E_{45} \cdot E_{56} \cdot E_{67} \cdot E_{78} \cdot n_8
    , \\
    m_5 \cdot E_{45} \cdot E_{34} \cdot (1-E_{23}^2) \cdot E_{34} \cdot E_{45} \cdot E_{56} \cdot E_{67} \cdot E_{78} \cdot n_8
    ,
  \end{gather*}
  or, in other words,
  \begin{gather*}
    m_5 (1-E_{45}^2) E_{58} n_8
    , \\
    m_5 (E_{45}^2-E_{35}^2) E_{58} n_8
    , \\
    m_5 (E_{35}^2-E_{25}^2) E_{58} n_8
    .
  \end{gather*}
  Multiplying each of these by the first weight $m_1 E_{12} n_2$
  gives the weights of the three vertex-disjoint path pairs $13 \to 14$,
  which we add up to obtain the determinant (according to the Lindström--Gessel--Viennot Lemma):
  \begin{equation*}
    \det (\mathcal{M} \mathcal{E} \mathcal{N})_{13,14} =
    m_1 E_{12} n_2 \cdot m_5 (1-E_{25}^2) E_{58} n_8
    .
  \end{equation*}

  The formulas for the coefficients $[A]_k$ found in this way are
  \begin{subequations}
    \begin{equation}
      \begin{split}
        [A]_1 &= \sum_{i \le j} \det (\mathcal{M} \mathcal{E} \mathcal{N})_{ij}
        = \sum_{i \le j} (\mathcal{M} \mathcal{E} \mathcal{N})_{ij}
        \\ &=
        m_1 E_{12} n_2 + m_1 E_{14} n_4 + m_1 E_{16} n_6 + m_1 E_{18} n_8 + m_3 E_{34} n_4
        \\ & \quad
        + m_3 E_{36} n_6 + m_3 E_{38} n_8 + m_5 E_{56} n_6
        + m_5 E_{58} n_8 + m_7 E_{78} n_8
        ,
      \end{split}
    \end{equation}
    \begin{equation}
      \begin{split}
        [A]_2 &= \sum_{i_1 \le j_1 < i_2 \le j_2}  \det (\mathcal{M} \mathcal{E} \mathcal{N})_{i_1 i_2, j_1 j_2}
        \\ &=
        \det (\mathcal{M} \mathcal{E} \mathcal{N})_{12,12}
        + \det (\mathcal{M} \mathcal{E} \mathcal{N})_{12,13}
        + \det (\mathcal{M} \mathcal{E} \mathcal{N})_{12,14}
        \\ & \quad
        + \det (\mathcal{M} \mathcal{E} \mathcal{N})_{13,13}
        + \det (\mathcal{M} \mathcal{E} \mathcal{N})_{13,14}
        + \det (\mathcal{M} \mathcal{E} \mathcal{N})_{14,14}
        \\ & \quad
        + \det (\mathcal{M} \mathcal{E} \mathcal{N})_{13,23}
        + \det (\mathcal{M} \mathcal{E} \mathcal{N})_{13,24}
        + \det (\mathcal{M} \mathcal{E} \mathcal{N})_{14,24}
        \\ & \quad
        + \det (\mathcal{M} \mathcal{E} \mathcal{N})_{14,34}
        + \det (\mathcal{M} \mathcal{E} \mathcal{N})_{23,23}
        + \det (\mathcal{M} \mathcal{E} \mathcal{N})_{23,24}
        \\ & \quad
        + \det (\mathcal{M} \mathcal{E} \mathcal{N})_{24,24}
        + \det (\mathcal{M} \mathcal{E} \mathcal{N})_{24,34}
        + \det (\mathcal{M} \mathcal{E} \mathcal{N})_{34,34}
        \\ &=
        m_1 E_{12} n_2 \cdot m_3 (1-E_{23}^2) E_{34} n_4  
        + m_1 E_{12} n_2 \cdot m_3 (1-E_{23}^2) E_{36} n_6  
        \\ & \quad
        + m_1 E_{12} n_2 \cdot m_3 (1-E_{23}^2) E_{38} n_8  
        + m_1 E_{12} n_2 \cdot m_5 (1-E_{25}^2) E_{56} n_6  
        \\ & \quad
        + m_1 E_{12} n_2 \cdot m_5 (1-E_{25}^2) E_{58} n_8  
        + m_1 E_{12} n_2 \cdot m_7 (1-E_{27}^2) E_{78} n_8  
        \\ & \quad
        + m_1 E_{14} n_4 \cdot m_5 (1-E_{45}^2) E_{56} n_6  
        + m_1 E_{14} n_4 \cdot m_5 (1-E_{45}^2) E_{58} n_8  
        \\ & \quad
        + m_1 E_{14} n_4 \cdot m_7 (1-E_{47}^2) E_{78} n_8  
        + m_1 E_{16} n_6 \cdot m_7 (1-E_{67}^2) E_{78} n_8  
        \\ & \quad
        + m_3 E_{34} n_4 \cdot m_5 (1-E_{45}^2) E_{56} n_6  
        + m_3 E_{34} n_4 \cdot m_5 (1-E_{45}^2) E_{58} n_8  
        \\ & \quad
        + m_3 E_{34} n_4 \cdot m_7 (1-E_{47}^2) E_{78} n_8  
        + m_3 E_{36} n_6 \cdot m_7 (1-E_{67}^2) E_{78} n_8  
        \\ & \quad
        + m_5 E_{56} n_6 \cdot m_7 (1-E_{67}^2) E_{78} n_8  
        ,
      \end{split}
    \end{equation}
    \begin{equation}
      \begin{split}
        [A]_3 &= \sum_{i_1 \le j_1 < i_2 \le j_2 < i_3 \le j_3}  \det (\mathcal{M} \mathcal{E} \mathcal{N})_{i_1 i_2 i_3, j_1 j_2 j_3}
        \\ &=
        \det (\mathcal{M} \mathcal{E} \mathcal{N})_{123,123}
        + \det (\mathcal{M} \mathcal{E} \mathcal{N})_{123,124}
        + \det (\mathcal{M} \mathcal{E} \mathcal{N})_{124,124}
        \\ &\quad
        + \det (\mathcal{M} \mathcal{E} \mathcal{N})_{124,134}
        + \det (\mathcal{M} \mathcal{E} \mathcal{N})_{134,134}
        + \det (\mathcal{M} \mathcal{E} \mathcal{N})_{134,234}
        \\ &\quad
        + \det (\mathcal{M} \mathcal{E} \mathcal{N})_{234,234}
        \\ &=
        m_1 E_{12} n_2 \cdot m_3 (1-E_{23}^2) E_{34} n_4  \cdot m_5 (1-E_{45}^2) E_{56} n_6  
        \\ &\quad
        + m_1 E_{12} n_2 \cdot m_3 (1-E_{23}^2) E_{34} n_4  \cdot m_5 (1-E_{45}^2) E_{58} n_8  
        \\ &\quad
        + m_1 E_{12} n_2 \cdot m_3 (1-E_{23}^2) E_{34} n_4  \cdot m_7 (1-E_{47}^2) E_{78} n_8  
        \\ &\quad
        + m_1 E_{12} n_2 \cdot m_3 (1-E_{23}^2) E_{36} n_6  \cdot m_7 (1-E_{67}^2) E_{78} n_8  
        \\ &\quad
        + m_1 E_{12} n_2 \cdot m_5 (1-E_{25}^2) E_{56} n_6  \cdot m_7 (1-E_{67}^2) E_{78} n_8  
        \\ &\quad
        + m_1 E_{14} n_4 \cdot m_5 (1-E_{45}^2) E_{56} n_6  \cdot m_7 (1-E_{67}^2) E_{78} n_8  
        \\ &\quad
        + m_3 E_{34} n_4 \cdot m_5 (1-E_{45}^2) E_{56} n_6  \cdot m_7 (1-E_{67}^2) E_{78} n_8  
        ,
      \end{split}
    \end{equation}
    and
    \begin{equation}
      \begin{split}
        [A]_4 &= \sum_{i_1 \le j_1 < i_2 \le j_2 < i_3 \le j_3 < i_4 \le j_4}  \det (\mathcal{M} \mathcal{E} \mathcal{N})_{i_1 i_2 i_3 i_4, j_1 j_2 j_3 j_4}
        \\ &=
        \det (\mathcal{M} \mathcal{E} \mathcal{N})_{1234,1234}
        \\ &=
        m_1 E_{12} n_2 \cdot m_3 (1-E_{23}^2) E_{34} n_4  \\ &\quad \cdot m_5 (1-E_{45}^2) E_{56} n_6 \cdot m_7 (1-E_{67}^2) E_{78} n_8
        .
      \end{split}
    \end{equation}
  \end{subequations}
  (Cf. the expressions \eqref{eq:ABC-somecoeffs} for the lowest and
  highest coefficients $[A]_1$ and $[A]_K$ in general.)

  Similarly, we compute the coefficients $[\twin A]_k$
  using \eqref{eq:Ak-as-sum-of-minors-II}:
  \begin{subequations}
    \begin{equation}
      \begin{split}
        [\twin A]_1 &= \sum_{i < j} \det(\mathcal{M} \mathcal{E} \mathcal{N})_{ji}
        = \sum_{i < j} (\mathcal{M} \mathcal{E} \mathcal{N})_{ji}
        \\ &=
        m_3 E_{23} n_2 + m_5 E_{25} n_2 + m_7 E_{27} n_2 \\ & \quad + m_5 E_{45} n_4 + m_7 E_{47} n_4 + m_7 E_{67} n_6
        ,
      \end{split}
    \end{equation}
    \begin{equation}
      \begin{split}
        [\twin A]_2 &= \sum_{i_1 < j_1 \le i_2 < j_2}  \det (\mathcal{M} \mathcal{E} \mathcal{N})_{j_1 j_2, i_1 i_2}
        \\ &=
        \det (\mathcal{M} \mathcal{E} \mathcal{N})_{23,12}
        + \det (\mathcal{M} \mathcal{E} \mathcal{N})_{24,12}
        + \det (\mathcal{M} \mathcal{E} \mathcal{N})_{24,13}
        \\ &\quad
        + \det (\mathcal{M} \mathcal{E} \mathcal{N})_{34,13}
        + \det (\mathcal{M} \mathcal{E} \mathcal{N})_{34,23}
        \\ &=
        m_3 E_{23} n_2 \cdot m_5  E_{45} (1-E_{34}^2) n_4
        \\ &\quad
        + m_3 E_{23} n_2 \cdot m_7  E_{47} (1-E_{34}^2) n_4
        \\ &\quad
        + m_3 E_{23} n_2 \cdot m_7  E_{67} (1-E_{36}^2) n_6
        \\ &\quad
        + m_5 E_{25} n_2 \cdot m_7  E_{67} (1-E_{56}^2) n_6
        \\ &\quad
        + m_5 E_{45} n_4 \cdot m_7  E_{67} (1-E_{56}^2) n_6
      \end{split}
    \end{equation}
    and
    \begin{equation}
      \begin{split}
        [\twin A]_3 &= \sum_{i_1 < j_1 \le i_2 < j_2 \le i_3 < j_3}  \det (\mathcal{M} \mathcal{E} \mathcal{N})_{j_1 j_2 j_3, i_1 i_2 i_3}
        \\ &=
        \det (\mathcal{M} \mathcal{E} \mathcal{N})_{234,123}
        \\ &=
        m_3 E_{23} n_2 \cdot m_5  E_{45} (1-E_{34}^2) n_4 \cdot m_7 E_{67} (1-E_{56}^2) n_6
        .
      \end{split}
    \end{equation}
  \end{subequations}
  (Cf. the expressions \eqref{eq:ABC-twin-somecoeffs} for $[\twin A]_1$ and $[\twin A]_{K-1}$ in
  general.)
\end{example}

\section{Guide to notation}
\label{app:notation}

For the convenience of the reader, here is an index of the notation used in this article.
\begin{center}
  \begin{longtable}{l|l}
    $m(x)$, $n(x)$, $z$, $\Psi(x;z) = (\psi_1,\psi_2,\psi_3)^T$ & \autoref{sec:intro} \\
    Spectral problems for $\Psi$ & \eqref{eq:laxI-x} + \eqref{eq:boundary-conditions-x} \\ & \eqref{eq:laxII-x} + \eqref{eq:boundary-conditions-x} \\
    $g(y)$, $h(y)$, $\lambda=-z^2$, $\Phi(y;\lambda) = (\phi_1,\phi_2,\phi_3)^T$ & \eqref{eq:liouville-trf} \\
    Spectral problems for $\Phi$ & \eqref{eq:spectral-problem-y}, \eqref{eq:spectral-problem-y-twin} \\
    Coefficient matrices $\coeffmatrix(y;\lambda)$, $\twin \coeffmatrix(y;\lambda)$ & \eqref{eq:coeffmatrices} \\
    $\J = \begin{smallpmatrix} 0 & 0 & 1 \\ 0 & -1 & 0 \\ 1 & 0 & 0 \end{smallpmatrix}$ & \eqref{eq:A-conjugation} \\
    Involution $X(\lambda)^\sigma = \J X(-\lambda)^{-T} \J$ & \eqref{eq:involution-sigma} \\
    Fundamental matrices $U(y;\lambda)$, $\twin U(y;\lambda)$ & \eqref{eq:fundsol}, \eqref{eq:fundsol-twin} \\
    Transition matrices & \\ \qquad $S(\lambda) = U(1;\lambda)$, $\twin S(\lambda) = \twin U(1;\lambda)$ & \eqref{eq:transition-matrices} \\
    Weyl functions & \\ \qquad $W = -S_{21}/S_{31}$, $Z = -S_{11}/S_{31}$ & \eqref{eq:WZ} \\
    Twin Weyl functions & \\ \qquad $\twin W = -\twin S_{21}/\twin S_{31}$, $Z = -\twin S_{11}/\twin S_{31}$ & \eqref{eq:WZ} \\
    Bilinear form $\bilinearRthree{\Phi}{\Omega} = \int_{-1}^1 \Phi(y)^T \J \, \Omega(y) \, dy$ & \eqref{eq:bilinear-R3} \\
    Adjoint spectral problems for & \\ \qquad $\Omega = \Omega(y;\lambda) = (\omega_1,\omega_2,\omega_3)$ & \eqref{eq:adjoint}, \eqref{eq:adjoint-twin} \\
    Adjoint Weyl functions & \\ \qquad $W^* = -S_{32}/S_{31}$, $Z^* = -S_{33}/S_{31}$ & \eqref{eq:WZ-adjoint} \\
    Twin adjoint Weyl functions & \\ \qquad $\twin W^* = -\twin S_{32}/\twin S_{31}$, $\twin Z^* = -\twin S_{33}/\twin S_{31}$ & \eqref{eq:WZ-adjoint} \\
    Discrete interlacing measures & \\ \qquad $m = m_1 \delta_{x_1} + m_3 \delta_{x_3} + \dots + m_{N-1} \delta_{x_{N-1}}$ & \\ \qquad $n = n_2 \delta_{x_2} + n_4 \delta_{x_4} + \dots + n_N \delta_{x_N}$ \\ \qquad with $x_1 < x_2 < \dots < x_N$, $N=2K$ & \eqref{eq:mn-interlacing} \\
    Transformed measures & \\ \qquad $g = g_1 \delta_{y_1} + g_2 \delta_{y_3} + \dots + g_K \delta_{y_{2K-1}}$ & \\ \qquad $h = h_1 \delta_{y_2} + h_2 \delta_{y_4} + \dots + h_K \delta_{y_{2K}}$ & \eqref{eq:gh} \\
    \qquad with $y_k = \tanh x_k$, & \eqref{eq:yk} \\ \qquad $g_a = 2 m_{2a-1} \cosh x_{2a-1}$, $h_a = 2 n_{2a} \cosh x_{2a}$ &
    \eqref{eq:trf-discrete-measures} \\
    Interval lengths $l_k = y_{k+1} - y_k$ & \eqref{eq:lk} \\
    Propagation matrices & \\ \qquad $L_k(\lambda) = \begin{smallpmatrix} 1 & 0 & 0 \\ 0 & 1 & 0 \\ -\lambda l_k & 0 & 1 \end{smallpmatrix}$,   $\displaystyle\jumpmatrix{x}{y} = \begin{smallpmatrix} 1 & x & \frac12 xy \\ 0 & 1 & y \\ 0 & 0 & 1 \end{smallpmatrix}$ & \eqref{eq:jump-matrix-L}, \eqref{eq:jump-matrix} \\
    Transition matrix in the discrete case & \\ \qquad $S(\lambda) = L_{2K}(\lambda) \jumpmatrix{h_K}{0} L_{2K-1}(\lambda) \jumpmatrix{0}{g_K} \dotsm$ & \eqref{eq:S-discrete} \\
    \qquad and its partial products $T_j(\lambda)$ & \eqref{eq:Tj} \\
    Twin transition matrix & \\ \qquad $\twin S(\lambda) = L_{2K}(\lambda) \jumpmatrix{0}{h_K} L_{2K-1}(\lambda) \jumpmatrix{g_K}{0} \dotsm$ & \eqref{eq:S-twin-discrete} \\
    \qquad and its partial products $\twin T_j(\lambda)$ & \eqref{eq:Tj-twin} \\
    Eigenvalues & \\ \qquad $0 = \lambda_0 < \lambda_1 < \dots < \lambda_K$ & \\ \qquad $0 = \mu_0 < \mu_1 < \dots < \mu_{K-1}$ & \autoref{thm:simple-spectra} \\
    Residues of Weyl functions & \\ \qquad $a_i$, $b_j$, $b_{\infty}$, $c_i$, $d_j$ ($1 \le i \le K$, $1 \le j \le K-1$)& \autoref{thm:weyl-parfrac} \\
    Spectral measures & \\ \qquad $\alpha = \sum_{i=1}^K a_i \delta_{\lambda_i}$, $\beta = \sum_{j=1}^{K-1} b_j \delta_{\mu_j}$ & \eqref{eq:spectral-measures} \\
    Weyl functions as integrals & \\ \qquad $W(\lambda) = \int \frac{\da}{\lambda - x}$, etc. & \eqref{eq:WZ-as-integrals} \\
    Entries of $T(\lambda)=T_j(\lambda)$ (for some fixed $j$) & \\ \qquad $Q = -T_{32}$, $P = T_{22}$, $R = T_{12}$ \\ \qquad ($W \approx P/Q$, $Z \approx R/Q$) & \eqref{eq:QPR} \\
    Residues of adjoint Weyl functions & \\ \qquad $a_i^*$, $b_j^*$, $b_{\infty}^*$, $c_i^*$, $d_j^*$ & \eqref{eq:weyl-star-parfrac} \\
    Adjoint transition matrix & \\ \qquad $S^*(\lambda) = \twin S(-\lambda)^{-1} = \J S(\lambda)^T \J$ & \autoref{thm:symmetry} \\
    Adjoint Weyl functions in terms of $S^*$ & \\ \qquad $W^* = +S^*_{21}/S^*_{31}$, $Z^* = -S^*_{11}/S^*_{31}$ & \autoref{rem:S-star} \\
    Moments and bimoments of spectral measures & \\ \qquad $\alpha_k = \int x^k \da$, $\beta_k = \int y^k \db$, & \eqref{eq:alpha-beta-moments}; see also \eqref{eq:alpha-beta-moments-appendix} \\ \qquad $I_{km} = \iint \frac{x^k y^m}{x+y} \da\db$ & \eqref{eq:alpha-beta-bimoments}; see also \eqref{eq:alpha-beta-bimoments-appendix} \\
    Determinant $\weirddeterminant_n$ involving bimoments & \eqref{eq:R-zero}; see also \eqref{eq:weirddeterminant} \\
    Spectral map & \\
    \qquad Pure peakon sector $\mathcal{P} \subset \R^{4K}$ & \\
    \qquad Admissible spectral data $\mathcal{R} \subset \R^{4K}$ & \\
    \qquad Forward map $\mathcal{S} \colon \mathcal{P} \to \mathcal{R}$ & \\
    \qquad Inverse map $\mathcal{T} \colon \mathcal{R} \to \mathcal{P}$ & \autoref{def:spectral-map} \\
    Cauchy biorthogonal polynomials $p_n(x)$, $q_n(y)$ & \eqref{eq:biorth-pn}, \eqref{eq:biorth-qn} \\
    Vandermonde-type expression & \\ \qquad $\Delta(x) = \Delta(x_1,\dots,x_n) = \prod_{i < j} (x_i - x_j)$ & \eqref{eq:def-Delta} \\ \qquad $\Gamma(x) = \Gamma(x_1,\dots,x_n) = \prod_{i < j} (x_i + x_j)$ & \eqref{eq:def-Gamma} \\ \qquad $\Gamma(x;y) = \Gamma(x_1,\dots,x_n;y_1,\dots,y_m)$ & \\ \qquad $\phantom{\Gamma(x;y)} = \prod_{i=1}^{n} \prod_{j=1}^{m} (x_i + y_j)$ & \eqref{eq:def-Gamma-mixed} \\
    Generalized Heine-type integrals & \\ \qquad $\heineintegral_{nm}^{rs} = \int\limits_{\sigma_{n} \times \sigma_{m}} \!\!\! \frac{\Delta(x)^2 \Delta(y)^2 \bigl( \prod x_i \bigr)^r \bigl( \prod y_j \bigr)^s \dA{n} \dB{m}}{\Gamma(x;y)}$ & \\ \qquad where $\sigma_n = \{ x \in \R^n : 0 < x_1 < \dots < x_n \}$ & \eqref{eq:heine-integral} \\
    Degenerate cases $\heineintegral_{0m}^{rs}$, $\heineintegral_{n0}^{rs}$, $\heineintegral_{00}^{rs}$ & \eqref{eq:heine-integral-one-index-zero} \\
    The basic bimoment determinant & \\ \qquad $D_n = \det(I_{ij})_{i,j=0}^{n-1} = \heineintegral_{nn}^{00}$ & \eqref{eq:bimoment-det} \\
    General discrete setup & \\ \qquad $\alpha = \sum_{i=1}^A a_i \delta_{\lambda_i}$, $\beta = \sum_{j=1}^B b_j \delta_{\mu_j}$ & \autoref{sec:heine-integral-discrete-case} \\ \qquad ($A=K$, $B=K-1$ in the main text) & \\
    Heine-type integrals as sums in the discrete case & \\ \qquad $\heineintegral_{nm}^{rs} = \sum_{I \in \binom{[A]}{n}} \sum_{J \in \binom{[B]}{m}} \Psi_{IJ} \, \lambda_I^r a_I \, \mu_J^s b_J$ & \eqref{eq:heine-integral-as-sum} \\
    where & \\
    \qquad $[A] = \{1,2,\dots,A \}$ & \\[0.7ex]
    \qquad $\binom{[A]}{n} = \text{set of $n$-element subsets of $[A]$}$ & \\
    \qquad $\lambda_I^r a_I \, \mu_J^s b_J = \Bigl( \prod_{i \in I} \lambda_i^r a_i \Bigr) \Bigl( \prod_{j \in J} \mu_j^s b_j \Bigr)$ & \eqref{eq:product-notation-explanation} \\
    \qquad $\Psi_{IJ} = \frac{\Delta_I^2 \twin\Delta_J^2}{\Gamma_{IJ}}$ & \eqref{eq:PsiIJ} \\
    \qquad\qquad $\Delta_I^2 = \Delta(\lambda_{i_1},\dots,\lambda_{i_n})^2$, & \\ \qquad\qquad $\twin\Delta_J^2 = \Delta(\mu_{j_1},\dots,\mu_{j_m})^2$, & \\ \qquad\qquad $\Gamma_{IJ} = \Gamma(\lambda_{i_1},\dots,\lambda_{i_n};\mu_{j_1},\dots,\mu_{j_m})$ & \eqref{eq:DeltaI} \\
    $\Delta_{I_1 I_2}^2 = \prod_{i_1 \in I_1, \, i_2 \in I_2} (\lambda_{i_1} - \lambda_{i_2})^2$ & \eqref{eq:DeltaI1I2} \\
    $\heineintegral_{nm}^{00}$ written out in the case $A=3$, $B=2$ & \autoref{ex:nonzero-heineintegrals} \\
    $\starheineintegral_{nm}^{rs}$ & \autoref{lem:heine-integral-symmetry} \\
    Polynomials $A_k(\lambda)$, $B_k(\lambda)$, $C_k(\lambda)$ & \eqref{eq:psi-inbetween}, \eqref{eq:jumpI-ABC} \\
    Jump matrix $S_k(\lambda)$ & \eqref{eq:jumpI-even-odd} \\
    $\bigl( A(\lambda), B(\lambda), C(\lambda) \bigr) = \bigl( A_N(\lambda), B_N(\lambda), C_N(\lambda) \bigr)$ & \eqref{eq:ABC-jump-product} \\
    Coefficients $[A]_i$, $[B]_i$, $[C]_i$ in $A$, $B$, $C$ & \eqref{eq:ABC} \\
    $E_{ab}=e^{-\abs{x_a-x_b}}$ & \eqref{eq:Eab} \\
    Polynomials $\twin A_k(\lambda)$, $\twin B_k(\lambda)$, $\twin C_k(\lambda)$ & \eqref{eq:psi-twin-inbetween}, \eqref{eq:jumpII-ABC} \\
    Jump matrix $\twin S_k(\lambda)$ & \eqref{eq:jumpII-even-odd} \\
    $\bigl( \twin A(\lambda), \twin B(\lambda), \twin C(\lambda) \bigr) = \bigl( \twin A_N(\lambda), \twin B_N(\lambda), \twin C_N(\lambda) \bigr)$ & \eqref{eq:twin-ABC-jump-product} \\
    Coefficients $[\twin A]_i$, $[\twin B]_i$, $[\twin C]_i$ in $\twin A$, $\twin B$, $\twin C$ & \eqref{eq:ABC-twin} \\
    $\psi_{2,\text{even}}$, $\psi_{3,\text{odd}}$, $\mathcal{L}$, $\mathcal{E}$, $\mathcal{M}$, $\mathcal{N}$ & \eqref{eq:eigenvalue-problem} \\
    $\twin\psi_{2,\text{odd}}$, $\twin\psi_{3,\text{even}}$ & \eqref{eq:big-matrix-equation} \\
    $\twin\psi'_{2,\text{even}}$, $\twin\psi'_{3,\text{odd}}$, $\mathcal{L}'$, $\mathcal{E}'$, $\mathcal{M}'$, $\mathcal{N}'$ & \eqref{eq:truncated-big-matrix-equation} \\
    $[K] = \{ 1,2,\dots,K \}$ & \\
    $\binom{S}{k}$, the set of $k$-element subsets of a set $S$ & \\
    Index sets $I = \{ i_1 < \dots < i_m \}$, $J = \{ j_1 < \dots < j_n \}$ & \\
    Submatrix $X_{IJ} = \bigl( X_{i_a j_b} \bigl)_{\substack{a=1,\dots,m \\ b=1,\dots,n}}$ & \autoref{app:coeffs-via-planar-networks} \\
    ``Half-strictly interlacing'' relations: & \\
    \qquad $I \lel J \quad \Longleftrightarrow \quad i_1 \le j_1 < i_2 \le j_2 < \dots < i_k \le j_k$ & \eqref{eq:half-interlacing-I} \\
    \qquad $I \lle J \quad \Longleftrightarrow \quad i_1 < j_1 \le i_2 < j_2 \le \dots \le i_k < j_k$ & \eqref{eq:half-interlacing-II} \\
  \end{longtable}
\end{center}

\section*{Acknowledgements}

Hans Lundmark is supported by the Swedish Research Council (Veten\-skaps\-r{\aa}det),
and Jacek Szmigielski by the National Sciences and Engineering Research Council of Canada (NSERC).
We thank Marcus Kardell for useful comments.

\bibliographystyle{abbrv}
\bibliography{GX-inverse-problem}

\end{document}